\documentclass[10pt,a4paper]{amsart}
\usepackage{fullpage}
\usepackage{mathrsfs}
\usepackage[left=1.0in, right=1.0in, top=1in, bottom=1.3in, includefoot, headheight=13.6pt]{geometry}
\usepackage[colorlinks=true,hyperindex,citecolor=blue,linkcolor=black]{hyperref}
\input{xy}
\usepackage{cite}
\usepackage{tikz-cd}
\usepackage[capitalize]{cleveref}
\usepackage{eucal,times,amsmath,amsthm,amssymb,mathrsfs,stmaryrd,color,enumerate,accents}
\usepackage{tensor}

\usepackage{textcomp}
\usepackage{mathtools,bm,eucal} % math related
\usepackage[T1]{fontenc} % font encoding
\usepackage{enumerate,comment,braket,xspace,csquotes} % utilities

 \numberwithin{equation}{subsection}

\newtheorem{theorem}{Theorem}[subsection]
\newtheorem{lemma}[theorem]{Lemma}
\newtheorem{corollary}[theorem]{Corollary}
\newtheorem{proposition}[theorem]{Proposition}

\newtheorem{thm-intro}{Theorem}

\newtheorem{ap-theorem}{Theorem}
\newtheorem{ap-corollary}{Corollary}
\newtheorem{ap-lemma}{Lemma}
\newtheorem{ap-proposition}{Proposition}
\newtheorem{ap-definition}{Definition}

\theoremstyle{definition}
\newtheorem{remark}[theorem]{Remark}
\newtheorem{example}[theorem]{Example}
\newtheorem{notation}[theorem]{Notation}
\newtheorem{definition}[theorem]{Definition}
\newtheorem{construction}[theorem]{Construction}
\newtheorem{warning}[theorem]{Warning}

\newcommand{\cXysquare}[8]{
\[\cXymatrix{
#1 \ar@{#5}[r] \ar@{#6}[d] & #2 \ar@{#7}[d]\\
#3 \ar@{#8}[r] & #4
}\]
}

\def\lim{\mathrm{lim}}

\DeclareRobustCommand{\SkipTocEntry}[5]{}

\newcommand{\Map}{\mathrm{Map}}
\newcommand{\Hom}{\mathrm{Hom}}
\newcommand{\ad}{\mathrm{ad}}
\newcommand{\Ex}{\mathrm{Ex}}
\newcommand{\fCoh}{\mathfrak{Coh}^+}

\newcommand{\sfX}{\mathsf{X}}
\newcommand{\sfY}{\mathsf{Y}}
\newcommand{\sfZ}{\mathsf{Z}}
\newcommand{\sfW}{\mathsf{W}}

\newcommand{\Ok}{k^\circ}

\DeclareMathOperator{\Spec}{Spec}
\DeclareMathOperator{\Spf}{Spf}
\DeclareMathOperator{\Sp}{Sp}
\newcommand{\st}{\mathrm{st}}
\newcommand{\Der}{\mathrm{Der}}
\newcommand{\trun}{\mathrm{t}}
\newcommand{\spe}{\mathrm{sp}}

\newcommand{\bfA}{\mathfrak{A}_{\Ok}}
\newcommand{\anA}{\mathbf{A}_k}
\newcommand{\anB}{\mathbf{B}_k}

\newcommand{\Tdisc}{\mathcal{T}_{\mathrm{disc}}}
\newcommand{\Tet}{\mathcal{T}_{\text{\'et}}}
\newcommand{\Tetph}{\mathcal{T}_{\emph{\text{\'et}}}}

\newcommand{\Tan}{\mathcal{T}_{\an}(k)}
\newcommand{\Tad}{\mathcal{T}_{\ad}(\Ok)}

\newcommand{\CAlg}{\mathcal{C}\mathrm{Alg}}
\newcommand{\adCAlg}{\cC \mathrm{Alg}^{\ad}_{\Ok}}
\newcommand{\admCAlg}{\cC \mathrm{Alg}^{\mathrm{adm}}_{\Ok}}
\newcommand{\fCAlg}{\mathrm{f} \cC \mathrm{Alg}_{\Ok}}

\newcommand{\taft}{\mathrm{taft}}
\newcommand{\AnRing}{\mathrm{AnRing}_k}

\newcommand{\locStr}{\mathrm{Str}^{\mathrm{loc}}}
\newcommand{\adm}{\mathrm{adm}}

\newcommand{\Cat}{\mathcal{C}\mathrm{at}_\infty}
\newcommand{\Coh}{\mathrm{Coh}^+}

\newcommand{\bCoh}{\mathrm{Coh}^\mathrm{b}}
\newcommand{\cHom}{\mathcal{H} \mathrm{om}}

\newcommand{\op}{\mathrm{op}}
\newcommand{\der}{\mathrm{der}}

\newcommand{\Shv}{\mathrm{Shv}}
\newcommand{\ind}{\mathrm{Ind}}

\newcommand{\infcat}{$\infty$-category\xspace}
\newcommand{\infcats}{$\infty$-categories\xspace}

\newcommand{\et}{\text{\'et}}

\newcommand{\tcomp}{(-)^\wedge_t}
\newcommand{\functor}{(-)}
\newcommand{\rigg}{(-)^{\mathrm{rig}}}
\newcommand{\rig}{\mathrm{rig}}
\newcommand{\alg}{\mathrm{alg}}
\newcommand{\Fun}{\mathrm{Fun}}

\newcommand{\sh}{\mathrm{sh}}
\newcommand{\disc}{\functor^\mathrm{disc}}

\newcommand{\Top}{\tensor[^{\mathrm{R}}]{\cT \op}{}}

\newcommand{\adTop}{\tensor[^{\mathrm{R}}]{\cT \op}{}( \Tad)}
\newcommand{\anTop}{\tensor[^{\mathrm{R}}]{\cT \op}{}( \Tan)}
\newcommand{\etTop}{\tensor[^{\mathrm{R}}]{\cT \op}{} (\Tet(\Ok))}
\newcommand{\etphTop}{\tensor[^{\mathrm{R}}]{\cT \op}{} (\Tetph(\Ok))}

\newcommand{\discTop}{\tensor[^{\mathrm{R}}]{\cT \op}{}(\Tdisc(\Ok))}

\newcommand{\discTopn}{\tensor[^{\mathrm{R}}]{\cT \op}{} (\Tdisc(\Ok_n))}

\newcommand{\Mod}{\mathrm{Mod}}

\newcommand{\An}{\mathrm{An}_k}
\newcommand{\dAfd}{\mathrm{dAfd}_k}
\newcommand{\dAn}{\mathrm{dAn}_k}

\newcommand{\dfSch}{\mathrm{dfSch}_{\Ok}}
\newcommand{\fSch}{\mathrm{fSch}_{\Ok}}

\newcommand{\dfDM}{\mathrm{dfDM}}

\newcommand{\an}{\mathrm{an}}
\newcommand{\Ab}{\mathrm{Ab}}

\DeclareMathOperator*{\colim}{colim}

\newcommand{\bA}{\mathbb A}

\newcommand{\bE}{\mathbb E}

\newcommand{\bL}{\mathbb L}

\newcommand{\bZ}{\mathbb Z}

\newcommand{\ff}{\mathfrak f}
\newcommand{\fX}{\mathfrak X}

\newcommand{\cA}{\mathcal A}
\newcommand{\cB}{\mathcal B}
\newcommand{\cC}{\mathcal C}
\newcommand{\cD}{\mathcal D}
\newcommand{\cE}{\mathcal E}
\newcommand{\cF}{\mathcal{F}}
\newcommand{\cG}{\mathcal{G}}
\newcommand{\cH}{\mathcal{H}}
\newcommand{\cK}{\mathcal{K}}

\newcommand{\cO}{\mathcal{O}}

\newcommand{\cT}{\mathcal{T}}
\newcommand{\cX}{\mathcal X}
\newcommand{\cY}{\mathcal Y}
\newcommand{\cZ}{\mathcal Z}
\newcommand{\cS}{\mathcal S}

\author{Jorge Ant\'onio}
\date{}
\newcommand{\adress}{{% additional braces for segregating \footnotesize
  \bigskip
  \footnotesize

\textsc{Jorge Ant\'onio,  IRMA, UMR 7501
 7 rue René-Descartes
 67084 Strasbourg Cedex}\par\nopagebreak
  \textit{E-mail address},  \texttt{jorgeantonio@unistra.fr}

}}

\title{Derived $\Ok$-adic geometry and derived Raynaud localization Theorem} 
\begin{document}

\maketitle

\begin{abstract}
The goal of the present text is to state and prove a generalization of Raynaud localization theorem in the setting of derived geometry. More explicitly, we show that the \infcat of quasi-paracompact and quasi-separated
derived $k$-analytic spaces can be realized as a localization of the \infcat of
admissible derived formal schemes. We construct a derived rigidification functor generalizing Raynaud's rigidification functor. In order to construct the latter we will need to formalize derived formal $\Ok$-adic formal geometry via a structured spaces approach. We prove that $\Ok$-adic Postnikov towers of derived $\Ok$-adic Deligne-Mumford stacks decompose
and we relate these
to
Postnikov towers of derived $k$-analytic spaces.  This is possible by a precise comparison between the $\Ok$-adic cotangent complex and the $k$-analytic cotangent complex.
\end{abstract}

\tableofcontents

\section{Introduction}

\subsection{Background material}
Let $k$ be a non-archimedean field of rank 1 valuation, $\Ok$ its ring of integers and let $t \in \Ok$ be a fixed pseudo-uniformizer of $k$.
Denote $\mathrm{fSch}_{\Ok}$ the category of quasi-paracompact and quasi-separated \emph{admissible} $\Ok$-adic formal schemes, i.e. formal schemes over $\Spf k^\circ$ topologically of finite presentation whose structure sheaf is $(t)$-torsion free. Let $\An$ denote the category of $k$-analytic spaces.
There exists a rigidification functor, also referred as \emph{Raynaud's generic fiber functor},
	\[
		\rigg \colon \mathrm{fSch}_{\Ok} \to \An,
	\]
given by the formula
	\[
		\sfX \in \mathrm{fSch}_{\Ok} \mapsto \sfX^\rig \in \An.
	\]
When $\sfX = \Spf (A)$, for some topologically of finite presentation $\Ok$-adic algebra $A$, we have that
	\[
		\Spf(A)^\rig \simeq \Sp(A \otimes_{\Ok} k) \in \An,
	\]
where the latter denotes the $k$-affinoid space associated to the $k$-affinoid algebra $A \otimes_{\Ok} k$.
Moreover, every quasi-paracompact and quasi-separated $k$-analytic space $X$ admits a formal model over $\Spf \Ok$. Concretely, there exists $\sfX \in 
\mathrm{fSch}_{\Ok}$ such that
	\[
		 \sfX^{\rig} \simeq X.
	\]
In principle, one is then able to understand the analytic structure on $X$ through the formal model $\sfX$. Furthermore one has the following crucial result, proved by Michel Raynaud in the quasi-compact case and by
Rachid Lamjoun, cf. \cite[Theorem 3.14]{lamjoun1999varieties},  in the quasi-paracompact case:

\begin{thm-intro}[Raynaud/Lamjoun, Theorem 8.4.3 in \cite{bosch2005lectures}] \label{Raynaud}
The functor $\rigg \colon \mathrm{fSch}_{\Ok} \to \An$ is a localization functor. More precisely, the latter factors through the localization of $ \mathrm{fSch}_{\Ok} $ at the class of admissible
blow ups, denoted $S$. Moreover, the induced functor
	\[
		\fSch[S^{-1}] \to \An^{\mathrm{qpcqs}}
	\]
is an equivalence of categories, where $\An^{\mathrm{qpcqs}} \subseteq \An$ denotes the full subcategory of quasi-paracompact quasi-separated $k$-analytic spaces.
\end{thm-intro}

\cref{Raynaud} has many applications in practice. Indeed, such result allows us to extrapolate methods of modern algebraic geometry into the rigid analytic setting.
For instance, \cref{Raynaud} is useful to study flatness and base change theorems in the setting of rigid $k$-analytic geometry. Another instance of this principle is the study of flat descent of coherent sheaves for rigid $k$-analytic geometry, proved in
great generality by Bosch-G\"ortz \cite{1998coherent}. \cref{Raynaud} has also found applications to the study of rigid $k$-analytic moduli spaces,
such as Drinfeld's half upper plane, the rigidification of moduli of abelian varieties equipped with level structures and more recently the moduli of $p$-adic local systems of
\'etale local systems on smooth varieties, $\mathrm{LocSys}_{p, n}(X)$, see \cite{antonio2017moduli}.

Raynaud's localization theorem allows us to bypass the intrinsic analytic difficulties by reducing the questions at hand to the formal level. Furthermore, techniques coming from algebraic geometry can be used effectively to study the geometry of formal schemes.

\subsection{Main results} The main goal of the current text is to prove an analogue of \cref{Raynaud} in the setting of derived $k$-analytic geometry.

Derived $k$-analytic geometry was introduced and studied at lenght by M. Porta and T. Yu Yue in \cite{porta2016derived, porta2017representability}. On the other hand, spectral formal algebraic geometry was developped in \cite[\S 8]{lurie2016spectral}
by J. Lurie. Unfortunately, no link between the theory of derived $\Ok$-adic schemes and that of derived $k$-analytic spaces has been established prior to this text.

In order to state a derived analogue of \cref{Raynaud} 
one needs a crucial ingredient, namely the existence of a \emph{derived rigidification functor}. Inspired by the construction of the derived analytification functor of \cite[\S 3]{porta2017representability}, we will give a construction of the derived rigidification
functor, see \cref{cor:derived_rig_functor}. Our construction requires the development of a structured spaces approach to derived formal $\Ok$-adic geometry.

This is achieved in \S 3: we develop a theory of derived $(t)$-adic formal geometry by means of $\Tad$-structured spaces. Here $\Tad$ denotes the $\Ok$-adic pregeometry introduced in \cref{defin:T-adic_structures}. More concretely, we will consider pairs $(\cX, \cO)$ where
$\cX$ is an $\infty$-topos and
	\[
		\cO \colon \Tad \to \cX
	\]
is a \emph{local $\Tad$-structure}, see \cref{def:pre}. We prove that the datum of such an object $\cO$ is equivalent to the datum of a sheaf of derived $\Ok$-adic algebras, $\cO^\ad$, such that $\pi_0(\cO^\ad)$ is equipped with a natural adic topology. The latter is further assumed to be compatible with the $(t)$-adic topology
on $\Ok$. This last statement can be interpreted as a rectification type statement and it is the content of \cref{rect}.

We will then give a definition of derived formal $\Ok$-adic Deligne-Mumford stacks over $\Ok$, via $\Tad$-structured spaces. Moreover, under certain mild finiteness conditions, this notion agrees with the simplicial analogue of the spectral notion introduced in \cite[\S 8]{lurie2016spectral}.
We then proceed to study \emph{$\Ok$-adic Postnikov tower
decompositions} and the \emph{$\Ok$-adic cotangent complex}. To the author's best knowledge, the study of the $\Ok$-adic cotangent complex and its role in the study of Postnikov towers of $\Ok$-adic Deligne-Mumford stacks
has never been addressed before in the literature. 

Using the machinery developed in \S 3,
we define a 
rigidification functor 
	\[
		\rigg \colon ^\mathrm{R} \mathcal{T} \mathrm{op} \left( \Tad \right) \to \ ^\mathrm{R} \mathcal{T} \mathrm{op} \left( \Tan \right)
	\]
which restricts to a well-defined functor 
	\[
		\rigg \colon  \mathrm{dfDM}_{\Ok} \to \dAn,
	\]
where $\mathrm{dfDM}_{\Ok} $ denotes the \infcat of derived $\Ok$-adic Deligne-Mumford stacks and $\dAn$ the \infcat of derived $k$-anlytic spaces. See \cref{defin:derived O_k adic DM stack} for the definition of the former and \cref{defin derived analytic spaces} for the definition of $\dAn$.
We prove that the derived rigidification functor $\rigg$ coincides with the usual rigidification functor when restricted to the category of ordinary formal schemes, cf. \cref{rig_cl}.

\begin{theorem}[\cref{main1}] \label{int:main1}
Let $Z \in \dAn$ be a quasi-paracompact and quasi-separated derived $k$-analytic space. There exists $\sfZ \in \dfDM$ such that one has an equivalence $
\sfZ^{\rig} \simeq Z$ in $\dAn$. In other words, $Z$ admits a formal model $\sfZ \in \mathrm{dfDM}_{\Ok}$.
\end{theorem}

Let $\mathrm{dfSch}_{\Ok}$ denote the full subcategory of $\mathrm{dfDM}_{\Ok}$ spanned by those $\sfX \in \mathrm{dfDM}_{\Ok}$ such that $\mathrm t_{0}\left( \sfX \right)$ is equivalent to an ordinary admissible quasi-paracompact and quasi-separated formal
scheme over $\Ok$. We say that a morphism $f \colon \sfX \to \sfY$ in $\dfDM_{\Ok}$ is \emph{rig-strong} if, for each $i>0$, the induced map
	\[
		\pi_i \left((f^\rig)^{-1}\cO_{\sfY^\rig} \right) \otimes_{\pi_0((f^\rig)^{-1}\cO_{\sfY^\rig})} \pi_0( \cO_{\sfX^\rig}) \to \pi_i \left( \cO_{\sfX^\rig} \right),
	\]
is an equivalence in the \infcat of the structure sheaf $\cO_{\sfX^\rig}$-modules, $\Mod_{\pi_0(\cO_{\sfX^\rig})}$.
Let $\dAn^{\mathrm{qpcqs}} \subseteq \dAn$ denote the full subcategory spanned by those $X \in \dAn$ such that its $0$-th truncation $\trun_{\leq 0}(X)$ is equivalent to a quasi-paracompact and quasi-separated ordinary $k$-analytic space.
The following is a direct generalization of Raynaud's localization theorem to the derived setting:

\begin{thm-intro}[\cref{main}] \label{loc}
Let $S$ denote the saturated class generated by rig-strong morphisms $ f\colon \sfX\to \sfY$ such that $\mathrm t_{0} \left(f \right)$ is an admissible blow up, in $\mathrm{dfSch}_{\Ok} $. Then
the rigidification functor
	\[
		(-)^{\mathrm{rig}}: \mathrm{dfSch}_{\Ok} \to \mathrm{dAn}^{\mathrm{qpcqs}}_k.
	\]
factors through the localization \infcat $\mathrm{dfSch}_{\Ok}[S^{-1}]$. Moreover, the induced functor
	\[
		\mathrm{dfSch}^{}_{\Ok}[S^{-1}] \to
		\mathrm{dAn}^{\mathrm{qpcqs}}_k.
	\]
is an equivalence of \infcats.
\end{thm-intro}

Let us briefly sketch the proof of \cref{main}. In order to prove the statement it suffices to prove that given $X \in \dAn$, as above, the comma category $\cC_X \coloneqq \big( \dfSch \big)_{X / }$ is contractible.
We will prove a slightly stronger
result, namely $\cC_X$ is a filtered \infcat. In order to illustrate the main ideas behind the proof it suffices to explain how to lift a morphism $f \colon X \to Y$ in $\dAn$ to a morphism $f^+ \colon \sfX \to \sfY$ in $\dfSch$, such that $(f^+)^\rig \simeq f$.

Our argument follows by induction on the Postnikov tower
of $X$. Suppose that $X \simeq \trun_{\leq 0}(X)$ in the \infcat $\dAn$. Notice that \cref{Raynaud} implies that we can lift $\trun_{\leq 0}(f)$ to a morphism $f^+_0 \colon \sfX_0 \to \sfY_0$ in the category $\fSch$.
As $X \to Y$ factors through the canonical morphism $\trun_{\leq 0} Y \to Y$ in $\dAn$,
we conclude by \cref{int:main1} together with \cref{Raynaud} that we can find a formal model for $f \colon X \to Y$, up to an admissible blow up of the $0$-th truncations.

Let $n \geq 0 $ be an integer. Assume further that we are giving a morphism 
	\[
		(f^+_n) \colon \sfX_n \to \sfY_n,
	\]
in $\dfSch$ such that 
	\[
		(f^+_n)^\rig \simeq \trun_{\leq n} f \colon \trun_{\leq n} X \to \trun_{\leq n} Y.
	\]
Consider the $(n+1)$-th step of the
Postnikov tower, namely the pushout diagram
	\[
	\begin{tikzcd}
		\trun_{\leq n} X[ \pi_{n+1} (\cO_{X} ) [n+2] ] \ar{r} \ar{d} & \trun_{\leq n} X \ar{d} \\
		\trun_{\leq n } X \ar{r} & \trun_{\leq n+1} X,
	\end{tikzcd}
	\]
in the \infcat $\dAn$. In order to proceed, we will need to prove:

\begin{proposition}[\cref{cor:adic_cotangent_complex_is_compatible_with_analytic_via_rig}]
Let $\sfX \in \dfSch$ and denote $X \coloneqq \sfX^\rig \in \dAn$. Then the rigidification functor induces a canonical equivalence
	\[
		\big( \bL^\ad_{\sfX} \big)^\rig \simeq \bL^\an_X,
	\]
in the \infcat $\Mod_{\cO_X}$.
\end{proposition}

The induction hypothesis combined with the universal property of the adic and analytic cotangent complexes and with refined results on the existence of formal models for almost perfect modules on $X$, proved in both \cite{antonio_Hilbert} and
Appendix A, imply that
we can extend the morphism $f^+_n \colon \sfX_n \to \sfY_n$ to a diagram
	\begin{equation} \label{diag_of_diag}
		 f^+_n \leftarrow f^+_n [ \pi_{n+1}(f)^+[n+2] ] \rightarrow \trun_{\leq n } f^+_n.
	\end{equation}
The latter is considered as an object in $\Fun \big( \Lambda^2_0, \dfSch^{\Delta^1} \big)$, where 
	\[
		\pi_{n+1}(f)^+ \in \Coh(\sfX_0)^{\Delta^1},
	\]
denotes a formal model for $\pi_{n+1}(f)$. By taking pushouts along $\Lambda^2_0$
we obtain the desired lifting
	\[
		f^+_{n+1} \colon \sfX_{n+1} \to \sfY_{n+1},
	\]
of $\trun_{\leq n+1}(f)$.

The main technical difficulty of the proof comes from the lifting of higher coherences involved in finite diagrams of derived $k$-analytic spaces to higher coherences of diagrams of formal models. This is needed in order to extend
\eqref{diag_of_diag} above in the case of more complex diagrams.

\subsection{Notations and conventions} Throughout the text, unless otherwise stated, $k$ denotes a non-archimedemean field of rank $1$ valuation and $\Ok = \{ x \in k : |x| \leq 1 \}$ its ring of integers. We denote $t \in \Ok$ a pseudo-uniformizer of $k$.
Given an integer $n \geq 1$, we will denote by $\Ok_n $ the reduction modulo $(t^n)$ of $\Ok$. 
We denote
$\mathrm{fSch}_{\Ok}$ the (classical) category of formal schemes (topologically) of finite presentation over $\Ok$.

Let $n \geq 0 $ be an integer, we define $\Ok \langle T_1, \dots , T_n \rangle $ as the sub-algebra of $\Ok [[T_1 , \dots, T_n ]]$ consisting of those
formal power series $f = \Sigma_I a_i T_I^{b_I}$, such that the coefficients $a_I \to 0 $, with $I \to \infty$, in $\Ok$. Denote by $\bfA^n \coloneqq \Spf \Ok \langle T_1, \dots, T_n \rangle$ the $\Ok$-adic affine $n$-space. 

We say that a morphism between two $(t)$-adic complete $\Ok$-algebras $A \to B$ is \emph{formally \'etale} if, for each $n \geq 0$, its mod $t^n$ reduction is
an \'etale homomorphism of $\Ok/ t^n$-algebras.

We shall further denote by $\anB^n \coloneqq \Sp \langle T_1, \dots, T_n \rangle$ the closed unit disk
and $\anA^n \coloneqq \big( \bA^n_k \big)^\an$ the $k$-analytic affine $n$-space.

Let $R$ be a commutative simplicial ring, we denote $\CAlg_R$ its \infcat of simplicial $R$-algebras.
We will refer to an object $A \in \CAlg_R$ as a \emph{derived $R$-algebra}. Similarly, if $R$ is a discrete ring, we denote $\CAlg^\heartsuit_R$ the category of discrete $R$-algebras. 

Given an object $B \in \CAlg_R$ we denote by $\pi_i \left(B \right)$ the $i$-th homotopy group of the underlying space associated to $B$. We will denote $\Mod_R$ the derived \infcat of $R$-modules.
Throughout the text we will employ homological convention. Thus given $M \in \Mod_R$ we denote by $\pi_i(M) \coloneqq H_i(M)$ its $i$-th homology group.

We denote by $\cS$ the \infcat of spaces and $^\mathrm{R} \mathcal{T} \mathrm{op}$ the \infcat of $\infty$-topoi together with geometric morphisms between these. 
In this paper we will extensively use the machinery of structure spaces, developped in \cite{lurie2011dag}. We will introduce the $\Ok$-adic pregeometry $\Tad$, spanned by (topologically) of finite presentation formally smooth $\Ok$-adic spaces.
Given an $\infty$-topos $\cX$ we will denote $\CAlg_R(\cX) \coloneqq \locStr_{\Tdisc(R)}(\cX)$, $\CAlg_R^{\sh}(\cX) \coloneqq \locStr_{\Tet(R)}(\cX)$, $\fCAlg(\cX) \coloneqq \locStr_{\Tad}(\cX)$ and $\AnRing(\cX) \coloneqq \locStr_{\Tan}(\cX)$.
We will often denote a general pregeometry by the letter $\cT$. If $R$ is discrete, we shall also denote by $\CAlg_R^\heartsuit(\cX)$ the category of discrete $R$-algebras on $\cX$.
Let $\cT$ be a pregeometry. In this paper, we always work with local structures, that is $\cO \in \locStr_\cT(\cX)$.

Throughout the present text we will freely cite \cite{lurie2016spectral}. However, we warn the reader that \cite{lurie2016spectral} deals with spectral algebraic geometry.
On the other hand, this paper is devoted to derived geometry and we never make use of $\bE_\infty$-ring spectra. Fortunately, the statements which we will need from \cite{lurie2016spectral} hold true in the simplicial setting. Moreover, the corresponding proofs apply mutadis mutandins
in the simplicial setting.
It is also possible to define spectral $\Ok$-adic geometrical analogues of the results proved in the current text. However, we do not explore this direction here.

\subsection{Acknowledgments} The author is deeply thankful to Mauro Porta for sharing many ideas, remarks and advices about the contents of the paper. The author is also thankful to Bertrand To\"{e}n and Marco Robalo for valuable discussions about the contents of the paper.

\section{Review on derived algebraic and analytic geometry}

\subsection{Derived algebraic geometry}
In \cite{lurie2011dag}, J. Lurie introduced the notion of a (spectral) scheme, and more generally (spectral) Deligne-Mumford stack via a structured spaces approach. We review some of the basic notions for the reader's convenience.
The reader is referred to \cite{lurie2011dag} and \cite{porta2016derived} for more details. 

\begin{definition}
A \emph{ringed $\infty$-topos} is a pair $(\cX, \cO)$ where $\cX$ denotes an $\infty$-topos and $\cO \in \CAlg(\cX)$ is a $\CAlg$-valued sheaf, on $\cX$. We say that a ringed $\infty$-topos is a \emph{locally ringed $\infty$-topos} if for each
geometric point
	\[
		x_* \colon \cX \rightleftarrows \cS \colon x^{-1},
	\]
the derived ring $x^{-1} \cO $, on $\cS$, can be identified with a derived local $k$-ring.
\end{definition}

\begin{remark}
Let $X$ a topological space. We form the associated $\infty$-topos $\cX := \Shv(X)$, of $\cS$-valued sheaves on $X$. To a classical locally ringed pair $(X, \cO)$  one (functorially) associates a locally ringed $\infty$-topos $(\cX, \cO)$. Indeed, $\cO$ can be naturally
promoted to a local $\CAlg$-valued sheaf on $\cX$.
\end{remark}

\begin{definition}
	Let $\cX$ be $\infty$-topos. Given $\cA \in \CAlg(\cX)$, we say that \emph{$\cA$ is almost of finite presentation} if $\pi_0(\cA)$ is a classical ring object on $\cX$, which is further assumed to be of finite presentation. We further require that, for every $i > 0$, $\pi_i(\cA)$ is a coherent module over $\pi_0(\cA)$.
\end{definition}

We now reformulate the notion of locally ringed $\infty$-topos in terms of pregeometries:

\begin{definition}
A \emph{pregeometry} consists of an $\infty$-category $\cT$ equipped with a class of \emph{admissible morphisms} and a Grothendieck topology. The latter is generated by admissible morphisms. Moreover we require the following conditions to hold:
\begin{enumerate}
\item $\cT$ admits finite products;
\item Pullbacks along admissible morphisms exist and are again admissible;
\item If $f$ and $g$ are morphisms in $\cT$ such that $g$ and $g \circ f$ are admissible, then so is $f$.
\item Retracts of admissible morphisms are admissible. 
\end{enumerate}
\end{definition}

We give a list of well known examples of pregeometries which will be useful later on.

\begin{example}
\begin{enumerate}
\item Let $\mathcal{T}_{\mathrm{disc}}(k)$ denote the pregeometry whose underlying category is the full subcategory of the category of affine $k$-schemes spanned by affine spaces $\{ \mathbb{A}^n_k \}_{n \ge 0}$. The family of admissible morphisms is the family of isomorphisms in $\Tdisc(k)$.
We further equip it with the discrete Grothendieck 
topology.

\item Let $\Tet(k)$ denote the pregeometry whose underlying category is the full subcategory of the category of affine schemes spanned by smooth $k$-schemes. A morphism in
$\Tet(k)$ is admissible if and only if it is an \'etale morphism of affine schemes. We equip $\Tet(k)$ with the \'etale topology.
\end{enumerate}
\end{example}

\begin{definition} \label{def:pre}
Let $\cT$ be a pregeometry and $\mathcal{X}$ an $\infty$-topos. A \emph{$\cT$-local structure on $\mathcal{X}$} is a functor between $\infty$-categories $\mathcal{O}: \cT \to \mathcal{X}$ satisfying the following conditions:
\begin{enumerate}
\item The functor $\mathcal{O}$ preserves finite products in $\cT$;
\item For a pullback square, in $\cT$
	\[
	\begin{tikzcd}
		U' \ar{r} \ar{d} & X' \ar{d}{f} \\
		U \ar{r} & X
	\end{tikzcd},
	\]
such that $f$ is admissible, the square
	\[
	\begin{tikzcd}
		\mathcal{O}(U') \ar[r] \ar[d] &  \mathcal{O}(X') \ar{d}{\mathcal{O}(f)} \\
		\mathcal{O}(U) \ar[r] & \mathcal{O}(X)
	\end{tikzcd}
	\]
is also a pullback square in $\mathcal{X}$.
\item Let $\{ f_\alpha \colon U_{\alpha} \to U\}$ denote a $\tau$-covering in $\cT$, such that the $f_\alpha$'s are admissible. Then the induced map
	\[ 
		\coprod \mathcal{O}( U_{\alpha}) \to \mathcal{O}(U),
	\]
is an effective epimorphism in $\mathcal{X}$.
\end{enumerate}
A morphism $\mathcal{O} \to \mathcal{O}'$ between $\cT$-local structures is said to be \emph{local} if it is a natural transformation satisfying the following additional condition: for every admissible morphism $U \to X$ in $\cT$, the resulting diagram
	\[
	\begin{tikzcd}
		\mathcal{O}(U) \ar[r] \ar[d] & \mathcal{O}'(U) \ar[d] \\
		\mathcal{O}(X) \ar[r] & \mathcal{O}'(X),
	\end{tikzcd}
	\] 
is a pullback square in $\mathcal{X}$. We denote $\locStr_{\cT} \left( \cX \right)$ the \infcat of local $\cT$-structures on $\cX$ together with local morphisms between these.
\end{definition}

\begin{construction}
\begin{enumerate}
	\item In virtue of \cite[Example 3.1.6, Remark 4.1.2]{lurie2011dag},
	we have an equivalence of \infcats 
		\[
			\locStr_{\Tdisc}( \cX) \simeq \Shv_{\CAlg}(\cX)
		,
		\]
	where the latter denotes the \infcat of $\CAlg$-valued sheaves on $\cX$.
	More explicitly, given a
	ringed $\infty$-topos $ ( \cX, \cO)$, we can promote it naturally to a $\Tdisc$-structured via the construction:
		\[
			\mathbb{A}^n_k \in \Tdisc \mapsto \left( \cO^{\times n} \in \Shv \left( \cX \right) \simeq \cX \right).
		\]
	\item Similarly, a $\Tet(k)$-local structure on $\cX$ corresponds to a $\CAlg_k$-valued sheaf on $\cX$ whose stalks are strictly Henselian. We refer the reader to \cite[Lemma 1.4.3.9]{lurie2016spectral} for a detailed proof of 
	this result. 
\end{enumerate}
\end{construction}

\begin{definition} \label{def_local_str_spaces}
A \emph{$\cT$-structured $\infty$-topos} is a pair $X :=( \mathcal{X}, \mathcal{O})$, where $\mathcal{X}$ denotes an $\infty$-topos and $\mathcal{O}$ is a $\cT$-local structure on $\mathcal{X}$. We denote by $\Top(\cT)$ the $\infty$-category of $\cT$-structured $\infty$-topoi, cf.
\cite[Definition 3.1.9]{lurie2011dag}. 
\end{definition}

\begin{definition}
A derived Deligne-Mumford stack is a $\Tet(k)$-structured $\infty$-topos $(\cX, \cO)$, verifying the following conditions:
	\begin{enumerate}
		\item The $0$-truncation $\mathrm t_{\leq 0 }  \left( \cX, \cO \right)  := \left( \cX, \pi_0 ( \cO^{\alg} )\right)$ is equivalent to an (ordinary) Deligne-Mumford stack;
		\item For each $i> 0 $, the higher homotopy sheaf $\pi_i \left( \cO^{\alg} \right) $ is a quasi-coherent sheaf on $\left( \cX, \cO \right)$.
	\end{enumerate}
\end{definition}

\subsection{Derived $k$-analytic geometry} Let $k$ denote a non-archimedean field of non-trivial valuation. Derived $k$-analytic geometry, developped in \cite{porta2016derived}, is a vast generalization of the classical theory of rigid analytic geometry. 
In this \S, we will review the basic definitions and we refer the reader to \cite{porta2016derived, porta2017representability} for a detailed
account of the foundational aspects of the theory. 

\begin{definition}
Let $\Tan$ denote the pregeometry whose underlying category consists of quasi-smooth $k$-analytic spaces and whose admissible morphisms correspond to \'etale maps between them. We equip $\Tan$ with the \'etale topology.
\end{definition}

\begin{construction} \label{const:an}
Let $X$ be an ordinary $k$-analytic space and denote $X_\et$ the small \'etale site associated to $X$. Let $\cX := \Shv_\et(X_\et)^{\wedge}$ denote the \emph{hypercompletion} of the $\infty$-topos of \'etale sheaves on $X$. We can define a natural $\Tan$-structure, on $\cX$, as follows: given $U \in \Tan$, we define the sheaf $\cO(U) \in \cX$ by the formula
	\[ 
		X_\et \ni V \mapsto \Hom_{\An} \left( V, U \right) \in \cS . 
	\]
As in the algebraic case, we can canonically identify $\cO(\mathbf{A}^1_k)$ with the usual sheaf of analytic functions on $X$.
\end{construction}

\begin{definition} \label{defin derived analytic spaces}
We say that $\Tan$-structured $\infty$-topos $(\mathcal{X}, \mathcal{O})$ is a \emph{derived $k$-analytic space} if the following conditions are satisfied:
\begin{enumerate}
\item $\mathcal{X}$ is hypercomplete and there exists an effective epimorphism $\coprod_i U_i \to 1_{\mathcal{X}}$ on $\cX$ verifying:
\item For each $i$, the pair $(\mathcal{X}_{ | U_i}, \pi_0( \mathcal{O}^{\mathrm{alg}}|  U_i )  )$ is equivalent, in $\anTop$, to an ordinary $k$-analytic space, via \cref{const:an}.
\item For each index $i$ and $j \geq 1$, $ \pi_j(\mathcal{O}^{\mathrm{alg}}|U_i)$ is a coherent sheaf over $\pi_0( \mathcal{O}^{\mathrm{alg}}|U_i)$.
\end{enumerate}
We denote by $\dAn$ the full subcategory of $\anTop$ spanned by derived $k$-analytic spaces.
\end{definition}

\begin{remark}
In \cite{porta2016derived} and \cite{porta2017representability} the authors prove three key statements concerning derived analytic geometry. Namely the unramifiedness of the pregeometry $\Tan$, cf. \cite[Corollary 3.11]{porta2016derived}, a ''gluing along closed immersions'' statement, cf. \cite[Theorem 6.5]{porta2017representability}, and the
existence of an analytic cotangent complex classifying analytic square-zero extensions, loc. cit. \cite[Proposition 5.18]{porta2017representability}. There is also a $k$-analytic analogue of the Artin-Lurie representability theorem, cf. \cite[Theorem 7.1]{porta2017representability}.
\end{remark}

\section{Derived $\Ok$-adic geometry}
In this section we introduce the \emph{$\Ok$-adic pregeometry}, denoted $\Tad$, and study the corresponding theory of $\Tad$-structured spaces. Our first goal is to give an alternative description of a $\Tad$-structured $\infty$-topos $(\cX, \cO)$. Indeed, we prove
that such $(\cX, \cO)$ can be alternatively described as a locally ringed $\infty$-topos $(\cX, \cO^\alg)$ whose $\pi_0 \left( \cO^{\alg} \right)$ is equipped with an adic topology compatible with the $(t)$-adic topology. We will prove such assertion in \S 3.1, under mild finiteness assumptions on $\cO$.

We will also extend the $\Spf$-construction introduced in \cite[\S 8.2]{lurie2016spectral} to the context of $\Tad$-structured spaces. We will then proceed to develop a theory of modules in the setting of derived $\Ok$-adic geometry, see \S 3.3.
Part \S 3.4 is devoted to the study of the \emph{$\Ok$-adic cotangent complex}.
We will establish unramifiedness of the pregeometry $\Tad$ in \S 3.5.
In \S 3.6, we show that Postnikov towers for $\Tad$-structured spaces do converge and they are fully controlled by the adic cotangent complex.

\subsection{Derived $\Ok$-adic spaces}

\begin{definition} \label{defin:T-adic_structures}
The \emph{$\Ok$-adic pregeometry}, denoted $\Tad$, is the full subcategory of the category of affine formal $\Ok$-adic schemes spanned by those formally smooth formal $\Ok$-schemes.
We define the class of admissible morphisms on $\Tad$ as
the one generated by \'etale morphisms. We further equip $\Tad$ with the \'etale topology.
\end{definition}

\begin{notation}
Denote by $\adTop$ the \infcat of $\Tad$-structured $\infty$-topoi. Given $\cX \in \Top$, we set
	\[
		\fCAlg (\cX) \coloneqq \locStr_{\Tad}(\cX),
	\]
the \infcat of \emph{local $\Tad$-structures on} $\cX$.
\end{notation}

We start by showing that ordinary formal $\Ok$-adic Deligne-Mumford stacks admit a natural description as $\Tad$-structured $\infty$-topoi:

\begin{notation}
Consider the transformation of pregeometries
	\[
		\tcomp \colon	\Tdisc(\Ok) \to \Tad,  \  \
		\tcomp \colon \Tet(\Ok) \to \Tad,
	\]
obtained by performing completion along the $(t)$-locus. Precomposition along these transformations induce functors
	\begin{align*}
		\functor ^\alg \colon \fCAlg(\cX) \to \CAlg_{\Ok}(\cX) ,\\
		 \functor ^\sh \colon \fCAlg(\cX) \to \CAlg^{\sh}_{\Ok}(\cX),
	\end{align*}
which we will refer as the \emph{underlying algebra functor} and the \emph{underlying $\Tetph(\Ok)$-structure functor}, respectively. The functor $\functor^\alg$ sends every $\Tad$-structure, on $\cX$, to its underlying algebra object. The latter is obtained by evaluation
on the formal affine line, $\mathfrak{A}^1_{\Ok}$.
Furthermore, the above construction induces canonical functors of $\infty$-categories of structured $\infty$-topoi:
	\begin{align*}
		\functor^\alg \colon \adTop \to \discTop, \\ \functor^\sh \colon \adTop \to \etTop,
	\end{align*}
which are determined by the associations
	\begin{align*}
		(\cX , \cO) \in \adTop \mapsto (\cX, \cO^\alg) \in \discTop \\
		(\cX , \cO) \in \adTop \mapsto (\cX, \cO^\sh) \in \etTop,
	\end{align*}
respectively.
\end{notation}

\begin{definition}
Let $\cX$ be an $\infty$-topos. We denote by $(\adCAlg)^\heartsuit(\cX)$ the usual $1$-category of discrete \emph{$\Ok$-adic algebras} on $\cX$, i.e. discrete $\Ok$-algebras equipped with an adic topology compatible with the $(t)$-adic topology. Morphisms in $(\adCAlg)^\heartsuit(\cX)$ correspond to continuous ring morphisms for the adic topologies.
\end{definition}

\begin{definition}
	Let $\cX$ denote an $\infty$-topos. We define the \infcat of \emph{derived $\Ok$-adic algebras on $\cX$}, denoted $\adCAlg(\cX)$, via the pullback diagram
		\[
		\begin{tikzcd}
			\adCAlg(\cX) \ar{r} \ar{d} & \CAlg_{\Ok}(\cX) \ar{d}{\pi_0} \\
			(\adCAlg)^\heartsuit(\cX) \ar{r} & \CAlg_{\Ok}^\heartsuit(\cX),
		\end{tikzcd}
		\]
	computed in $\Cat$.
\end{definition}

\begin{lemma} \label{rem:magic}
Let $\cX \in \Top$ be an $\infty$-topos. The underlying algebra functor 
	\[
		\functor^\alg \colon \fCAlg( \cX) \to \CAlg_{\Ok}(\cX),
	\]
can be upgraded to a well defined functor 
	\[
		\functor^\ad \colon \fCAlg(\cX) \to \adCAlg(\cX).
	\]
\end{lemma}

\begin{proof}
We start by explicitly describe $\functor^\ad$ as follows: for each integer $n \geq 1$ and for
each $\cA \in \fCAlg(\cX)$, consider the canonical morphism 
	\[
		\cA^{\alg} \to \cA^\alg \otimes_{\Ok} \Ok_n \in \CAlg_{\Ok}(\cX).
	\]
Denote by $I_n \coloneqq \ker \left( \pi_0(\cA^\alg) \to \pi_0( \cA^\alg \otimes_{\Ok} \Ok_n ) \right)$. The sequence of ideals
$\{ I_n \}_{n \geq 1}$ defines an $I$-adic structure on $\cA^{\alg}$. The latter is further compatible with the $(t)$-adic topology on $\Ok$. Moreover, for every morphism $f \colon \cA \to \cB $ in $\fCAlg(\cX)$,
the underlying algebra morphism
	\[
		f^{\alg} \colon \cA^\alg \to \cB^\alg,
	\]
is compatible with the constructed adic topologies on both $\cA^{\alg}$ and $\cB^{\alg}$. Indeed, the latter can be checked directly at the level of $\pi_0$. In this case, the assertion follows from the fact that the composite
	\[
		\cA^\alg \to \cB^\alg \to \cB^\alg
		\otimes_{\Ok} \Ok_n,
	\]
induces a unique morphism $\cA^\alg \otimes_{\Ok} \Ok_n \to \cB^{\alg} \otimes_{\Ok} \Ok_n$, which is a consequence of the universal property of base change along $\Ok \to \Ok_n$.
By the construction of $\adCAlg(\cX)$, we conclude that the $\functor^\alg \colon \fCAlg(\cX) \to \CAlg_{\Ok}
(\cX)$ can be upgraded to a well-defined functor
	\[
		\functor^{\ad} \colon \fCAlg(\cX) \to \adCAlg (\cX),
	\]
as desired.
\end{proof}

\begin{proposition}
Both functors
	\begin{align*}
		\functor^\alg \colon \adTop \to \discTop \\
		 \functor^\sh \colon \adTop \to \etphTop,
	\end{align*}
admit right adjoints
	\begin{align*}
		L \colon \discTop \to \adTop \\
		L^{\sh} \colon \etphTop \to \adTop).
	\end{align*}
\end{proposition}

\begin{proof}
This is an immediate consequence of \cite[Theorem 2.1]{lurie2011dag}.
\end{proof}

We now proceed to have a better understanding of the action of $L$ at the level of $\Tdisc(\Ok)$-structures:

\begin{construction} \label{const1}
Let $(\cX, \cO) \in \adTop$ be a $\Tad$-structured $\infty$-topos. Consider the comma \infcat $\fCAlg(\cX)_{/ \cO}$. The latter is a presentable \infcat thanks to \cite[Corollary 9.4]{porta2015derived}. The underlying algebra functor induces a well defined functor
	\[
		\functor^\alg \colon \fCAlg(\cX)_{/ \cO} \to \CAlg_{\Ok}(\cX)_{/ \cO^{\alg}}.
	\]
Thanks to \cite[Corollary 9.5]{porta2015derived} the above functor commutes with limits and sifted colimits. Furthermore the Adjoint functor theorem implies that
	\[
		\functor^\alg \colon \fCAlg(\cX)_{/ \cO} \to \CAlg_{\Ok}(\cX)_{/ \cO^{\alg}},
	\]
admits a left adjoint which
we shall denote by $\Psi_{\cX} \colon  \CAlg_{\Ok}(\cX)_{/ \cO^{\alg}} \to  \fCAlg(\cX)_{/ \cO}$. If the underlying $\infty$-topos $\cX$ is clear from the context, we shall denote $\Psi_\cX$ simply by $\Psi$.
\end{construction}

\begin{construction} \label{const:2}
Let $\cA \in \CAlg_{\Ok}(\cX)_{/ \cO^{\alg}}$ be a $\Tdisc(\Ok)$-structure on $\cX$. We define $\cA_n$ as the pushout of the diagram
	\begin{equation} \label{eq:construction_of_A_n_as_pushout}
	\begin{tikzcd}
		\cA[u] \ar{r}{u \mapsto t^n} \ar{d}  &\cA \ar{d} \\
		\cA \ar{r}{u \mapsto 0} & \cA_n,
	\end{tikzcd}
	\end{equation}
in the \infcat $\CAlg_{\Ok}(\cX)_{/ \cO^{\alg}_n}$. We have denoted $\cA[u]$ the (commutative) derived free algebra on one generator in degree $0$, over $\cA$. As $\Psi$ is a left adjoint, we obtain a pushout square
	\begin{equation} \label{1}
	\begin{tikzcd}
		\Psi(\cA[u] ) \ar{r}{u \mapsto t^n} \ar{d}  & \Psi( \cA ) \ar{d} \\
		\Psi( \cA ) \ar{r}{u \mapsto 0} & \Psi( \cA_n),
	\end{tikzcd}
	\end{equation}
in the \infcat $\fCAlg(\cX)_{/ \cO_n}$.
Every epimorphism is effective in an $\infty$-topos. Moreover, $\Psi$ is a left adjoint, thus it preserves epimorphisms. These facts combined imply that the top horizontal morphism displayed in \eqref{1} is an effective epimorphism, on $\cX$. The transformation of
pregeometries $\Tdisc(\Ok) \to \Tad$ is unramified. It thus follows from \cite[Proposition 10.3]{lurie_closed} that we have a pushout diagram
	\[
	\begin{tikzcd}
		\Psi(\cA[u] )^\alg \ar{r}{u \mapsto t^n} \ar{d}  & \Psi( \cA )^\alg \ar{d} \\
		\Psi( \cA )^\alg \ar{r}{u \mapsto 0} & \Psi( \cA_n)^\alg,
	\end{tikzcd}
	\]
in the \infcat $\CAlg_{\Ok}(\cX)_{/ \cO^{\alg}_n}$. Thus, for each integer $n \geq 1$, the unit of the adjunction $(\Psi, \functor^\alg)$ induces morphisms
	\[
		f_{\cA, n } \colon \cA_n \to \Psi( \cA)_n^{\alg}.
	\]
Since $\cA_n$ can be realized as pushout of the diagram \eqref{eq:construction_of_A_n_as_pushout} and $\Psi$ is a left adjoint we have a natural equivalence 
	\[\Psi(\cA_n) \simeq \Psi(\cA)_n.\]
Therefore, we can consider $f_{\cA, n}$ naturally as a morphism
	\[
		f_{\cA, n} \colon \cA_n \to \Psi(\cA_n)^{\alg}.	
	\]
Moreover, the ideals
	\[
		I_n \coloneqq \ker \left( \pi_0(\cA ) \to \pi_0(\cA_n) \right),
	\]
are mapped, under $f_{\cA, n}$, to the ideals
	\[
		J_n \coloneqq \ker \left( \pi_0 \left(  \Psi( \cA)^{\alg} \right) \to \pi_0 \left(   \Psi( \cA)^{\alg}_n \right) \right).
	\]
Therefore, the universal property of $(t)$-completion induces a
canonical morphism 
	\[
		f_{\cA} \colon \cA^\wedge_t \to \Psi( \cA)^\alg,
	\]
in the \infcat $\CAlg_{\Ok}(\cX)$. Moreover, the natural morphism
	\[
		f_{\cA} \colon \cA^{\wedge}_t \to \Psi(\cA)^\alg,
	\]
is continuous with respect to the $I$-adic and $J$-adic topologies on $\cA$ and $\Psi(\cA)^\alg$, respectively. For this reason, we can naturally consider the morphism $f_{\cA}$ as a morphism in the \infcat $\adCAlg
(\cX)$. This latter assertion is a consequence of \cref{rem:magic}.
\end{construction}

\begin{definition} \label{st_hen} Let $(\cX, \cO) \in \Top( \Tet)$.
Let $\cA \in \CAlg_{\Ok}(\cX)_{/ \cO}$, we say that $\cA$ is \emph{strictly Henselian} if it belongs to the essential image of the functor $\CAlg^\sh_{\Ok}(\cX)_{/ \cO} \to \CAlg_{\Ok}(\cX)_{/ \cO^\alg}$, given on objects by the formula
	\[
		\cA \in \CAlg^\sh_{\Ok}(\cX)_{/ \cO} \mapsto \cA^\alg \coloneqq \cA  \circ \iota \in \CAlg_{\Ok}(\cX)_{/ \cO^\alg}.
	\]
Here $\iota \colon \Tdisc(\Ok) \to \Tet(\Ok)$ denotes the canonical transformation of pregeometries.
\end{definition}

We wish to prove that $\Psi(\cA)^\alg$ identifies with the $(t)$-completion of $\cA$, via the morphism $f_\cA$, constructed in \cref{const:2}. In order to establish this result, we need a few preliminaries:

\begin{lemma} \label{mauro_result}
Let $(F, G) \colon \cC \to \cD$ be an adjunction of presentable $\infty$-categories. Suppose further that:
\begin{enumerate}
\item Any epimorphism in $\mathcal{C}$ is effective;
\item $G$ is conservative, preserves epimorphisms and sifted colimits;
\end{enumerate}
Then epimorphisms in $\mathcal{D}$ are also effective. Moreover, if $\{ X_{\alpha} \}$ is a family of compact generators for $\mathcal{C}$ then the family $\{ F(X_{\alpha}) \}$ generates $\mathcal{D}$, under sifted colimits.
\end{lemma}

\begin{proof}
Let $g \colon V \to Y $ be an epimorphism in the \infcat $\cD$. We wan to show that it is effective, that is
the canonical morphism $g' \colon Y' \coloneqq \vert \check{\cC} (g) \vert  \to Y$, where $\check{\cC}(g)$ denotes the geometric realization of the 
Cech nerve of $g$, is an equivalence in $\cD$. By assumption, $G(g)$ is an epimorphism. Since $G$ is a right adjoint,
we have a canonical equivalence
	\[
		G \left( \check{\cC} (g) \right) \simeq \check{\cC} \left( G(g) \right).
	\]
As $G$ commutes with sifted colimits, we see that $G(Y') \simeq \vert \check{\cC} \left( G(g) \right) \vert \simeq G(Y)$, in $\cD$.
We thus conclude that $Y' \simeq Y$ using the conservativity of $G$. This finishes the proof of the first assertion.

Let $Y \in \cD$. We can find a filtered category $I$ and a diagram $T \colon I \to \cC$ such that
	\[
		\colim_{\alpha \in I} T_\alpha \simeq G(Y) \in \cC.
	\]
Consider the composition $F \circ T \colon I \to \cD$. For every $\alpha \in I$, we obtain a natural map
	\[
		\varphi_\alpha \colon F(T_\alpha) \to F(G(Y)) \to Y,
	\]
where the latter morphism is induced by the counit of the adjunction $(F, G)$. These maps $\varphi_\alpha$ can be arranged
into a cocone from $F \circ T $ to $Y$. For each $\alpha$, we can form the \v{C}ech nerve $\check{\cC}(\varphi_\alpha)$.
This produces a functor
	\[
		\widetilde{T} \colon I \times \mathbf \Delta^{\op} \to \cD,
	\]
informally defined by 
	\[
		(\alpha, n) \mapsto \check{\cC} (\varphi_\alpha)^n.
	\]
There is a natural cocone from $\widetilde{T}$ to $Y$, and we claim that the induced map
	\[
		\psi \colon \colim_{(\alpha, n ) \in I \times \mathbf \Delta^{\op}} \widetilde{T}(\alpha, n) \to Y
	\]
is an equivalence. We remark that
	\[
		\colim_{(\alpha, n) \in I \times \mathbf \Delta^\op} \widetilde{T}(\alpha, n) \simeq \colim_{n \in \mathbf \Delta^{\op}}
		 \colim_{\alpha \in I} \widetilde{T}(\alpha, n).
	\]
Since $G$ is conservative, it is enough to check that $G(\psi)$ is an equivalence. Observe that, since $G$ commutes with
sifted colimits, we have 
	\[
		G \left(  \colim_{n \in \mathbf \Delta^{\op}} \colim_{\alpha \in I} \widetilde{T}(\alpha, n) \right) \simeq
		 \colim_{n \in \mathbf \Delta^{\op}}
		 \colim_{\alpha \in I} G \left(  \widetilde{T}(\alpha, n) \right).
	\]
Since $I$ is a filtered category and $G$ is a right adjoint, we obtain:
	\[
		G \left( \colim_{\alpha \in I } \check{\cC} (\varphi_\alpha)^n \right) \simeq
		\check{\cC} \left( \colim_{\alpha \in I} G(F(T_\alpha)) \to G(Y) \right)^n.
	\]
The unit of the adjunction $(F, G)$ provide us with maps $\eta_\alpha \colon T_\alpha \to G(F(T_\alpha))$ such that the
induced composition
	\[
		\colim_{\alpha \in I} T_\alpha \to \colim_{\alpha \in I} G(F(T_\alpha)) \to G(Y)
	\]
is an equivalence. In particular, the map
	\[
		\colim_{\alpha \in I} G(F(T_\alpha)) \to G(Y)
	\]
is an effective epimorphism. Thus,
	\[
		\colim_{(\alpha, n) \in \mathbf \Delta^{\op}} G(\widetilde{T} (\alpha, n)) \simeq 
		\vert \check{\cC} ( \colim_{\alpha \in I} G(F(T_\alpha)) \to G(Y) \vert \simeq G(Y).
	\]
We deduce then that $G(\psi)$ is an equivalence. By conservativity of $G$ we conclude that $\psi$ was an equivalence to start with.
\end{proof}

\begin{remark} 
Notice that the functor $\CAlg^\sh_{\Ok}(\cX)_{/ \cO} \to \CAlg_{\Ok}(\cX)_{/ \cO}$ introduced in \cref{st_hen} is fully faithful. This follows from \cite[Proposition 4.3.19, Remark 2.5.13]{lurie2011dag} combined with \cite[Proposition 7.2.1.14]{lurie2009higher}
and the proof of \cite[Proposition 9.2]{porta2015derived}. Therefore, we will usually consider $\CAlg^\sh_{\Ok}(\cX)_{/ \cO}$ as a full subcategory of $ \CAlg_{\Ok}(\cX)_{/ \cO}$.
\end{remark}

We can now understand explicitly the composite $\functor^\alg \circ \Psi$:

\begin{proposition} \label{t_comp}
Let $(\cX, \cO) \in \discTop$. Suppose that the underlying $\infty$-topos $\cX$ has enough points.
Let $\cA \in \CAlg_{\Ok}(\cX)_{/ \cO}$ be an almost of finite presentation derived $\Ok$-algebra on $\cX$, which is further assumed to be strictly
Henselian. Then the canonical map
	\[
		f_{\cA} \colon  \cA^{\wedge}_t \to \Psi \left( \cA \right)^\alg ,
	\]
introduced in \cref{const:2}, is an equivalence in the \infcat $\adCAlg(\cX)_{/ \Psi( \cO )^{\alg}}$.
\end{proposition}

\begin{proof}
We wish to show that the natural map 
	\[
		f_{\cA} \colon \cA^\wedge_t \to \Psi( \cA)^{\alg}	,
	\]
constructed in \cref{const:2}, is an equivalence whenever $\cA \in \CAlg_{\Ok}(\cX)_{/ \cO^\alg}$ is almost of finite presentation.

By hypothesis, $\cX$ has enough geometric points. Thus, in order to show that $f_{\mathcal{A}}$ is an equivalence it suffices to show that its inverse image under any geometric point 
	\[
		(x^{-1}, x_*) \colon \cX \rightleftarrows \cS,
	\]
denoted $x^{-1} f_{\mathcal{A}}$, is an equivalence in the \infcat $\CAlg_{\Ok}$. Set $ A \coloneqq x^{-1} \mathcal{A}$. Thanks to \cite[Theorem 1.12]{porta_square_zero}, we deduce that $\Psi_{\cS}(A)^{\mathrm{alg}} \simeq x^{-1} \Psi(\mathcal{A})^{\mathrm{alg}}$.
We
are thus
reduced to the case where $\cX = \cS$.

The \infcat $\left( \CAlg_{\Ok} \right)_{/ \cO^\alg}$ is generated under sifted colimits by free objects of the form $\{ \Ok [T_1, \dots, T_m ] \}_{m \geq 1}$. 
Thanks to \cref{mauro_result} we conclude that $\left( \fCAlg \right)_{/ \cO} \coloneqq \left( \fCAlg \right)_{/ \cO}(\cS)
$ is generated under sifted colimits by the family $\{ \Psi( \Ok [ T_1, \dots T_m]) \}_m$. As $A \in \left( \CAlg_{\Ok} \right)_{/ x^{-1} \cO^{\mathrm{alg}}}$ is almost of
finite presentation we conclude that it can be written as a retract of a filtered colimit of a diagram of the form
	\[
		A_0 \to A_1 \to A_2 \to \dots,
	\]
where $A_0$ is an ordinary commutative ring of finite presentation over $\Ok$ and $A_{i+1} $ can be obtained from $A_i$ as the following pushout
	\begin{equation} \label{eq:1}
	\begin{tikzcd}
		\Ok [S^n] \ar{r} \ar{d} & \Ok[X] \ar{d} \\
		A_i \ar{r}  & A_{i+1}.
	\end{tikzcd}
	\end{equation}
We have denoted $\Ok [S^n]$ the free simplicial $\Ok$-algebra generated in degree $n$ by a single generator. Notice that, since $A$ is almost of finite presentation we can choose the above diagram in such a way that, for $i > 0$, sufficiently large, we have surjections $
\pi_0(A_i) \to
\pi_0(A_{i+1})$.
As $\Psi$ is a left adjoint it commutes, in particular, with pushout diagrams. We conclude that the diagram
	\begin{equation} 
	\begin{tikzcd}
		\Psi(\Ok [S^n]) \ar{r} \ar{d} & \Psi(\Ok[X]) \ar{d} \\
		\Psi(A_i) \ar{r} & \Psi( A_{i+1}),
	\end{tikzcd}
	\end{equation} 
is a pushout diagram in the \infcat $\fCAlg(\cX)_{/ \cO}$. Moreover,
the morphism $\Psi( A_i) \to \Psi( A_{i+1})$ is an epimorphism on $\pi_0$. For each $n > 0$, the morphism $\Ok[S^n] \to \Ok[X]$ is an effective epimorphism. As $\Psi$ is a left adjoint, the morphism $\Psi( \Ok[S^n] ) \to \Psi(\Ok[X])$ is an
epimorphism in the (hypercomplete) $\infty$-topos $\cX$ and thus an effective epimorphism. Thanks to
\cite[Proposition 3.14]{porta2016derived} it follows that the morphism
	\[
		\Psi(\Ok[S^n])^{\alg} \to \Psi(\Ok[X])^\alg
	\]
is an effective epimorphism, as well. Moreover, the transformation of pregeometries $\theta \colon \Tet(\Ok) \to \Tad$ is unramified, see \cref{prop:unramified_of_und_alg}.
It follows, cf. \cite[Propositon 10.3]{lurie_closed}, that the diagram,
	\begin{equation} \label{eq_1_formal}
	\begin{tikzcd}
		\Psi(\Ok[S^{n} ])^{\mathrm{alg}} \ar[r] \ar[d] & \Psi(\Ok [X])^{\mathrm{alg}} \ar[d] \\
		\Psi(A_i)^{\mathrm{alg}} \ar[r] & \Psi(A_{i+1})^{\mathrm{alg}},
	\end{tikzcd}
	\end{equation}
is a pushout square in $\CAlg_{\Ok}$. By induction we might assume that $\Psi( A_i)^{\mathrm{alg}}$ is equivalent to $(A_i)^{\wedge}_t$.

The transformation of pregeometries $(-)^\wedge_t \colon \cT_{\et}(\Ok) \to \Tad$ is given by $(t)$-completion along the $(t)$-locus. Therefore, one has a canonical equivalence
	\[
		\Psi( \Ok [ X ] )^{\mathrm{alg}} \simeq	\Ok 
		\langle X \rangle^\sh,
	\]
where the latter denotes the $(t)$-completion of the strict Henselianization of $\Ok [X]$. 
We now claim that the natural map
	\[
		\Psi( \Ok [ S^n])^{\mathrm{alg}} \to \big( \Ok [S^n]^\sh \big)^{\wedge}_t
	\]
is an equivalence: notice that $\Ok[S^n]$ fits into a pushout diagram
	\[
	\begin{tikzcd}
		\Ok[S^{n-1} ] \ar[r] \ar[d] & \Ok [X] \ar[d] \\
		\Ok [X] \ar[r] & \Ok [S^n].
	\end{tikzcd}
	\]
The result then follows by induction on $n \geq0$ and the case $n=0$ was already treated.
Since
	\[
		\Psi( A_i)^{\mathrm{alg}} \to \Psi(A_{i+1})^{\mathrm{alg}}
	\]
is surjective on $\pi_0$, it follows that $\pi_0( \Psi(A_{i+1})^{\mathrm{alg}})$ is $(t)$-complete. For each $i \geq 0$ the $\pi_0 \big( \Psi( A_{i+1})^{\alg} \big)$-modules $\pi_n\big( \Psi ( \mathcal{A}_{i+1} )^\alg \big)$ are of finite
presentation, thus they are $(t)$-adic complete 
$\pi_0(\Psi(A_{i+1})^{\mathrm{alg}})$-modules.
It follows that $\Psi(A_{i+1})^{\mathrm{alg}}$ is $(t)$-complete by \cite[Theorem 7.3.4.1]{lurie2016spectral}. 

Let now $A_{i+1} \to B$ be a morphism in $\CAlg_{\Ok}$ whose target is strictly Henselian and $(t)$-complete. Thanks to \eqref{eq:1}, such morphism induces morphisms 
	\[
		A_i \to B, \quad \Ok[T] \to B
	,\]
compatible with both $\Ok[S^n] \to \Ok[T]
$ and $\Ok[S^n] \to A_i$, in the \infcat $\CAlg_{\Ok}$.
By induction, the effect of $\functor^\alg \circ \Psi$ on 
	\[
	A_i, \quad \Ok[S^n] \quad \textrm{and } \Ok[X]
	\]
agrees with the composite of strict henselianization followed by $(t)$-completion.
Since $B$ is both strictly Henselian and $(t)$-complete, it follows that the map $A_{i+1} \to B$ induces a well defined morphism from the diagram displayed in \eqref{eq_1_formal}
to $B$. 
It follows that $\Psi(A_{i+1})^{\mathrm{alg}}$ satisfies the universal property of the $(t)$-completion of the derived $\Ok$-algebra
$A_{i+1}$. As $\Psi \left( A_{i+1} \right)^{\alg}$ is $(t)$-complete
we conclude that the morphism
	\[
		f_{A_{i+1}}:  (A^\sh_{i+1})^{\wedge}_t \to \Psi( A_{i+1})^{\mathrm{alg}} ,
	\] 
where $A^\sh_{i+1}$ denotes the strict henselianization of $A_{i+1}$, is necessarily an equivalence.
Let now 
	\[
		A \coloneqq \colim_i A_i,
	\]
in the \infcat $\CAlg_{\Ok}$. Fix $i \geq 0$, then 
	\[
		\tau_{\leq i} (\Psi( A)^{\mathrm{alg}}) \simeq \tau_{ \leq i } ( \Psi( A_j)^{\mathrm{alg}}),
	\]
for $j$ sufficiently large. We conclude then that $
\pi_i(\Psi(A)^{\mathrm{alg}})$ is 
$(t)$-adic complete for $i \geq 0$. \cite[Theorem 7.3.4.1]{lurie2016spectral} implies that $
\Psi(A)^{\mathrm{alg}}$ is also $(t)$-complete. Reasoning as before we conclude that it satisfies the universal property of $(t)$-completion of $A$. It follows that
	\[
		f_A:   A^{\wedge}_t \to \Psi(A)^{\mathrm{alg}}
	\]
is an equivalence in the \infcat $\CAlg_{\Ok}$, as desired.
\end{proof}

\begin{warning}
The functor $(-)^{\mathrm{alg}} \circ \Psi $ is not in general equivalent to the $(t)$-completion functor $(-)^{\wedge}_t$. In fact, both $(-)^{\mathrm{alg}}$ and $\Psi$ commute with filtered colimits, thus so it does $(-)^{\alg} \circ \Psi$.
This is not the case of the $(t)$-completion functor, in general. 
\end{warning}

We will need also the following ingredient:

\begin{construction} \label{magic:const}
Denote by $\Ok_n$ the reduction of $\Ok$ modulo $(t^n)$. Reduction modulo $(t^n)$ induces a transformation of pregeometries
	\begin{align*}
		p_n \colon \Tad  & \to  \Tdisc(\Ok_n) \\
		\Spf R  & \mapsto  \Spec R_n,
	\end{align*}
where $R_n \coloneqq R \otimes_{\Ok} \Ok_n$. Given $\cX \in \Top$, precomposition along $p_n$ induces a functor 
	\[
		p_n^{-1} \colon \CAlg_{\Ok_n}(\cX) \to \fCAlg(\cX),
	\]
given on objects by the formula
	\[
		\cO \in \CAlg_{\Ok_n}(\cX) \mapsto p_n^{-1} \cO \coloneqq \cO \circ p_n \in \fCAlg(\cX).
	\]
Consequently, we have a well defined functor
	\[
		p_n^{-1} \colon \discTopn \to \adTop ,
	\]
given on objects by the formula
	\[
		(\cX, \cO ) \in \discTop \mapsto (\cX, \cO \circ p_n ) \in \adTop.
	\]
\end{construction}

\begin{remark}
We have a commutative triangle of transformations of pregeometries
	\[
	\begin{tikzcd}
		\Tdisc(\Ok) \ar{rr}{\tcomp} \ar{rrd}[swap]{ - \otimes_{\Ok} \Ok_n} &  & \Tad \ar{d}{p_n} \\
		& & \Tdisc(\Ok_n)
	\end{tikzcd}.
	\]
For this reason, for every $\cX \in \Top$, it follows that the composite 
	\[
		\functor^\alg \circ p_n^{-1}  \colon \CAlg_{\Ok_n}(\cX) \to \CAlg_{\Ok}(\cX),
	\]
coincides with the usual forgetful functor $\CAlg_{\Ok_n} (\cX) \to \CAlg_{\Ok}(\cX)$ along
the map $\Ok \to \Ok_n$. Notice that the latter functor admits a left adjoint which is given by extension of scalars along $\Ok \to \Ok_n$, i.e. it is given on objects by the formula
	\[
		\cO \in \CAlg_{\Ok}(\cX) \mapsto \cO \otimes_{\Ok} \Ok_n \in \CAlg_{\Ok}(\cX)
	\]
 \end{remark}

\begin{notation} \label{not:1}
We will denote by $(-)_n \colon \CAlg_{\Ok}(\cX) \to   \CAlg_{\Ok_n}(\cX)$ the base change functor
	\[
		\cO \in \CAlg_{\Ok}(\cX) \mapsto   \cO_n \coloneqq \cO \otimes_{\Ok} \Ok_n \in \CAlg_{\Ok}(\cX).
	\]
\end{notation}

It follows by \cite[Theorem 2.1]{lurie2011dag} that $p_n^{-1}$ admits a right adjoint $L_n \colon \adTop \to \discTopn$ which we can explicitly describe:

\begin{proposition} \label{magic:prop}
The functor $p_n^{-1} \colon \discTopn \to \adTop$ admits a right adjoint
	\[
		L_n \colon \adTop \to \discTopn
	\]
whose restriction to the full subcategory of $\ \adTop$, spanned by pairs $(\cX, \cO) $ whose underlying $\infty$-topos $\cX$ has enough points, is given by the formula
	\[
		(\cX, \cO ) \in \adTop \mapsto (\cX, \cO^{\alg}_n ) \in \discTopn.
	\]
\end{proposition}

\begin{proof}
The existence of a left adjoint $L_n \colon \adTop \to \discTopn$ follows directly from \cite[Theorem 2.1]{lurie2011dag}. Let $(\cX, \cO) \in \discTopn$ and $(\cY, \cO') \in \adTop$. Assume that both $\cX, \ \cY \in \Top$ have enough points. Given any geometric
morphism $(f^{-1}, f_*) \colon \cX \to \cY$ we have a morphism of fiber sequences of the form
	\begin{equation} \label{mid}
	\begin{tikzcd}
		\Map_{\fCAlg(\cX) } \left( f^{-1} \cO', p_n^{-1} \cO \right) \ar{r} \ar{d}{q} &  \Map_{\adTop} \left( (\cX, p_n^{-1} \cO), (\cY, \cO') \right) \ar{d}{p} \ar{r} & \Map_{\Top} \left( \cX, \cY \right) \ar[equal]{d} \\
		\Map_{\CAlg_{\Ok_n}(\cX) } \left( ( f^{-1} \cO')^{\alg}_n, \cO \right) \ar{r} & \Map_{\discTopn} \left( (\cX, \cO), (\cY, (\cO')^{\alg}_n) \right) \ar{r} & \Map_{\Top} \left( \cX, \cY \right)
	\end{tikzcd}.
	\end{equation}
Moreover, the morphism $q \colon \Map_{\fCAlg(\cX) } \left( f^{-1} \cO', p_n^{-1}, \cO \right)  \to \Map_{\CAlg_{\Ok_n}(\cX) } \left( ( f^{-1} \cO')^{\alg}_n, \cO \right) $ coincides with the composite
	\[
	\begin{tikzcd}[column sep = small]
		\Map_{\fCAlg(\cX) } \left( f^{-1} \cO', p_n^{-1} \cO \right) \ar{r}{(\textrm{-})^\alg} & \Map_{\CAlg_{\Ok}(\cX)} \left( (f^{-1} \cO')^{\alg}, p_n^{-1}\cO^{\alg} \right) \ar{r}{}  & \Map_{\CAlg_{\Ok_n}(\cX)} \left( (f^{-1} \cO')_n^{\alg}, \cO \right).
	\end{tikzcd}
	\]
In order to prove the assertion of the proposition it suffices to show that the morphism $p$ displayed in \eqref{mid} is an equivalence of mapping spaces.
Thanks to the fact that the horizontal arrow diagrams in \eqref{mid} form fiber sequences we are reduced to prove that $q$ is an
equivalence
of mapping spaces. As both $\cX$ and $\cY$ have enough points we reduce ourselves to prove the statement of the Theorem at the level of stalks. For this reason we can assume from the start that $\cX = \cS = \cY$. Both target and source of $q$ commute with filtered colimits on
the first argument, thus we are reduced, as in the proof of \cref{t_comp} to prove that $q$ is an equivalence whenever $f^{-1} \cO' \simeq \Psi \left( \Ok [ T_1, \dots, T_n ]  \right)$. We have natural equivalences of mapping spaces
	\begin{align*}
		\Map_{\fCAlg} \left( \Psi( \Ok [T_1, \dots T_m ]) , p_n^{-1} \cO \right) & \simeq  \Map_{\CAlg_{\Ok}} \left( \Ok  [T_1, \dots T_m ],( p_n^{-1} \cO  )^{\alg} \right) \\
			 & \simeq  \Map_{\CAlg_{\Ok_n}} \left( \Ok  [T_1, \dots T_m ]_n, ( p_n^{-1} \cO )^{\alg} \right) \\
			 & \simeq  \Map_{\CAlg_{\Ok_n}} \left(   \Ok_n   [ T_1, \dots T_m ] , \cO \right) .
	\end{align*}
The result now follows from the observation that $\Psi \left( \Ok [T_1, \dots T_m ]  \right)^{\alg}_n \simeq \Ok_n [T_1, \dots, T_m]$ in the \infcat $\CAlg_{\Ok_n}$, which is a direct consequence \cref{t_comp}.
\end{proof}

\begin{corollary}
Let $\cX \in \Top$ be an $\infty$-topos. The functor $L_n \colon \adTop \to \discTopn$, introduced in \cref{magic:prop}, induces a well defined functor at the level of the corresponding \infcats of structures
	\[
		(-)^{\ad}_n \colon \fCAlg(\cX) \to \CAlg_{\Ok_n}(\cX),
	\]
given on objects by the formula
	\[
		\cO \in \fCAlg(\cX) \mapsto \cO^{\alg}_n \in \CAlg_{\Ok_n}(\cX).
	\]
Moreover, the functor $(-)^{\ad}_n$ is a left adjoint to the forgetful $p_n^{-1} \colon \CAlg_{\Ok_n} (\cX) \to \fCAlg(\cX)$.
\end{corollary}

\begin{proof}
The existence of $(-)^{\ad}_n $ is guaranteed by \cref{magic:prop}. The fact that $(-)^{\ad}_n$ is a left adjoint to $p_n^{-1} \colon \CAlg_{\Ok_n} (\cX) \to \fCAlg(\cX)$ follows from the proof of \cref{magic:prop} together with the fact that
both $(-)^{\ad}_n$ and $p_n^{-1}$ are defined at the level of \infcats of structures on the same underlying $\infty$-topos.
\end{proof}

\begin{notation}
Consider the forgetful functor $\discTopn \to \discTop$ given by restriction of scalars along the morphism $\Ok \to \Ok_n$. We will denote $- \times_{\Spec \Ok} \Spec \Ok_n \colon \colon \discTop \to \discTopn$ its right adjoint.
\end{notation}

\begin{corollary}
For each $n \geq 1$, the composite $L_n \circ L \colon \discTop \to \discTopn$ coincides with the base change functor 
	\begin{align*}
		- \times_{\Spec \Ok} \Spec \Ok_n \colon  \discTop \to \discTopn, \\
		(\cX, \cO) \in \discTop \mapsto (\cX, \cO) \times_{\Spec \Ok} \Spec \Ok_n \in \discTopn
	\end{align*}
\end{corollary}

\begin{proof}
This is a direct consequence of the definitions together with the commutative triangle displayed in \cref{magic:const}.
\end{proof}

\subsection{Comparison with derived formal geometry} Our main goal now is to give comparison statements between $\Tad$-structured $\infty$-topoi and \emph{locally adic ringed $\infty$-topoi}. The latter corresponding to pairs $(\cX, \cO)$ where $\cO$ is a
$\adCAlg$-valued sheaf on the $\infty$-topos $\cX$. This provides an explicit comparison between a simplicial analogue of Lurie's original definition of spectral $\Ok$-adic Deligne-Mumford stacks and ours.

\begin{definition}
Let $\cX$ be an $\infty$-topos and $\cA  \in \adCAlg(\cX)$.
We say that $\cA$ is \emph{topologically almost of finite presentation} if $\cA$ is $(t)$-complete, the sheaf $\pi_0(\cA)$ is topologically of finite presentation and, for each $i>0$, the homotopy sheaf
$\pi_i(\cA)$ is coherent as a $\pi_0(\cA)$-module. We shall denote by $\CAlg_{\Ok}^{\ad, \taft}(\cX)_{/ \cO^\ad}$ the full subcategory of $\adCAlg(\cX)_{/ \cO^\ad}$ spanned by topologically almost of finite presentation $\cA \in \adCAlg(\cX)$.
\end{definition}

\begin{definition}
Let $\cX$ be an $\infty$-topos and consider the functor $\functor^\ad \colon \fCAlg(\cX) \to \adCAlg(\cX)$ introduced in \cref{rem:magic}. We say that $\cA \in \fCAlg(\cX)$ is \emph{topologically almost of finite presentation} if the underlying sheaf of adic algebras
$\cA^\ad$ is topologically almost of finite presentation.
We denote $\fCAlg^{\taft}(\cX)$ the \infcat of topologically almost of finite presentation local $\Tad$-structures on $\cX$.
\end{definition}

\begin{construction} \label{const:adj_ad}
Consider the adjunction $\big(\Psi, \functor^\alg \big) \colon \CAlg_{\Ok}(\cX)_{/ \cO^\alg} \to \fCAlg(\cX)_{/ \cO}$, introduced in \cref{const1}. Let 
	\[
		\disc \colon \adCAlg(\cX) \to \CAlg_{\Ok}(\cX)
	\]
denote the canonical functor obtained by forgetting the adic structure. Then the pair
	\[
		\big( \Psi \circ \disc, \functor^\ad \big) \colon
		\adCAlg(\cX)_{/ \cO^\ad} \to \fCAlg(\cX)_{/ \cO}
	\] 
forms an adjunction pair after restriction
	\[
		\big( \Psi^\ad, \functor^\ad \big) \coloneqq \big( \Psi \circ \disc, \functor^\ad \big) \colon
		\CAlg_{\Ok}^{\ad, \taft}(\cX)_{/ \cO^\ad} \to \fCAlg^\taft(\cX)_{/ \cO}.
	\]
In order to see this consider the unit
	\[\mathrm{id} \to \functor^\alg \circ \Psi\]
of the adjunction in \cref{const1}.
It follows by the construction of $\functor^\ad \colon \fCAlg(\cX) \to \adCAlg(\cX)$ that we have an equivalence 
	\[
		\functor^\alg \simeq \disc \circ \functor^\ad
	\]
in the \infcat $\Fun \big( \fCAlg(\cX)_{/ \cO} , \CAlg_{\Ok}(\cX)_{/ \cO^\alg} \big)$. Therefore, for each $\cA \in \CAlg_{\Ok}^{\ad, \taft}(\cX)_{/ \cO^\ad}$
the unit of the adjunction
	\[
		\cA^\mathrm{disc} \to \left( \Psi(\cA^{\mathrm{disc}}) \right)^\alg 
	\]
induces a canonically defined, up to a contractible space of choices, morphism
	\[
		\cA \simeq \cA^\wedge_t \to \left( \Psi^\ad(\cA) \right)^\ad.
	\]
This construction is functorial and it satisfies the universal property of a unit of adjunction. Therefore we obtain an
adjunction $\left(\Psi^\ad, \functor^\ad \right) \colon \CAlg_{\Ok}^{\ad, \taft}(\cX)_{/ \cO^\ad} \to \fCAlg^\taft(\cX)_{/ \cO}$, as desired.
\end{construction}

\begin{notation} Let $\cX$ be an $\infty$-topos.
	We denote by $\adCAlg(\cX)^{\sh} \coloneqq \adCAlg(\cX) \times_{\CAlg_{\Ok}(\cX)} \CAlg_{\Ok}(\cX)^\sh$.
\end{notation}

\begin{theorem} \label{rect}
Let $\cX$ be an $\infty$-topos with enough geometric points. Consider the functor
	\[
		\functor^\ad \colon \fCAlg(\cX)_{/ \cO} \to \adCAlg(\cX)_{/ \cO^\ad},
	\]
introduced in \cref{rem:magic}. Then the induced restriction functor
	\[
		\functor^\ad \colon \fCAlg^\taft(\cX)_{/ \cO} \to \adCAlg(\cX)^\sh_{/ \cO^\ad},
	\]
is fully faithful and its essential image agrees with the full subcategory of $\adCAlg(\cX)^\sh_{/ \cO^\ad}$ spanned by those strictly henselian $\cA \in \adCAlg(\cX)_{/ \cO^\ad}$ topologically almost of finite presentation.
\end{theorem}

\begin{proof}
Consider the adjunction 
	\[\left(\Psi^\ad, \functor^\ad \right) \colon \CAlg_{\Ok}^{\ad, \taft}(\cX)_{/ \cO^\ad} \to \fCAlg^\taft(\cX)_{/ \cO},\]
of \cref{const:adj_ad}.
Thanks to \cref{t_comp} the composite $\functor^\ad \circ \Psi^\ad $ is an equivalence when restricted to the subcategory $\cC \subseteq \CAlg_{\Ok}^{\ad, \taft}(\cX)$ spanned by strictly Henselian objects. Therefore the left
adjoint functor 
	\[
		\Psi^\ad \colon \CAlg_{\Ok}^{\ad, \taft}(\cX)_{/ \cO^\ad} \to
		\fCAlg(\cX)_{/ \cO}
	\] 
is fully faithfully when restricted to the full subcategory $\cC$.
\cite[Lemma 3.13]{porta2016derived} implies that
the right adjoint functor $\functor^\ad$ is conservative, the conclusion now follows.
\end{proof}

\begin{remark}
\cref{rect} can be interpreted as a rectification statement. Indeed, an element 
	\[\cA \in \fCAlg(\cX),\]
is a functor $\cA \colon \Tad \to \cX$ satisfying the axioms of the definition of a $\Tad$-structure on $\cX$. Furthermore, morphisms
	\[
	\cA \to \cB,
	\]
in $\fCAlg(\cX)$ are
local morphisms in $\Fun \left( \Tad, \cX \right)$.
On the other hand, the \infcat $\adCAlg(\cX)$ admits a simpler description. Its objects are derived $\Ok$-algebras on $\cX$, which admit an adic topology on$\pi_0$
and morphisms are continuous 
\emph{local} adic morphisms 
of derived $\Ok$-algebras on $\cX$. 
\end{remark}

\begin{construction}[The $\Spf$-construction] \label{const:the_Spf_construction} Let $A \in \adCAlg$ be a derived adic $\Ok$-algebra. We can associate to $A$ an object 
	\[\Spf A \coloneqq (\cX_A, \cO_A ) \in \adTop,
	\]
as follows: we let 
	\[\cX_A \coloneqq \cH \Shv^{\ad}_A \in \Top,\]
denote the hypercompletion
of the $\infty$-topos $\Shv^{\ad}_A$, introduced in \cite[Notation 8.1.1.8]{lurie2016spectral}. We then define 
	\[\cO_A \colon \Tad \to \cX_A,\]
as the $\Tad$-structure on $\cX_A$ determined by the formula
	\[
		\Spf (R) \in \Tad \mapsto  \left( B \in \CAlg^{\ad, \et}_A \mapsto \Map_{\adCAlg} \left( R, B \right) \right).
	\]
Where $\CAlg_{A}^{\ad, \et}$ denotes the full subcategory of $\CAlg_{A}^{\ad}$ spanned by those derived $A$-algebras $B$ \'etale over $A$. One checks directly that $\cO_A \colon \Tad \to \cX_A$ is indeed a $\Tad$-structure on $\cX_A$. Such association is
functorial in $A \in \adCAlg$. We are provided with a well defined functor (up to contractible indeterminacy)
	\[
		\Spf \colon \left( \adCAlg \right)^{\op} \to \adTop.
	\]
We refer to the latter functor as the \emph{$\Spf$-construction functor}.
\end{construction}

\begin{remark}
Given $A \in \adCAlg$, it follows immediately from the definitions that 
	\[\Spf (A)^\ad \coloneqq (\cX_A, \cO_A^\ad) \in \discTop,\]
agrees with the "simplicial version" of the $\Spf$-construction introduced in \cite[$\S 8.1.1$]{lurie2016spectral}
\end{remark}

\begin{remark}
Let $n \geq 1$ and consider the right adjoint functor 
	\[L_n \colon \adTop \to \discTopn,\]
introduced in \cref{not:1}. Given $A \in \adCAlg$, it follows by the description of \cref{magic:prop} that 
	\[L_n \left( \Spf (A) \right) \simeq \left( \cX_A, \cO_{A, n} \right),\]
where $\cO_{A, n} \coloneqq \cO^{\alg}_A
\otimes_{\Ok}  \Ok_n \in \CAlg_{\Ok}(\cX)$.
\end{remark}

\begin{proposition} \label{prop:Spf_construction_is_f_f}
The functor $\Spf \colon \big( \adCAlg \big)^{\op} \to \adTop$ is fully faithful. Moreover, its essentially image agrees with the full subcategory of $\adTop$ spanned by pairs $(\cX, \cO) \in \adTop$ such that $(\cX, \cO^\alg) \in \discTop$
is equivalent to an affine formal spectrum as in \cref{const:the_Spf_construction}.
\end{proposition}

\begin{proof}
Let $A ,  \ B \in \adCAlg$ and consider the corresponding formal spectrums $\Spf (A) $ and $\Spf (B) \in \adTop$. 
The datum of a morphism of local $\Tad$-structures $f \colon \Spf (A ) \to \Spf (B)$ is equivalent to the datum of a geometric morphism of $\infty$-topoi $( f^{-1},
f_*) \colon \cX_A \to \cY_B$ together with a local morphism $\alpha \colon f^{-1} \cO_{ B} \to
\cO_{A}$ of $\Tad$-structures on $\cX_A$. Applying the underlying algebra functor at the level of structures we obtain a morphism 
	\[
		\alpha^{\alg} \colon
		f^{-1} \left( \cO^{\ad}_{ B} \right)^{\alg}  \to \left( \cO^{\ad}_{A} \right)^{\alg},
	\]
in the \infcat $\fCAlg(\cX_A)$.
The unit of the adjunction $(f^{-1},f_*)$ produces a well defined morphism of derived $\Ok$-algebras $\phi \colon B \to A $, up to contractible indeterminacy. 

By the construction of the underlying $\infty$-topoi of both $\Spf (A)$ and $\Spf(B)$ together with \cite[Remark 8.1.1.7]{lurie2016spectral} it follows that the morphism $\phi \colon B \to A$ is continuous with respect to the
adic topologies on both $A$ and $B$. We obtain thus a well defined morphism of mapping spaces
	\[\label{eq:2t}
		\Phi \colon \Map_{ \adTop } \left( \Spf A , \Spf B  \right) \to \Map_{\adCAlg} \left( B , A \right).
	\]
Let $\phi \colon B \to A $ be a continuous morphism of derived $\Ok$-adic algebras. In order to show that the functor
	\[
		\Spf \colon \big( \adCAlg \big)^{\op} \to \adTop
	\]
is fully faithful it suffices to show that the fiber $Z_\phi := \mathrm{fib}_\phi(\Phi)$ is contractible for any choice of $\phi$.
To any continuous adic morphism, we can attach a well defined, up to contractible indeterminacy,
morphism on the corresponding (formal) \'etale sites. We have thus a canonical morphism at the level of mapping spaces
 	\[
		\theta \colon \Map_{\adCAlg} \left( B, A \right) \to \Map_{ \Top } \left( \cX_A, \cY_B \right).
	\]
Let $(f^{-1}, f_*) \colon \cX_A \to \cY_B$ be a morphism of $\infty$-topoi which is equivalent to $\theta(\phi).$ The
fiber over $(f^{-1}, f_*)$ induces a fiber sequence of mapping spaces: 
	\[\label{eq:1}
	\begin{tikzcd}[column sep = small]
		\Map_{\fCAlg(\cX)} \left( f^{-1} \cO_{B}, \cO_{A} \right) \ar{r}  & \Map_{ \adTop } \left(  \Spf A, \Spf B  \right)  \ar{r}{\theta}  & \Map_{ \Top}\left( \cX_A, \cY_B \right).
	\end{tikzcd}
	\]
Consider the commutative diagram in the \infcat $\cS$
	\begin{equation} \label{ii}
	\begin{tikzcd}
		Z_{\phi} \ar{r} \ar{d} & \Map_{\fCAlg(\cX)} \left(  f^{-1} \cO_{B}, \cO_{A} \right) \ar{r} \ar{d} & \Map_{ \adTop } \left(  \Spf (A), \Spf (B)  \right) \ar{d} \\
		\{ \phi \} \ar{r} & W \ar{r} \ar{d} & \Map_{\adCAlg} \left( B , A \right) \ar{d} \\
		& \{ (f^{-1}, f_* ) \} \ar{r} & \Map_{\Top }( \cX_A, \cY_B)
	\end{tikzcd}.
	\end{equation}
Both the upper rectangle and and the bottom right squares are pullback diagrams. It follows that we can identify $Z_{\phi}$ with the pullback
	\[
		Z_{\phi} \simeq \Map_{\fCAlg(\cX)} \left(  f^{-1} \cO_{B}, \cO_{A} \right) \times_{		W	} \{ \phi \}.
	\]
Let $F \colon \Spf (A) \to \Spf(  B)$ be a morphism of $\Tad$-structured $\infty$-topoi such that $\Phi(F) \simeq \phi$. It follows by \cite[Remark 8.1.1.7]{lurie2016spectral} that the induced geometric morphism $(f^{-1}, f_*) \colon \cX_A 
\to \cY_B$ can be identified with the restriction to closed subtopoi of the geometric morphism of $\infty$-topoi $\cX_A \to \cY_B$. Thanks to the proof of
\cite[Proposition 1.4.2.4]{lurie2016spectral} it follows that the latter is uniquely determined. For this reason, $(f^{-1}, f_*)$ is also uniquely determined by $\phi$, up to a contractible space of choices.
As a consequence we can identify $Z_\phi$ with the fiber product:
	\[
		Z_\phi \simeq \Map_{\fCAlg(\cX)} \left( f^{-1}\cO_{B}, \cO_{A} \right) \times_{		W		} 	 \{\phi \}.
	\]
We have a sequence of natural equivalences of mapping spaces
	\begin{align*} 
		\Map_{\fCAlg(\cX)} \left( f^{-1}\cO_{B}, \cO_{A} \right) \times_{		W		} 	 \{\phi \} & \simeq 
		 \Map_{\fCAlg(\cX)} \left( f^{-1} \cO_{B}, \lim_{n \geq 1} \left(\cO_{A, n} \right) \right) \times_{		W		} 	 \{\phi \}  \\
		 & \simeq \left(   \lim_{n \geq 1} \Map_{\fCAlg(\cX)} \left( f^{-1} \cO_{B} , \cO_{A, n}  \right) \right) \times_{		W		} \{\phi\}   .
	\end{align*}
We can further identify the last term with
	\begin{align*} \label{eq:sim}
		\left(   \lim_{n \geq 1} \Map_{\fCAlg(\cX)} \left( f^{-1} \cO_{B} , \cO_{A, n}  \right) \right) \times_{		W		} \{\phi\}  \simeq & \\
		\left( \lim_{n \geq 1} \Map_{\CAlg_{\Ok_n} (\cX)} \left(  f^{-1} \cO_{B,n} , \cO_{ A, n} \right) \right) \times_{	W	} \{\phi\} . 
	\end{align*}
For each $n \geq 1$, denote $\phi_n$ the base change of $\phi$ to $\Ok_n$. Passing to the limit over $n \geq 1$ we can further identify the last term with
	\begin{equation} \label{lim} 
	\begin{split}
		\lim_{n \geq 1} \left( \Map_{\CAlg_{\Ok_n}( \cX)} \left( f^{-1} \cO_{B, n}  , \cO_{A, n}  \right) \times_{	W_n 		}   \{ \phi_n\}	 \right) 
		\simeq \lim_{n \geq 1} \left( \Map_{\discTopn  } \left( A_n , B_n \right) \times_{W_n} \{\phi_n \} \right).
	\end{split}
	\end{equation}
Where $W_n $ is defined as the fiber product of the corresponding diagram obtained as the reduction modulo $t^n$ of the bottom right square, displayed in \eqref{ii}. Thanks to the proof of \cite[Corollary 1.2.3.5.]{lurie2016spectral}
each term in the limit displayed in \eqref{lim} can be identified with
	\[
		 \Map_{\CAlg_{\Ok_n}}  \left( B_n , A_n \right) \times_{\CAlg_{\Ok_n} \left( B_n, A_n \right)} \{\phi_n \}.
	\] 
The latter is a contractible space. The result now follows by a simple analysis on the corresponding Milnor exact fiber sequence.
\end{proof}

\begin{definition}
A \emph{derived formal Deligne-Mumford stack}, over $\Spf \Ok$ is a locally $\Ok$-adic ringed $\infty$-topos $(\cX, \cO)$, where $\cX$ is assumed to be hypercomplete, $\cO \in \adCAlg(\cX)$ and there exists an effective epimorphism $\coprod_i U_i \to 1_{\cX}$, in $\cX$, such that, for each index $i$, $(\cX_{/ U_i}, \cO_{\vert U_i})$ is equivalent to $(\Spf A_i)^\ad$, for suitable $A _i \in \adCAlg$.

We denote by $\dfDM_{\Ok}^{\mathrm{Lurie}}$ the full subcategory of the $\Ok$-adic ringed $\infty$-topoi spanned by derived formal Deligne-Mumford stacks.
\end{definition}

\begin{definition} \label{defin:derived O_k adic DM stack}
A derived $\Ok$-adic Deligne-Mumford stack is a pair $(\cX, \cO) \in \adTop$ such that $(\cX, \cO^\ad)$ is equivalent to a derived formal $\Ok$-Deligne-Mumford stack. We say that a derived $\Ok$-adic Deligne-Mumford stack
$(\cX, \cO)$ is \emph{topologically almost of finite presentation} if the $\Tad$-structure
\[\cO \in \fCAlg(\cX),\] is topologically almost of finite presentation.
% If the underlying $\infty$-topos $\cX$ is further supposed to be coherent (cf. \cite[\S 3]{Lurie_Spectral_Schemes}) then we say that the pair $(\cX, \cO)$ is \emph{topologically almost of finite presentation.}
\end{definition}

\begin{notation}
We denote $\dfDM_{\Ok} $ (resp., $\dfDM_{\Ok}^\taft$) the full subcategory of $\adTop$ spanned by derived $\Ok$-adic Deligne-Mumford stacks (resp., topologically almost of finite presentation $\Ok$-adic Deligne-Mumford stacks).
\end{notation}

\begin{definition}
We denote by $\dfSch$ the full subcategory of $\dfDM_{\Ok}$ spanned by those objects $\sfX = (\cX, \cO)$ such that
$(\cX, \pi_0 (\cO^\ad))$ is equivalent to an ordinary formal scheme over $\Ok$. We refer to objects in
$\dfSch$ as derived $\Ok$-adic formal schemes. We also define the \infcat of topologically almost of finite presentation
derived $\Ok$-adic schemes as $\dfSch^\taft \coloneqq \dfDM^\taft \cap \dfSch$.
\end{definition}

\begin{remark}
The functor $\Spf \colon \adCAlg \to \adTop$, introduced in \cref{const:the_Spf_construction}, factors through the fully faithful embedding $\dfSch \hookrightarrow \adTop$. 
\end{remark}

The following results compares Lurie's definition of derived $\Ok$-adic Deligne-Mumford stacks and the definition in terms of $\Tad$-structured $\infty$-topoi:

\begin{corollary} \label{cor:ad_DM_equivalent_to_formal_DM}
	The functor
		\[
			\functor^\ad \colon \dfDM_{\Ok}^{\taft} \to \dfDM_{\Ok}^{\mathrm{Lur}},
		\]
	given on objects by the association
		\[
			(\cX, \cO) \in \dfDM_{\Ok}^{\taft} \mapsto (\cX, \cO^\ad) \in \dfDM_{\Ok}^{\mathrm{Lur}},
		\]
	is fully faithful and its essential image agrees with the full subcategory spanned by pairs $(\cX, \cO)$ such that $\cO$ is topologically almost of finite presentation.
\end{corollary}

\begin{proof}
 Let $\sfX \coloneqq (\cX, \cO_\sfX)$ and $\sfY \coloneqq (\cY, \cO_\sfY) $ be derived $\Ok$-adic Deligne-Mumford stacks. For each geometric morphism of $\infty$-topoi $f_* \colon \cX \rightleftarrows \cY \colon f^{-1}$, we have a commutative diagram of fiber sequences of mapping spaces
 	\[
	\begin{tikzcd}
		\Map_{\fCAlg(\cX)}(f^{-1} \cO_\sfY, \cO_\sfX) \ar{r} \ar{d}{\functor^\ad} & \Map_{\dfDM_{\Ok}}(\sfX, \sfY) \ar{r} \ar{d} & \Map_{\Top}(\cX, \cY) \ar{d}{=} \\
		\Map_{\adCAlg(\cX)}(f^{-1} \cO^\ad_\sfY, \cO^\ad_\sfX) \ar{r} & \Map_{\dfDM_{\Ok}^{\mathrm{Lurie}}}(\sfX^\ad, \sfY^\ad) \ar{r} & \Map_{\Top}(\cX, \cY) .
	\end{tikzcd}
	\]
It follows from \cref{rect} that the left vertical morphism is an equivalence whenever $\cO_\sfX$ and $\cO_\sfY$ are topologically almost of finite presentation $\Tad$-structures. Fully faithfulness of the restriction of $\functor^\ad \colon \dfDM_{\Ok} \to \dfDM^\mathrm{Lurie}_{\Ok}$ to $\dfDM^\taft_{\Ok} \subseteq \dfDM_{\Ok}$, now follows.

Let now $\sfX \coloneqq (\cX, \cO) \in \dfDM^\mathrm{Lurie}_{\Ok}$ be such that $\cO$ is topologically almost of finite presentation. Then, by \cref{rect}, there exists an essentially unique object $\cO' \in \fCAlg(\cX)$ such that $(\cO')^\ad \simeq \cO $ in $\adCAlg(\cX)$. By definition, $(\cX, \cO') \in \dfDM_{\Ok}^\taft$ and the result follows.
\end{proof}

\subsection{Derived \infcats of modules for $\Tad$-structured spaces} Let $\cC$ be an \infcat which admits finite limits. The \emph{abelianization} of $\cC$, denoted $\Ab(\cC)$, was defined in \cite[Definition 4.2]{porta2017representability}. On the other hand, the \emph{stabilization} of $\cC$ was introduced in \cite[Definition 1.4.2.8]{lurie2012higher}. We shall denote the latter as $\Sp(\cC)$.

\begin{definition}
Let $\sfX \coloneqq (\cX, \cO) \in \adTop$. We define the \infcat of \emph{$\cO$-modules} on $\sfX$ as
	\[
		\Mod_{\cO} \coloneqq \Sp \left( \Ab \left( \fCAlg(\cX)_{/ \cO} \right) \right).
	\]
Similarly, we define the \infcat $\Mod_{\cO^\alg}$ as the \infcat of $\cO^\alg$-module objects of $\Shv_{\cD(\Ab)}(\cX)$, where $\cD(\Ab)$ denotes the derived stable \infcat of abelian groups.
\end{definition}

\begin{remark}
Let $(\cX, \cO)$ be as above. Both the \infcats $\Mod_{\cO}$ and $\Mod_{\cO^\alg}$ are stable, by construction.
\end{remark}

\begin{construction}
Given $(\cX, \cO) \in \adTop$ we can also consider the \infcat of modules on its algebraization $(\cX, \cO^\alg)$ defined as $\Mod_{\cO^\alg} \coloneqq \Shv_{\cD(\Ab)} \left( \cX \right)$, where $\cD(\Ab) \coloneqq \Mod_{\bZ}$ denotes the
derived \infcat of $\bZ$-modules. Thanks to \cite[Theorem 8.3.4.13]{lurie2012higher} one has a natural equivalence
	\[
		\Mod_{\cO^\alg} \simeq \Sp \left( \Ab \left( \CAlg_{\Ok}(\cX)_{\cO^\alg} \right) \right),
	\]
in the \infcat $\Cat^\st$. As the underlying algebra functor $\functor^\alg \colon \fCAlg(\cX)_{/ \cO} \to \CAlg_{\Ok}(\cX)_{/ \cO^\alg}$ is a right adjoint it induces an exact functor at the level of derived \infcats of modules denoted
	\[
		g^\alg \colon \Mod_{\cO} \to \Mod_{\cO^\alg}.
	\]
\end{construction}

\begin{lemma}
The left adjoint $\Psi \colon \CAlg_{\Ok}(\cX)_{/ \cO^\alg} \to \fCAlg(\cX)_{\Psi(\cO^\alg)}$ induces a well defined functor
	\[
		f^\ad \colon \Mod_{\cO^\alg} \to \Mod_{\cO},
	\]
which is a left adjoint to $g^\alg$. 
\end{lemma}

\begin{proof}
It suffices to prove that $\Psi$ commutes with finite limits. By \cref{t_comp}, the composite $\functor^\alg \circ \Psi$ agrees with the $(t)$-completion functor on the full subcategory of almost of finite presentation objects, denoted
$\CAlg_{\Ok}(\cX)_{/ \cO}^{\mathrm{afp}}$. Consequently, it commutes with small limits on $\CAlg_{\Ok}(\cX)_{/ \cO^\alg}^{\mathrm{afp}}$. As $\functor^\alg$ is a conservative right adjoint, it follows that $\Psi$ itself commutes with finite limits on $\CAlg_{\Ok}(\cX)_{/ \cO^\alg}
^{\mathrm{afp}}$.
Let now $\cA \in \CAlg_{\Ok}(\cX)_{/ \cO^\alg}$ be a general object. We can realize $\cA$ as a filtered colimit of almost of finite presentation objects in $\CAlg_{\Ok}(\cX)_{/ \cO^\alg}$. Let $\{ \cA_i \}_i$ be a diagram indexed by a finite \infcat $I$, and for each
$i \in I$ choose a presentation
	\[
		\cA_i \simeq \colim_{m \in J} \cA_{i, m},
	\]
where $\cA_{i, m}$ is almost of finite presentation and $J$ is a filtered $\infty$-category. We have thus a sequence of natural equivalences in the \infcat $\CAlg_{\Ok}(\cX)_{/ \cO^\alg}$
	\begin{align*}
		\Psi(\underset{i}{\lim} \cA_i )^\alg & \simeq \colim_m \Psi(\underset{i}\lim \cA_{i, m})^\alg &  \\
		& \simeq    \colim_m \underset{i}{\lim}(\cA_{i, m})^\wedge_t \simeq  \underset{i}{\lim} \colim_m (\cA_{i, m})^\wedge_t   \\
		&  \simeq  \underset{i}{ \lim}\colim_m \Psi(\cA_{i, m})^\alg \simeq  \underset{i}{\lim} \Psi(\cA_{i, m})^\alg,
	\end{align*}
and the conclusion now follows as in the preceding case.
\end{proof}

\begin{proposition} \label{equiv}
Suppose $\cX$ has enough points and $\Psi(\cO^\alg)^\alg \simeq \cO^\alg$. Then the functor 
	\[
		g^\alg \colon \Mod_{\cO} \to \Mod_{\cO^\alg}
	\]
is an equivalence of stable \infcats.
\end{proposition}

\begin{proof}
Let $f^\ad \colon \Mod_{\cO^\alg} \to \Mod_{\cO}$ denote a left adjoint to $g^\alg$. Thanks to the previous proposition $f^\ad$ corresponds to the \emph{stabilization} of functor $\Psi$, introduced in \cref{const1}. We want to show that $f^\ad$ is an inverse to $g^\alg$, as functors. The functor $g^\alg$ is conservative
as $\functor^\alg$ was already conservative. Therefore, we are reduced to show that $f^\ad$ is a fully faithfully functor. It suffices to show that the unit $\eta$ of the adjunction $(f^\ad, g^\alg)$ is an equivalence. 

As $\cX$ has enough geometric points we reduce
ourselves to proof the last assertion at the level of stalks. 
We are thus reduced to the case $\cX = \cS$. In this case, the \infcat $\Mod_{\cO^\alg}$ is compactly generated by $\cO^\alg \in \Mod_{\cO^\alg}$.
The underlying algebra $\functor^\alg$ commutes with filtered colimits (even sifted
colimits). It follows that $g^\alg$ also commutes with filtered colimits. As $g^\alg$ is an exact functor between stable \infcats we conclude that it commutes with all colimits.
Therefore, the unit $\eta $ commutes with colimits. We are thus reduced to check that $\eta$ is an equivalence on the compact generator
$\cO^\alg \in \Mod_{\cO^\alg}$. By our assumption on $\cO^\alg$ the result follows thanks to \cref{t_comp}.
\end{proof}

\begin{definition}
The equivalence of stable \infcats provided in \cref{equiv} allows us to define a t-structure on the \infcat $\Mod_{\cO}$ by
means of the functor $g^\alg$: the \infcat $\Mod_{\cO^\alg}$ admits a natural t-structure, \cite[Proposition 1.7]{Lurie_Spectral_Schemes}. Therefore, we set $\cF \in \Mod_{\cO}$ to be connective if and only if $g^\alg(\cF) \in \Mod_{\cO^\alg}$ is connective.
\end{definition}

Thanks to \cite[Corollary 6, page 165]{bosch2005lectures}, every topologically almost of finite presentation (discrete) $k^\circ$-algebra is coherent. Therefore, the following definition is reasonable:

\begin{definition}
Let $\sfX = (\cX, \cO) \in \adTop$ be a topologically almost of finite presentation $\cT_{\ad}$-structured $\infty$-topos. We define the \infcat of \emph{coherent $\cO$-modules on $\sfX$}, $\Coh(\sfX)$, as the full subcategory $\Mod_{\cO}$ spanned by those $\cF$ such that, for each integer $i  ,$ the homotopy sheaves $\pi_i(\cF)$ are 
coherent $
\pi_0(\cO^\alg)$-modules which vanish for sufficiently small $i $.
\end{definition}

\subsection{$\Ok$-adic cotangent complex} In this \S, we will introduce the notion of formal cotangent complex, which will prove to be of fundamental importance for us: we have a projection functor
	\[
		\Omega^\infty_{\ad} \colon \Mod_{\cO} \to \fCAlg(\cX)_{/ \cO},
	\]
which is given by evaluation on the object $(S^0, * ) \in \cS^{\mathrm{fin}}_* \times \cT_{\Ab}$. The functor $\Omega^\infty_\ad$ admits a left adjoint
	\[
		\Sigma^\infty_\ad \colon \fCAlg(\cX)_{/ \cO} \to \Mod_{\cO}.
	\]
We refer the reader to \cite[\S 5.1]{porta2017representability} and \cite[\S 7.5]{lurie2012higher} for more details about these constructions.

\begin{definition}
Let $M \in \Mod_{\cO}$, we say that $\cO \oplus M \coloneqq \Omega^\infty_\ad (M)$ is the \emph{trivial adic square-zero extension of $\cO$ by the module $M$}.
\end{definition}

\begin{definition}
Let $\sfX \coloneqq (\cX, \cO) \in \adTop$ and let $\cA \in \fCAlg(\cX)_{/ \cO}$ be a $\Tad$-structure on $\cX$. Given $M \in \Mod_{\cO}$, we define the space of $\cA$-\emph{linear adic derivations} of $\cO$ with values in $M$ as
	\[
		\Der^\ad_{\cA} \left( \cO, M \right) \coloneqq \Map_{\fCAlg(\cX)_{ \cA / / \cO} } \left( \cO, \cO \oplus M \right) \in \cS.
	\]	
\end{definition}

\begin{proposition}
The functor 
	\[
		\Der^\ad_{\cA} \left( \cO, - \right) \colon \Mod_{\cO } \to \cS,
	\]
is corepresentable by an object 
	\[
		\bL^\ad_{\cO/ \cA} \in \Mod_{\cO},
	\]
which we refer to as the \emph{adic cotangent complex relative to} $\cO \to \cA$.
\end{proposition}

\begin{proof}
The proof is a direct consequence of the existence of a left adjoint $\Sigma^\infty_\ad \colon \fCAlg(\cX)_{/ \cO} \to \Mod_{\cO}$. Set $\bL^\ad_{\cO / \cA} \coloneqq \Sigma^\infty_{\ad} \left( \cO \otimes_{\cA} \cO \right)$. For every $M \in \Mod_{\cO}$ we have
a sequence of natural equivalences of mapping spaces
	\begin{align*}
		\Der^\ad_\cA(\cO, M) 
		&\simeq  \Map_{\fCAlg(\cX)_{\cA // \cO}} \left( \cO, \cO \oplus M \right) \\
		& \simeq  \Map_{\fCAlg(\cX)_{\cA // \cO}} \left( \cO, \Omega^\infty_{\ad} ( M ) \right) \\
		& \simeq  \Map_{\fCAlg(\cX)_{\cO // \cO}} \left( \cO \otimes_{\cA} \cO , \Omega^\infty_{\ad} ( M) \right) \\
		& \simeq  \Map_{\Mod_{\cO}} \left( \Sigma^\infty_\ad ( \cO \otimes_{\cA} \cO) , \cO \right) \\
		& \simeq \Map_{\Mod_{\cO}} \left( \bL^\ad_{\cO / \cA} , M \right).
	\end{align*}
The proof is now finished.
\end{proof}

\begin{proposition} \label{compact}
Let $\cA \to \cB$ be a morphism in $\fCAlg(\cX)_{/ \cO}$, topologically almost of finite presentation. Then $\bL^\ad_{\cB/ \cA}$
is a compact object in the \infcat $\Mod_{\cO}$.
\end{proposition}

\begin{proof}
The proof of \cite[Proposition 4.1.2.1]{lurie2016spectral} applies.
\end{proof}

\begin{construction} \label{com:} Thanks to \cite[Lemma 5.15]{porta2017representability}
we have a commutative diagram of \infcats
	\[
	\begin{tikzcd}
		\Mod_{\cO} \ar{r}{g^\alg} \ar{d}{\Omega^\infty_\ad} & \Mod_{\cO^\alg} \ar{d}{\Omega^\infty} \\
		\fCAlg(\cX)_{/ \cO} \ar{r}{\functor^\alg} & \CAlg_{\Ok}(\cX)_{/ \cO^\alg},
	\end{tikzcd}
	\]
therefore passing to left adjoints we obtain a commutative diagram
	\begin{equation} \label{com:cot}
	\begin{tikzcd}
		\Mod_{\cO}  & \Mod_{\cO^\alg} \ar{l}{f^\ad} \\
		\fCAlg(\cX)_{/ \cO} \ar{u}{\Sigma^\infty_\ad} & \CAlg_{\Ok}(\cX)_{/ \cO} \ar{u}{\Sigma^\infty} \ar{l}{\Psi}
	\end{tikzcd},
	\end{equation}
in the \infcat $\Cat$. The commutativity of \eqref{com:cot} provide us with a natural map
	\[
		f^\ad \left( \bL_{\cB^\alg/ \cA^\alg} \right) \to \bL_{\cB / \cA}^\ad
	\]
in the \infcat $\Mod_{\cO}$.
\end{construction}

\begin{proposition} 
Let $\cA \to \cB$ be a morphism in $\fCAlg(\cX)_{/ \cO}$. Consider the algebraic cotangent complex $\bL_{\cB^\alg / 
\cA^\alg}$ associated to the morphism $\cA^\alg \to \cB^\alg$. Then the natural map introduced in \cref{com:}
	\[
		f^\ad \left( \bL_{\cB^\alg/ \cA^\alg} \right) \simeq \bL^\ad_{\cB / \cA}
	\]
is an equivalence in the \infcat $\Mod_{\cO}$.
\end{proposition}

\begin{proof} Let $\cA \to \cB$ be a morphism in $\fCAlg(\cX)_{/ \cO}$ and let $\bL_{\cB^\alg/ \cA^\alg} \in \Mod_{\cO^\alg}$ denote the
algebraic cotangent complex associated to $\cA^\alg \to \cB^\alg$.

The construction of the adic cotangent complex, via the left adjoint $\Sigma^\infty_\ad$, commutes with filtered colimits of local $\Tad$-structures.
Therefore we can suppose from the beginning that $\cA$ is topologically almost of finite presentation and $(t)$-complete. Applying the same reasoning, we might as well assume that $\cB$ is topologically almost of finite presentation $\cA$-algebra and
moreover $(t)$-complete. 
By applying the functor
	\[
		f^\ad \colon \Mod_{\cB^\alg} \to \Mod_{\cB},
	\]
we obtain the following sequence of natural equivalences of mapping spaces	
	\begin{align*}
		\Map_{\Mod_{\cB}} \left( f^{\ad}(\bL_{\cB^\alg / \cA^\alg}), M \right) & \simeq  \Map_{\Mod_{\cO^\alg}}
		\left( \bL_{\cB^\alg/ \cA^\alg}, g^\alg(M) \right) \\
		& \simeq  \Map_{\CAlg_{\Ok}(\cX)_{\cA^\alg / / \cB^\alg}} \left( \cB^\alg, \cB^\alg \oplus g^\alg(M) \right) \\
		& \simeq  \Map_{\CAlg_{\Ok}(\cX)_{\cA^\alg / / \cB^\alg}} \left( \cB^\alg, (\cB \oplus M )^\alg \right) \\
		& \simeq  \Map_{\fCAlg(\cX)_{\cA / / \cB}} \left( \Psi(\cB^{\alg}) , \cB \oplus M \right) \\
		& \simeq \Map_{\fCAlg(\cX)_{\cA / / \cB}} \left(\cB, \cB \oplus M \right).
	\end{align*}
The latter equivalence follows from the fact that $\cB$ is assumed to be topologically almost of finite presentation and $(t)$-complete combined with \cref{t_comp}.
\end{proof}

\begin{proposition}
Let $f \colon \cA \to \cB$ and $g \colon \cB \to \cC$ be morphisms in the \infcat
$\fCAlg(\cX)_{/ \cO}$. Then one has a fiber sequence
	\[
		\bL^\ad_{\cB / \cA} \otimes_{\cB} \cC \to \bL^\ad_{\cC/ \cA} \to \bL^\ad_{\cC/
		\cB},
	\]
in $\Mod_{\cO}$.
\end{proposition}

\begin{proof}
This is a direct consequence of \cite[Proposition 5.10]{porta2017representability}.
\end{proof}

\begin{proposition}
Suppose we are given a pushout diagram 
	\[
	\begin{tikzcd}
		\cA \ar{r} \ar{d} & \cB \ar{d} \\
		\cC \ar{r} & \cD
	\end{tikzcd}
	\]
in the \infcat $\fCAlg(\cX)_{/ \cO}$. Then the natural morphism
	\[
		\bL^\ad_{\cB / \cA} \otimes_{\cB} \cD \to  \bL^\ad_{\cD / \cC}
	\]
is an equivalence in the \infcat $\Mod_{\cO}$.
\end{proposition}

\begin{proof}
The assertion is a particular case of \cite[Proposition 5.12]{porta2017representability}.
\end{proof}

\subsection{Unramifiedness of $\Tad$}
In this \S, we prove that both the $\Ok$-adic pregeometry $\Tad$ and the transformation of pregeometries $\tcomp \colon \Tet (\Ok) \to \Tad$ are unramified. Unramifiedness of both $\Tad$ and $\tcomp$ was used in a crucial way in the proof of \cref{t_comp}. It will further play an important role when discussing the \emph{derived rigidification functor}, in \S 4.

\begin{definition}
Let $\cT$ be a pregeometry. We say that $\cT $ is \emph{unramified} if for every morphism $f \colon X \to Y$ in $\cT$ and every object $Z \in \cT$, the diagram
	\[
	\begin{tikzcd}
		X \times Z \ar{r} \ar{d} & X \times Y \times Z \ar{d} \\
		X \ar{r} & X \times Y
	\end{tikzcd}
	\]
induces a pullback diagram
	\[
	\begin{tikzcd}
		\cX_{ X \times Z } \ar{r} \ar{d} & \cX_{X \times Y \times Z} \ar{d} \\
		\cX_{X} \ar{r} & \cX_{X \times Y}
	\end{tikzcd}
	\]
in $\Top$. We have denoted $\cX_{ X \times Z }, \ \cX_{X \times Y \times Z}, \ \cX_{X}$ and $ \cX_{X \times Y}$ the underlying $\infty$-topoi associated to the absolute spectrum construction, introduced in \cite[\S 2.2]{lurie2011dag}.
\end{definition}

\begin{remark}
Both the pregeometries $\Tet(k)$ and $\Tan$ are unramified, see \cite[Proposition 4.1]{lurie_closed} and \cite[Corollary 3.11]{porta2016derived}, respectively.
\end{remark}

\begin{proposition} \label{unr:ad}
The pregeometry $\Tad$ is unramified.
\end{proposition}

\begin{proof}
Let $Z \in \Tad$ and denote $\cX_{Z}$ denote the underlying $\infty$-topos of the corresponding absolute spectrum $\Spec^{\Tad}(Z)$. The $\infty$-topos $\cX_Z$ is equivalent to the hypercompletion of the \'etale $\infty$-topos on the special fiber of $Z$. As
pullback diagrams are preserved by taking special fibers the result follows by unramifiedness of $\Tet(\Ok)$, cf. \cite[Proposition 4.1]{lurie_closed}. 
\end{proof}

There is also a notion of relative unramifiedness:

\begin{definition} \label{def_rel_unr}
Let $\varphi \colon \cT \to \cT'$ be a transformation of pregeometries, and let $\Phi \colon \Top(\cT') \to  \Top(\cT)$ the induced functor given on objects by the formula
	\[
		(\cX, \cO ) \in \Top(\cT') \mapsto (\cX, \cO \circ \varphi ) \in \Top(\cT).
	\]
We say that the transformation $f$ is \emph{unramified} if the following conditions are satisfied:
\begin{enumerate}
\item Both $\cT$ and $\cT'$ are unramified;
\item For every morphism $f \colon X \in Y$ and every object $Z \in \cT$, we have a pullback diagram
	\[
	\begin{tikzcd}
		\Phi \big( \Spec^{\cT'} \big( X \times Z \big) \big) \ar{r} \ar{d} & \Phi \big( \Spec^{\cT'} \big( X \times Y \times Z \big) \big) \ar{d} \\
		\Phi \big( \Spec^{\cT'} \big( Z \big) \big) \ar{r} & \Phi \big( \Spec^{\cT'} \big( X \times Y \big) \big)
	\end{tikzcd}
	\]
in the \infcat $\Top(\cT)$.
\end{enumerate}
\end{definition}

\begin{proposition}  \label{prop:unramified_of_und_alg}
The transformation of pregeometries $\tcomp \colon \cT_{\emph{\textrm{\'et}}}(\Ok) \to \Tad$ is unramified.
\end{proposition}

\begin{proof}
It suffices to prove condition (ii) in \cref{def_rel_unr}. \cref{prop:Spf_construction_is_f_f} implies that $\Phi \big( \Spec^{\Tad} (-) \big)$ is an ind-\'etale spectrum and the latter commutes with finite limits.
\end{proof}

\subsection{Postnikov towers of $\Ok$-adic spaces}

\begin{definition}
Let $\sfX = (\cX, \cO) \in \adTop$ and $M \in \left( \Mod_{\cO} \right)_{\geq 1}$ be an $\cO$-module concentrated in
homological degrees $\geq 1$. A \emph{$\Ok$-adic square zero extension of $\sfX$ by $M$} consists of a $\Tad$-adic
structured $\infty$-topos $\sfX' = (\cX, \cO')$ equipped with a morphism $f \colon \sfX \to \sfX'$ satisfying:
	\begin{enumerate}
		\item The underlying geometric morphism of $f$ is equivalent to the identity of $\cX$;
		\item There exists a $\Ok$-adic derivation
			\[
				d \colon \bL^\ad_{\sfX} \to M[1] \in \Mod_{\cO},
			\]
		such that we have a pullback diagram in the \infcat $\fCAlg(\cX)_{/ \cO}$
			\[
			\begin{tikzcd}
				\cO' \ar{r} \ar{d} & \cO \ar{d}{d} \\
				\cO \ar{r}{d_0} & \cO \oplus M [1]
			\end{tikzcd},
			\]
		where $d_0$ denotes the trivial $\Ok$-adic derivation.
	\end{enumerate}
\end{definition}

\begin{definition}
Let $\cT$ be a pregeometry and let $n \geq -1$ be an integer. We say that $\cT$ is \emph{compatible with $n$-truncations} if for every $\infty$-topos $\cX$, every $\cT$-structure $\cO \colon \cT \to \cX$ and every admissible morphism $U \to V$ in $\cT$, the
induced square
	\[
	\begin{tikzcd}
		\cO(U) \ar{r} \ar{d} & \tau_{\leq n} \cO(U) \ar{d} \\
		\cO(V) \ar{r} & \tau_{\leq n} \cO(U)
	\end{tikzcd},
	\]
is a pullback diagram in $\cX$.
\end{definition}

\begin{remark}
The above definition is equivalent to require that given a pair $(\cX, \cO) \in \Top(\cT)$, the truncation $(\cX, \tau_{\leq n} \cO) $ is again an object of the \infcat $\Top(\cT)$.
In other terms, $\tau_{\leq n} \cO\colon \cT \to \cX$ is still a $\cT$-structure on $\cX$. Here $\tau_{\leq n} \colon \cX \to \cX$ denotes the $n$-truncation functor on $\cX$.
\end{remark}

\begin{notation} \label{trun:not}
Let $\cT$ be a preogeometry compatible with $n$-truncations. We will denote $\Top(\cT)^{\leq n} \subseteq \Top(\cT)$ the full subcategory spanned by those pairs $(\cX, \cO)$ such that the $\cT$-structure $\cO \colon \cT \to \cX$ is $n$-truncated. 
\end{notation}

\begin{remark}
The inclusion functor $\Top(\cT)^{\leq n} \subseteq \Top(\cT)$ admits a right adjoint $\trun_{\leq n} \colon \Top ( \cT) \to \Top(\cT)^{\leq n}$ which is given on objects by the formula
	\[
		(\cX, \cO) \in \Top(\cT) \mapsto (\cX, \tau_{\leq n} \cO) \in \Top(\cT)^{\leq n}.
	\]
\end{remark}

\begin{lemma} \label{trun:lem}
The pregeometry $\Tad$ is compatible with $n$-truncations.
\end{lemma}

\begin{proof} Assume first that the valuation on $k$ is discretely valued.
We follow closely \cite[Proposition 4.3.28]{lurie2011dag}. Reasoning as in the proof of the cited reference or as in the proof of \cref{t_comp}, it suffices to prove the following assertion:
let $U \to V$ be an admissible morphism in $\Tad$ and $\cO \in \fCAlg(\cS)$. Then the commutative square
	\begin{equation} \label{n:trun}
	\begin{tikzcd}
		\cO(U) \ar{r} \ar{d} &\tau_{\leq 0} \cO(U) \ar{d} \\
		\cO(U) \ar{r} & \tau_{\leq 0}(V)
	\end{tikzcd}
	\end{equation}
is a pullback square in the $\infty$-topos $\cS$. By the definition of $\Tad$, there are $(t)$-complete ordinary $\Ok$-algebras $A$ and $B$ such that $U \cong \Spf A$ and $V \cong \Spf B$. Moreover, by construction, $B$ is \'etale over some ring of the form
$\Ok \langle T_1, \dots T_m \rangle$. \cite[Tag 0AR1, Lemma 84.10.3]{destacks} implies that there exists an \'etale $\Ok [T_1, \dots T_m]$-algebra $B'$ such that $B \cong (B')^\wedge_t$. The morphism $U \to V$ being admissible in $\Tad$ implies that the induced morphism $B \to A$ is formally \'etale. \cite[Tag 0AR1, Lemma 84.10.3]{destacks} implies that the morphism $B \to A$ can be realized
as the $(t)$-completion of $\Ok$-algebras $B' \to A'$, where $A'$ is an \'etale $\Ok [T_1, \dots, T_n]$ itself. Therefore, the morphism of spaces 
	\[
		\cO(U) \to \cO(V),
	\]
can be identified with a morphism
	\[
		\cO^\sh( \Spec A' ) \to \cO^\sh(\Spec B').
	\]
The same conclusion holds for the morphism $\tau_{\leq 0}( \cO(U) )\to \tau_{\leq 0 }( \cO(V))$. Therefore we can identify the diagram \eqref{n:trun} with the diagram
	\[
	\begin{tikzcd}
		\cO^\sh (\Spec A' ) \ar{r} \ar{d} & \tau_{\leq 0} \cO^\sh(\Spec A') \ar{d} \\
		\cO^\sh(\Spec B' ) \ar{r} & \tau_{\leq 0} \cO^\sh(\Spec B')
	\end{tikzcd}
	\]
in the \infcat $\cS$. The result now follows thanks to \cite[Proposition 4.3.28]{lurie2011dag}. 

Suppose now that $k$ is a rank $1$ valuation non-archimedean field. The proof follows along the same lines. Indeed the main input in the previous argument was
\cite[Tag 0AR1, Lemma 84.10.3]{destacks}. The latter result is implied by \cite[Lemma 84.7.4.]{destacks}, which is proved under noetherian assumptions. 

We claim that the proof of \cite[Tag 0AR1, Lemma 84.7.4]{destacks} holds over the ring of integers $\Ok$,
under the assumption that the morphism we start with is \'etale. We will employ the same notations as in \cite[TAG 0AR1, Lemma 84.7.4]{destacks}. In this case, the $\Ok$-adic cotangent complex $\bL^\ad_{A/ B}$ vanishes, by (formal) \'etaleness of $B \to A$.
Thus one verifies directly that the statement of \cite[TAG 0AR1, Lemma 84.4.1.]{destacks} holds in our setting. We can thus find $C$ is a finitely generated $\Ok$-algebra together with a surjection
	\[
		\varphi \colon C^\wedge \to A,
	\]
such that $C$ is \'etale over $B$. The latter assertion is a consequence of \cite[Lemma 84.7.4, (4) (ii)]{destacks}. 
For this reason, the morphism $\varphi$ is necessarily (formally) \'etale too. Consequently, $A$ is a retract of
$C^\wedge$, and thus can be obtained as the $(t)$-adic completion of $B \times_{C^\wedge} C$ which is itself a retract of $C$ over $B $. This last assertion follows from the fact that (classical) $(t)$-adic completion preserves finite limits of $B$-algebras.
\end{proof}

\begin{definition}
Let $\sfX = (\cX, \cO) \in \adTop$. We define its \textit{$n$-th truncation} as
	\[\mathrm t_{\leq n}(\sfX) \coloneqq (\cX, \tau_{\leq n} \cO)
\in \adTop.\]
\end{definition}

\begin{proposition}
Let $\sfX = (\cX, \cO) \in \adTop$. Then for each integer $n \geq 0$, the $(n+1)$-th
truncation $\mathrm t_{\leq n+1} (\sfX)$ is a square zero extension of $\mathrm t_{\leq n }(\sfX)$. In particular,
given a derived $\Ok$-adic Deligne-Mumford stack, $\sfX$, its $n$-th truncation
$\trun_{\leq n}(\sfX)$ is again a derived $\Ok$-adic Deligne-Mumford stack.
\end{proposition}

\begin{proof}
We have a canonical morphism $\trun_{\leq n}(\sfX) \hookrightarrow \trun_{\leq n+1} (\sfX)$ induced by the identity functor on the
underlying $\infty$-topos $\cX$ and the natural map $\tau_{\leq n+1} \cO \to \tau_{\leq n} \cO$ at the level of structures.
Let $\cB \coloneqq \tau_{\leq n+1} \cO$ and $\cA \coloneqq \tau_{\leq n} \cO$. Thanks to
\cite[Corollary 7.4.1.28]{lurie2012higher} we deduce that the underlying algebra morphism
	\[
		\cB^\alg \to \cA^\alg
	\]
is a square zero extension. Thus we can identify $\cB^\alg$ with the pullback of the diagram
	\begin{equation} \label{pull}
	\begin{tikzcd}
		\cB^\alg \ar{r} \ar{d} & \cA^\alg \ar{d}{d} \\
		\cA^\alg \ar{r}{d_0} & \cA^\alg \oplus \bL_{\cB^\alg / \cA^\alg}
	\end{tikzcd}
	\end{equation}
in the \infcat $\CAlg_{\Ok}(\cX)_{/ \tau_{\leq n} \cO}$. Consider the induced $\Ok$-adic derivation
	\[
		f^\ad (d) \colon \bL^\ad_{\cA} \to \bL^\ad_{\cB/ \cA},
	\]
and form the pullback diagram	
	\begin{equation} \label{eq:pull}
	\begin{tikzcd}
		\cB' \ar{r} \ar{d} & \cA \ar{d} \\
		\cA \ar{r}{d_0} & \cA \oplus \bL^\ad_{ \cB / \cA}
	\end{tikzcd}
	\end{equation}
in the \infcat $\fCAlg(\cX)_{/ \cA}$. In this way the canonical morphism $\cB' \to \cA$ is a $\Ok$-adic square zero extension
and there exists a morphism $\cB \to \cB'$. As filtered colimits commute with finite limits we reduce ourselves to the case
that $\cO$, and therefore both $\cA$ and $\cB$, are topologically almost of finite presentation. Thanks to \cref{t_comp}, the
functor $\Psi$ applied to the pullback diagram \eqref{pull} is the identity. Thus by conservativity of $\functor^\alg$ it follows
that the diagram 
	\begin{equation} \label{ab}
	\begin{tikzcd}
		\Psi(\cB^\alg) \ar{r} \ar{d} &\Psi( \cA^\alg) \ar{d}{d} \\
		\Psi(\cA^\alg) \ar{r}{d_0} & \Psi( \cA^\alg \oplus \bL_{\cB^\alg / \cA^\alg})
	\end{tikzcd}
	\end{equation}
is a pullback diagram in the \infcat $\fCAlg(\cX)_{/ \cA}$. Thanks to \cref{rect}, one concludes that the diagram \eqref{ab}
is equivalent to the pullback diagram \eqref{eq:pull}. Therefore, the canonical map $\cB' \to \cB$ is an equivalence in the
\infcat $\fCAlg(\cX)_{/ \cO}$, as desired.
\end{proof}

\section{Derived rigidification functor}
\subsection{Construction of the rigidification functor}Raynaud's generic fiber construction \cite[$\S8$]{bosch2005lectures}, induces a transformation of pregeometries 
	\[
		\rigg \colon \Tad \to \Tan.
	\]

\begin{proposition} \label{prop:rig_functor}
Precomposition along the transformation of pregeometries $\rigg \colon \Tad \to \Tan$ induces a functor
	\[
		\functor^+ \colon \anTop \to \adTop,
	\]
which admits a right adjoint denoted
	\[
		\rigg \colon \adTop \to \anTop.
	\]
\end{proposition}	

\begin{proof}
It is a direct consequence of \cite[Theorem 2.1]{lurie2011dag}.
\end{proof}

\begin{definition}
	We refer to the functor introduced in \cref{prop:rig_functor} as the \emph{rigidification functor}.
\end{definition}

\begin{lemma} \label{rig:trun}
For each integer $n \geq 0$, we have a commutative diagram
	\[
	\begin{tikzcd}
		\anTop \ar{d} \ar{r}{\functor^+}  & \adTop \ar{d} \\
		\anTop^{\leq n} \ar{r}{\functor^+} & \adTop^{\leq n}
	\end{tikzcd},
	\]
of \infcats.
\end{lemma}	

\begin{proof}
It follows immediately from the fact that both preogemetries $\Tan $ and $\Tad$ are compatible with $n$-truncations, cf. \cite[Theorem 3.23]{porta2016derived} and \cref{trun:lem}, in the $k$-analytic and $\Ok$-adic settings, respectively.
\end{proof}

\begin{corollary} \label{tr}
Let $n \geq -1$ be an integer. The diagram
	\[
	\begin{tikzcd}
		\adTop \ar{r}{\rigg} \ar{d}{\trun_{\leq n}}  & \anTop \ar{d}{\trun_{\leq n}} \\
		\adTop^{\leq n} \ar{r}{\rigg} & \anTop 
	\end{tikzcd}
	\]
is commutative. 
\end{corollary}

\begin{proof}
It follows by taking right adjoints in the diagram displayed in \cref{rig:trun}.
\end{proof}

These considerations imply the following useful result:

\begin{corollary} \label{rig_cl}
Let $\sfX = (\cX, \cO)$ be a $\Tad$-structured space which is equivalent to an ordinary $\Ok$-adic formal scheme
topologically of finite presentation. Then $\sfX^\rig$ is equivalent to an ordinary $k$-analytic space which agrees with the
usual generic fiber of $\sfX$.
\end{corollary}

\begin{proof}
The question is local on $\sfX$. We can thus assume that $\sfX \simeq \Spf (A)$, where $A \in \adCAlg$ is
a topologically of finite presentation ordinary $\Ok$-adic algebra. Therefore, choosing generators and relations for $A$
we can find an (underived) pullback diagram of the form
	\begin{equation}
	\begin{tikzcd} \label{und:pull}
	 	\Spf(A) \ar{r} \ar{d} & \bfA^m \ar{d} \\
		\Spf ( \Ok ) \ar{r} & \bfA^n
	\end{tikzcd},
	\end{equation}
of ordinary $\Ok$-adic formal schemes. Let $Z$ denote the (derived) pullback of the diagram \eqref{und:pull}, computed in the \infcat
$\dfSch$. Its existence is guaranteed by \cite[Proposition 8.1.6.1]{lurie2016spectral}. It follows that $\trun_{\leq 0}
(Z) \simeq \Spf (A)$. Notice that both $\bfA^m$, $\bfA^n$ and $\Spf (\Ok)$ are objects of the pregeometry $\Tad$. Moreover, $\rigg$ is induced
by the usual generic fiber functor. Thus it follows that
	\[
		\Spf (\Ok)^\rig \simeq    \Sp k, \quad
		\left( \bfA^{m} \right)^\rig \simeq   \anA^m, \quad
		\left( \bfA^n \right)^\rig \simeq  \anA^n.
	\]
As $\rigg$ is a right adjoint, it commutes with pullback diagrams. We thus have a pullback diagram in the \infcat $\anTop$
	\[
	\begin{tikzcd}
		Z^\rig \ar{r} \ar{d} & \anA^m \ar{d} \\
		\Sp (k) \ar{r} & \anA^n 
	\end{tikzcd}.
	\]
\cref{tr} implies that $\trun_{\leq 0} \left( Z^\rig \right) \simeq \trun_{\leq 0 }(Z)^\rig$. As $\trun_{\leq 0} (Z) \simeq \Spf(A)$ we
deduce that $\left( \Spf(A) \right)^\rig$ is equivalent to the (underived) pullback diagram
	\[
	\begin{tikzcd}
		\Spf (A)^\rig \ar{r} \ar{d} & \anA^m \ar{d} \\
		\Sp(k) \ar{r} & \anA^n
	\end{tikzcd}
	\]
computed in the category of rigid $k$-analytic spaces. This agrees with the usual rigidification construction applied to $\Spf
A$.
\end{proof}

\begin{lemma} \label{cl:lem}
Let $f \colon \sfZ \to \sfX$ be a closed immersion of derived $\Ok$-adic Deligne-Mumford stacks topologically almost of finite
presentation. Then $f^\rig$ is a closed immersion in the \infcat $\dAn$.
\end{lemma}

\begin{proof}
It suffices to show that the truncation $\trun_{\leq 0} (f^\rig) \colon \trun_{\leq 0} (\sfZ^\rig) \to \trun_{\leq 0} (\sfX^\rig)$ is
a closed immersion. The latter assertion is a consequence of \cref{rig_cl}.
\end{proof}

\begin{proposition} \label{prop:rig_sends_formal_DMs_to_Ansp}
Let $\sfX$ be a topologically almost of finite presentation derived $\Ok$-derived Deligne-Mumford stack. Then $\sfX^\rig$
is a derived $k$-analytic space.
\end{proposition}

\begin{proof}
Our proof is inspired by \cite[Proposition 3.7]{porta2017representability}. The question is \'etale local, by
\cite[Lemma 2.1.3]{lurie2011dag}. We can thus reduce ourselves to the case $\sfX = \Spf (A)$, where $A \in \adCAlg$
is a $(t)$-complete topologically almost of finite presentation derived $\Ok$-adic algebra. We wish to prove that $\Spf (A)^\rig$ is
a derived $k$-affinoid space. Let $\cC$ denote the full subcategory of $\dfDM_{\Ok}$ spanned by those affine derived
$\Ok$-adic Deligne-Mumford stacks $\Spf(A)$ such that $\Spf(A)^\rig $ is equivalent to a derived $k$-affinoid space.
We have:
\begin{enumerate}
	\item The \infcat $\cC$ contains the essential image of $\Tad$, cf. \cite[Proposition 2.3.18]{lurie2011dag}.
	\item $\cC$ is closed under pullbacks along closed immersions. Indeed, let
		\begin{equation} \label{pull:unr}
		\begin{tikzcd}
			\sfW \ar{r} \ar{d} & \sfZ \ar{d} \\
			\sfY \ar{r}{f} & \sfX,
		\end{tikzcd}
		\end{equation}
	be a pullback diagram in the \infcat $\dfDM_{\Ok}$, where $\sfX, \ \sfY $ and $\sfZ \in \cC$. Assume further that
	$f \colon \sfY \to \sfX $ is a closed immersion. By
	unramifiedness of the pregeometry $\Tad$, \cref{{unr:ad}}, the diagram \eqref{pull:unr} is also a pullback diagram
	in the \infcat $\adTop$. As $\rigg$ is a right adjoint the diagram
		\begin{equation} \label{pull:an}
		\begin{tikzcd}
			\sfW^\rig \ar{d} \ar{r} & \sfZ^\rig \ar{d} \\
			\sfY^\rig \ar{r} & \sfX^\rig
		\end{tikzcd}
		\end{equation}
	is a pullback diagram in the \infcat $\anTop$. The \infcat $\dAn$ is closed under pullbacks along closed immersions
	thanks to \cite[Proposition 6.2]{porta2016derived}.
	\cref{cl:lem} then implies that the diagram \eqref{pull:an} is a pullback square in the \infcat $\dAn$. Thus
	$\sfW \in \cC$, as desired.
	\item The \infcat $\cC$ is closed under finite limits. It suffices to prove that $\cC$ is closed under finite products and
	pullbacks. \cite[Lemma 6.4]{porta2016derived} implies that $\cC$ is closed under finite products.
	General pullback diagrams can be constructed as pullbacks along
	along closed immersions as in the proof of \cite[Theorem 6.5]{porta2016derived}. Thanks to \cref{rig_cl},
	$\rigg$ commutes with finite products of ordinary
	formal schemes and preserves closed immersions by \cref{cl:lem}, the assertion follows.
	\item $\cC$ is closed under retracts: let $X \in \cC$ and let
		\[
		\begin{tikzcd}
		\sfY \ar{r}{j} & \sfX \ar{r}{p} & Y,
		\end{tikzcd}
		\]
	be a retract diagram in the \infcat $\dfDM_{\Ok}$. Assume further that $\sfY$ is affine. By assumption, $\sfX^\rig \in
	\dAn$ and $\trun_{\leq 0} (\sfY)^\rig \in \dAn$, see \cref{rig_cl}. It suffices to prove that for each $i > 0$, the
	homotopy sheaf $\pi_i \big( \cO_{\sfY}^\rig \big) $ is a coherent sheaf over $\pi_0 \big( \cO_{\sfY}^\rig \big)$.
	The former is a retract of $\pi_i \big( \cO_{\sfX}^\rig \big)$, which is a coherent sheaf over $\pi_0\big(\cO_{\sfX}^\rig
	\big)$. In this way, it follows that $\pi_i \big( \cO_{\sfY}^\rig \big)$ is coherent over $\pi_0 \big( \cO_{\sfX}^\rig \big)$.
	As $\pi_0 \big( \cO_{\sfY}^\rig \big)$ is a retract of $\pi_0 \big( \cO_{\sfX}^\rig \big)$, we deduce that
	$\pi_i \big( \cO_{\sfY}^\rig \big)$ is coherent over $\pi_0 \big( \cO_{\sfY}^\rig \big)$, as desired.
\end{enumerate}
Suppose now that we are given an affine object $\sfX  \in \dfDM^\taft_{\Ok}$. Write $\sfX \simeq \Spf (A)$. We wish to prove that $\sfX \in \cC$. \cref{rig_cl}
guarantees that $\trun_{\leq 0 } \left( \sfX^\rig \right)$ is a $k$-analytic space. We are thus reduced to show
that $\pi_i \big( \cO^\rig_{\sfX} \big)$ is a coherent sheaf over $\pi_0 \big( \cO^\rig_{\sfX} \big)$.
For every $n \geq 0$, the algebra $\tau_{\leq n}(A)$ is a compact object in the \infcat $
\big( \adCAlg \big)^{\leq n}$,
of $n$-truncated derived adic $\Ok$-algebras. We can thus find a \emph{finite} diagram of free simplicial
$\Ok$-algebras
	\[
		g \colon I \to  \CAlg_{\Ok},
	\]
such that $\tau_{\leq n} (A)$ is a retract of $\tau_{\leq n} (B)$, where
	\[
		B \coloneqq \colim_I \left( g \right)^\wedge_t \in \adCAlg.
	\]
We have denoted $\left( g\right)^\wedge_t$ the $(t)$-completion of the diagram $g \colon I \to \CAlg_{\Ok}$. As
the
$(t)$-completion functor commutes with finite colimits it follows that 
	\[
		B \simeq B^\wedge_t,
	\]
and in particular $B$ is $(t)$-complete. As $\cC$ is closed under finite limits and it contains all
objects in the pregeometry $\Tad$, we conclude that
$\Spf ( B) \in \cC$. In particular $\Spf \left( \tau_{\leq n} B \right) \in \cC$. As $\cC$ is moreover closed under
retracts, it follows that $\Spf \left( \tau_{\leq n} A \right) \in \cC$ as well. It follows, that for each $0 \leq i \leq
n$, $\pi_n \big(\cO_{\sfX}^\rig \big)$ is coherent over $\pi_0 \big( \cO_{\sfX}^\rig \big)$. Repeating the
argument, for every $n \geq 0$, we conclude.
\end{proof}

As a consequence of the previous results we obtain the following statement:

\begin{corollary} \label{cor:derived_rig_functor}
	The rigidification functor $\rigg \colon \adTop \to \anTop$ restricts to a well defined functor $\rigg \colon \dfSch^\taft \to \dAn$, which agrees with the classical Raynaud's rigidification functor. Moreover, the functor
	$\rigg \colon \dfSch^\taft \to \dAn$ commutes with finite limits.
\end{corollary}

\begin{proof}
	The first assertion is an immediate consequence of \cref{prop:rig_sends_formal_DMs_to_Ansp}. The second claim follows from the arguments provided in (ii) and (iii) in the proof of \cref{prop:rig_sends_formal_DMs_to_Ansp}.
\end{proof}

\subsection{Rigidification of structures}

\begin{construction} \label{const2}
Let $X  = (\cX, \cO) \in \anTop$ be a $\Tan$-structured $\infty$-topos. Suppose further that there exists
$\sfX = (\mathfrak X,\cO_{\sfX}) \in \adTop$ such that we have an equivalence $\sfX^\rig \simeq X$ in $\anTop$. Precomposition along
the transformation of pregeometries
	\[	
		\rigg \colon \Tad \to \Tan
	\]
induces a functor at the level of \infcats of structures
	\[
		\functor^+ \colon \AnRing(\cX)_{/ \cO} \to \fCAlg( \cX)_{/ \cO^+}
	\]
given on objects by the formula
	\[
		\cA \in \AnRing(\cX)_{/ \cO}  \mapsto \cA^+ \coloneqq \cA \circ \rigg \in \fCAlg (\cX)_{/ \cO^+}.
	\]
The functor of presentable \infcats $\functor^+ \colon \AnRing(\cX)_{/ \cO} \to \fCAlg( \cX)_{/ \cO^+}$ preserves limits and
filtered colimits. Thanks to the Adjoint functor theorem it follows that there exists a left adjoint
	\begin{equation} \label{com:rig}
		\functor^{\rig, \circ} \colon \fCAlg(\cX)_{/ \cO^+} \to \AnRing(\cX)_{/ \cO}.
	\end{equation}
The counit of the adjunction $\left( \functor^+, \ \rigg \right) \colon \adTop \to \anTop$ produces a well defined, up to
contractible indeterminacy, morphism
	\begin{equation} \label{coun}
		f \colon X^+ = (\cX, \cO^+ ) \to (\mathfrak X,\cO_{\sfX}) = \sfX,
	\end{equation}
in the \infcat $\adTop$. Let $(f^{-1}, f_*) \colon \cX \rightleftarrows \fX $ denote the underlying geometric morphism associated to
$f$. Then $f^{-1} \colon \fX \to \cX$ induces a well defined functor
	\begin{equation} \label{com:f}
		f^{-1} \colon \fCAlg(\fX)_{/ \cO_{\sfX}} \to \fCAlg(\cX)_{/ f^{-1} \cO_{\sfX}}.
	\end{equation}
Moreover, the morphism \eqref{coun} induces a morphism at the level of structures
	\[
		\theta \colon f^{-1} \cO_{\sfX} \to \cO^+.
	\]
The latter induces a well defined functor at the level of \infcats of structures
	\begin{equation} \label{com:theta}
		\theta \colon	\fCAlg(\cX)_{/  f^{-1} \cO_{\sfX}} \to \fCAlg(\cX)_{/ \cO^+},
	\end{equation}
which is given on objects by the formula
	\[
		\left( \cA \to f^{-1} \cO_{\sfX} \right) \in \fCAlg(\cX)_{/ f^{-1} \cO_{\sfX}} \mapsto 
		\left( \cA \to \cO^+ \right) \in \fCAlg(\cX)_{/ \cO^+}.
	\]
Therefore the composite $\rigg \coloneqq \functor^{\rig, \circ} \circ \theta \circ f^{-1} $ induces a functor
	\[
		\rigg \colon \fCAlg(\fX)_{/ \cO_{\sfX}} \to  \AnRing(\cX)_{/ \cO},
	\]
which we refer to as \emph{the rigidification functor at the level of structures}.
\end{construction}

\begin{definition}
Let $X \in \dAn$ and $\sfX \in \dfDM$. We say that $\sfX$ is a \emph{formal model for $X$} if $\sfX^\rig \simeq X$, in $\dAn$.
\end{definition}

\begin{remark} Let $X = (\cX, \cO_X)$ be a derived $k$-analytic space and $\sfX = (\mathfrak X, \cO_\sfX) \in 
\dfDM_{\Ok}^\taft$ a formal model for $X$.
Thanks to \cref{rig_cl},
the geometric morphism of underlying $\infty$-topoi $f \colon \cX \to \mathfrak X$ agrees with the classical \emph{specialization morphism},
in the $\infty$-categorical setting.
\end{remark}

\begin{notation}
We will denote the geometric morphism introduced in \cref{const2} $(f^{-1}, f_*) \colon \cX \to \fX$ by
$\spe = (\spe^{-1}, \spe_*)$.
\end{notation}

\begin{construction} \label{const3}
Notations as in \cref{const2}. Consider the following square of pregeometries
	\begin{equation} \label{all}
	\begin{tikzcd}
		\Tdisc(\Ok) \ar{r}{- \otimes_{\Ok} k} \ar{d}[swap]{\functor^\wedge_t} & \Tdisc(k) \ar{d}{\functor^\an} \\
		\Tad \ar{r}[swap]{\rigg} & \Tan 
	\end{tikzcd}
	\end{equation}
Notice that \eqref{all} is not commutative. The lower composite sends
	\[
	 	\bA^1_{\Ok} \in \Tdisc \mapsto \anA^1 \in \Tan
	\]
whereas the top composite sends
	\[
		\bA^1_{\Ok} \in \Tdisc \mapsto \anB^1 \in \Tan.
	\]
Let $\cA \in \fCAlg(\fX)_{/ \cO_{\sfX}}$. The
counit of the adjunction $(\functor^+, \rigg)$ induces a natural morphism at the level of $\Tad$-structures on $\cX$
	\[
		\theta_{\cA} \colon \spe^{-1} \cA \to  \cA^{\rig, + } \coloneqq \big( \cA^\rig \big)^+.
	\]
Applying the underlying algebra functor $\functor^\alg \colon \fCAlg(\cX)_{/ \cO^+} \to 
\CAlg_{\Ok}(\cX)_{/ \cO^{+, \alg}}$ to the morphism $\theta_{\cA}$ we obtain a morphism
	\begin{equation} \label{uau}
		\theta^\alg_{\cA} \colon ( \spe^{-1} \cA)^{\alg} \to \cA^{\rig, +, \alg},
	\end{equation}
in the \infcat $\CAlg_{\Ok}(\cX)_{/ \cO^{+, \alg}}$. As $\cA^{\rig, + , \alg}$ is a
$k$-linear object, we obtain by adjunction a morphism
	\[
		\overline{\theta}^\alg_{\cA} \colon (\spe^{-1} \cA)^{\alg } \otimes_{\Ok} k \to \cA^{\rig, +, \alg} \coloneqq \big( \cA^{\rig, +} \big)^\alg,
	\]
in the \infcat $\CAlg_k(\cX)_{\cO^{+, \alg}}$. We can identify $\cA^{\rig, +, \alg} \simeq \cA( \anB^1)$. There is a natural
inclusion of $k$-analytic spaces $\anB^1 \to \anA^1$. We obtain thus a canonical morphism
	\begin{equation} \label{wow}
		\cA(\anB^1) \to \cA(\anA^1),
	\end{equation}
in the \infcat $\CAlg_k(\cX)_{/ \cO(\anA^1)}.$ Composing both \eqref{uau} with \eqref{wow} we obtain a natural morphism
	\begin{equation} \label{atlast}
		\overline{\theta}_{\cA} \colon (\spe^{-1} \cA)^{\alg} \otimes_{\Ok} k \to \cA^{\rig} (\anA^{1}),
	\end{equation}
in the \infcat $\CAlg(\cX)_{/ \cO(\anA^1)}$. We will take as (a probably confusing) convention to denote
precomposition with $\functor^{\an}$ in \eqref{all} by 
	\[
		\functor^\alg \colon \AnRing(\cX)_{/ \cO} \to 
		\CAlg_{k}(\cX)_{/ \cO^\alg}.
	\]
In this case, we might as well write \eqref{atlast} as
	\[
		\overline{\theta}_{\cA} \colon (\spe^{-1} \cA)^{\alg} \otimes_{\Ok} k \to  \cA^{\rig, \alg} \coloneqq \big( \cA^{\rig} \big)^\alg.
	\]
\end{construction}

\begin{proposition} \label{loc_t}
Let $X = (\cX, \cO_X) \in \dAn$ and $\sfX = (\mathfrak X,\cO_{\sfX}) \in \dfDM_{\Ok}$ be a formal model for $X$. Then for every $\cA \in \fCAlg(\fX)_{/ \cO_\sfX}$ the natural morphism
	\[
		\overline{\theta}_{\cA} \colon (\spe^{-1} \cA)^{\alg} \otimes_{\Ok} k \to \cA^{\rig, \alg},
	\]
constructed above, is an equivalence in the \infcat $\CAlg_k(\cX)_{/ \cO^{\alg}}$.	
\end{proposition}

\begin{proof}
Both the underlying $\infty$-topoi of $X$ and $\sfX$ have enough points, as these are hypercomplete and $1$-localic.
Therefore, thanks to \cite[Theorem 1.12]{porta_square_zero}, we are reduced to check the statement of the proposition on stalks, (notice that
given a geometric point $x_* \colon \cS \to \cX$ the composite $\spe_* \circ x_* \colon \cS \to \fX$ is also a geometric point).

By doing so, we might assume from
the start that $\cX = \cS = \fX$. Both composites $\functor^\alg \circ \rigg $ and $\left(\functor^\alg \circ \spe^{-1}\right)
\otimes_{\Ok} k$ commute with sifted colimits. The proof of
\cref{t_comp} implies that the \infcat $\fCAlg(\cS)_{/ \cO_{\sfX}}$ is generated under sifted colimits
by the family $\{ \Psi \left( \Ok [ T_1, \dots , T_m ]\right) \}_m$, where the $T_i$'s sit in homological degree $0$.
It thus suffices to show that 
	\[
		\overline{\theta}_{\cA} \colon (\spe^{-1} \cA)^{\alg} \otimes_{\Ok} k \to \cA^{\rig, \alg},
	\]	
is an equivalence whenever $\cA \simeq \Psi \big( \Ok [ T_1, \dots, T_m ] \big)$. In this case, we have natural equivalences
	\[
		 (\spe^{-1} \Psi (\Ok [T_1, \dots , T_m ])^{\alg} \otimes_{\Ok} k \simeq k \langle T_1, \dots, T_m \rangle.
	\]
Since $\Psi \left(  \Ok [T_1, \dots , T_m ] \right)$ can be identified with (a germ) of $\bfA^m \in \adTop$, it follows that
	\[
		\Psi \left( \Ok [T_1, \dots , T_m ] \right)^{\rig, \alg} \simeq k \langle T_1, \dots, T_m \rangle,
	\]
in the \infcat $\CAlg_k (\cX)_{/ \cO^\alg}$. The result now follows.
\end{proof}

\subsection{Rigidification of modules} We start by recalling \cite[Definition 4.3]{porta2017representability}:

\begin{definition}
Let $X = (\cX, \cO_X) \in \dAn$ be a derived $k$-analytic space. The \infcat of \emph{$\cO_X$-modules on $X$} is defined as
	\[
		\Mod_{\cO_X} \coloneqq \Sp \left( \Ab \left( \AnRing(\cX)_{/ \cO_X} \right) \right).
	\]
Similarly, the \infcat $\Mod_{\cO_X^\alg}$ is defined as the \infcat of $\cO_X$-module objects in $\Shv_{\cD(\Ab)}(\cX)$.
\end{definition}

One has the following result:

\begin{proposition}{\emph{\cite[Theorem 4.5]{porta2017representability}}}
There exists a canonical equivalence of \infcats
	\[
		\Mod_{\cO_X} \simeq \Mod_{\cO_X^\alg}.
	\] 
\end{proposition}

\begin{remark}
	In particular, the \infcat $\Mod_{\cO_X}$ inherits a natural t-structure induced from the t-structure on $\Mod_{\cO_X^\alg}$, where the latter is defined in \cite[Proposition 1.7]{Lurie_Spectral_Schemes}.
\end{remark}

\begin{lemma}
Let $X = (\cX, \cO_X) \in \anTop$ and $\sfX = (\mathfrak X,\cO_\sfX) \in \adTop$ such that $\sfX^\rig \simeq X$, in $\anTop$.
The rigidification functor $\rigg \colon \fCAlg(\fX)_{/ \cO_\sfX} \to \AnRing(\cX)_{/ \cO_X}$ induces a well defined functor
	\[
		\rigg \colon \Mod_{\cO_\sfX} \to \Mod_{\cO_X}.
	\]
We shall refer it as the \emph{rigidification of modules functor}.
\end{lemma}

\begin{proof}
It suffices to show that the functor $\rigg \colon \fCAlg(\fX)_{/ \cO_\sfX} \to \AnRing(\cX)_{/ \cO_X}$ commutes with finite limits.
Thanks to \cref{loc_t} the composite functor $\functor^\alg \circ \rigg$ agrees with localization at $(t)$ and therefore it commutes
with finite limits. As $\functor^\alg$ is a conservative right adjoint it follows that \[\rigg \colon  \fCAlg(\fX)_{/ \cO_\sfX} \to
\AnRing(\cX)_{/ \cO_X}\] commutes with finite limits as well, and the proof is finished.
\end{proof}

By construction, we have a natural projection functor \[\Omega^\infty_\an \colon \Mod_{\cO_X} \to \AnRing(\cX)_{/ \cO_X}.\] 

\begin{definition}Let $M \in \Mod_{\cO_X}$, we shall
denote $\cO_X \oplus M \coloneqq \Omega^\infty_\an(M)$ and refer to it as the \emph{analytic split square zero extension
of $\cO_X$ by $M$}.
\end{definition}

Following the discussion in \cite[\S 5.1]{porta2017representability} prior to \cite[Definition 5.4]{porta2017representability}
we conclude that the functor $\Omega^\infty_\an$ admits a left adjoint $\Sigma^\infty_\an \colon \AnRing(\cX)_{/ \cO_X} \to \Mod_{\cO_X}$.

\begin{definition}
Suppose we are given $\cA \in \AnRing(\cX)_{/ \cO_X}$ and consider the \infcat $\AnRing(\cX)_{\cA / / \cO_X}$. We can
consider the \emph{analytic $\cA$-linear derivations functor} $\Der^\an_{\cA} \left( \cO_X, - \right) \colon \Mod_{\cO_X} \to \cS$ given
on objects by the formula
	\[
		M \in \Mod_{\cO_X} \mapsto \Map_{\AnRing(\cX)_{\cA / / \cO_X} } \left( \cO_X, \cO_X \oplus M \right).
	\]
\end{definition}

Thanks to \cite[Proposition 5.18.]{porta2017representability}
such functor is corepresentable by the \emph{analytic cotangent complex relative to} $\cA \to \cO_X$, which we denote
by $\bL^\an_{\cO_X / \cA}$. Explicitly, one has a natural equivalence of mapping spaces	
	\[
		\Map_{\Mod_{\cO_X}} \left( \bL^\an_{\cO_X / \cA} , M \right) \simeq \Der^\an_{\cA} \left( \cO_X, M \right).
	\]
We can explicitly describe $\bL^\an_{\cO_X / \cA} \simeq \Sigma^\infty_\an \left( \cO_X \otimes_{\cA} \cO_X \right) \in \Mod_{\cO_X}$.

\begin{lemma} \label{comm:ad:an}
Let $X  = (\cX, \cO_X) \in \anTop$ and $\sfX = (\cZ , \cO_\sfX) \in \adTop$ such that $\sfX^\rig \simeq X $, in \infcat $\anTop$.
Then the diagram
	\[
	\begin{tikzcd}
		\Mod_{\cO_\sfX} \ar{r}{\rigg}  & \Mod_{\cO_X} \\
		\fCAlg(\fX)_{ / \cO_\sfX} \ar{u}{\Sigma^\infty_\ad} \ar{r}{\rigg} & \AnRing(\cX)_{/ \cO_X} \ar{u}{\Sigma^\infty_\an}
	\end{tikzcd}
	\]
is commutative up to coherent homotopy.
\end{lemma}

\begin{proof}
It suffices to prove that the corresponding diagram of right adjoints
	\[
	\begin{tikzcd}
		\Mod_{\cO_\sfX}  \ar{d}{\Omega^\infty_\ad}  & \Mod_{\cO_X} \ar{l}[swap]{\functor^+} \ar{d}{\Omega^\infty_\ad} \\
		\fCAlg(\fX)_{ / \cO_\sfX}   & \AnRing(\cX)_{/ \cO_X} \ar{l}[swap]{\functor^+}
	\end{tikzcd}
	\]
is commutative. But this follows by an immediate verification as done in \cite[Lemma 5.15]{porta2017representability}. The result follows.
\end{proof}
	
\begin{corollary} \label{cor:adic_cotangent_complex_is_compatible_with_analytic_via_rig}
We have a natural equivalence
	\[
		\big( \bL^\ad_{\cO_X / \cA} \big)^\rig \simeq \bL^\an_{\cO_\sfX^\rig / \cA^\rig}
	\]
in the \infcat $\Mod_{\cO_X}$.
\end{corollary}
	
\begin{proof}
It is an immediate consequence of \cref{comm:ad:an} above.
\end{proof}

\begin{definition}
Let $M \in \Coh(X)$. We say that $M$ admits a formal model if there exists an $\cO_\sfX$-module $\mathfrak M \in \Coh(\cO_\sfX)$ such
that 
	\[
		\mathfrak M^\rig \simeq M \in \Mod_{\cO_X}.
	\]
\end{definition}
	
\begin{proposition}
Let $\sfX = (\mathfrak X,\cO_\sfX) \in \dfDM_{\Ok}$ and $X = (\cX, \cO_X) \coloneqq \sfX^\rig $, denote its rigidification. Then the functor
$\rigg \colon \Mod_{\cO_\sfX} \to \Mod_{\cO}$ is $t$-exact.
\end{proposition}
	
\begin{proof}
The statement follows readily from \cref{{t_ex:ver}} combined with \cite[Corollary 2.9]{hennion2016formal}.
\end{proof}
	
\subsection{Main results} In this \S, we state our two main results. The first one
concerns the existence of formal models for quasi-paracompact and quasi-separated derived $k$-analytic spaces.
The second is a direct generalization of Raynaud's localization theorem.

\begin{definition}
Let $A \in \adCAlg$. We say that $A$ is an \emph{derived admissible
$\Ok$-adic algebra} if $A$ is $(t)$-complete and topologically almost of finite presentation. We further required that,
for every $i \geq 0$, the homotopy sheaf $\pi_i(A)$ is $(t)$-torsion free.
We denote $\admCAlg$ the full subcategory of $\adCAlg$ spanned by
admissible adic derived $\Ok$-algebras.
\end{definition}

\begin{definition}
Let $\sfX \in \dfDM_{\Ok}^\taft$. We say that $\sfX$ is a \emph{derived admissible $\Ok$-adic
Deligne-Mumford stack} (resp. \emph{derived admissible $\Ok$-adic
scheme}) if there exists 
	\[
		\coprod_i \Spf (A_i) \to \sfX,
	\]
such that, for each index $i$, $A_i \in \admCAlg$. We denote by $\dfDM_{\Ok}^\adm$ (resp.,
$\dfSch^\adm$) the \infcat of derived admissible $\Ok$-adic
Deligne-Mumford stacks (resp. derived admissible $\Ok$-adic
schemes).
\end{definition}

\begin{definition}
Let $X = (\cX, \cO)$ be a derived $k$-analytic space. We say that $X$ is \emph{quasi-paracompact and quasi-separated} if the $0$-th truncation $\trun_{\leq 0}( X)$ is equivalent to a quasi-paracompact and quasi-separated ordinary $k$-analytic space.
\end{definition}

Thanks to \cite[Theorem 3, page 204]{bosch2005lectures}, it follows that any (ordinary)
quasi-paracompact and quasi-separated $k$-analytic space, $X$, admits a classical formal model. We generalize this
result to the derived setting:

\begin{theorem} \label{main1}
Let $X = (\cX, \cO)$ be a quasi-paracompact and quasi-separated derived $k$-analytic space. Then $X$ admits a derived
formal model $\sfX = (\mathfrak X,\cO_{\sfX}) \in \dfSch$.
\end{theorem}
	
\begin{proof}
Let $X_0 \coloneqq \trun_{\leq 0} (X)$ denote the $0$-truncation of $X$. Thanks to
\cite[Theorem 3, page 204]{bosch2005lectures}, it follows that $X_0$ admits an admissible formal
model $\sfX_0 \in \fSch$. We will construct inductively
a sequence of derived admissible $\Ok$-adic schemes
	\[
		\sfX_0 \to \sfX_1 \to \sfX_2 \to \dots,
	\]
such that we have equivalences
	\[
		(\sfX_n)^\rig \simeq \trun_{\leq n} (X),
	\]
for each $n \geq 0$.The case $n=0$ being already dealt, we proceed with the inductive
step. Suppose that $\sfX_{n} = (\mathfrak X,\cO_{\sfX_n})
$ has already been constructed, for $n \geq 0$. As $X$ is a
derived $k$-analytic space, for each $n \geq 0$, the homotopy sheaf $\pi_n(\cO_X)$ is
a coherent module over $\pi_0(\cO_X)$. Thanks to
\cite[Corollary 5.42]{porta2017representability} there exists an analytic derivation
$d \colon \bL^\an_{\trun_{\leq n} X} \to \pi_{n+1} (\cO_X)[ n+ 2]$ together with a pullback
diagram
	\begin{equation}
	\begin{tikzcd}
		\tau_{\leq n+1} \cO_X \ar{r} \ar{d} & \tau_{\leq n} \cO_X \ar{d}{d} \\
		\tau_{\leq n} \cO_X \ar{r}{d_0} & \tau_{\leq n} \cO_X \oplus \pi_{n+1} ( \cO_X)[n+2],
	\end{tikzcd}
	\end{equation}
in the \infcat $\AnRing(\cX)_{/ \tau_{\leq n} \cO_X}$. Here $d_0$ denotes the trivial
analytic derivation. \cref{O'Neill} and its proof imply that we can find a formal model
for $d $ in the stable \infcat $\Coh (\sfX_{n})$, namely
	\[
		\delta \colon \bL^\ad_{\sfX_n} \to M_{n+1} [n+2],
	\]
where $M_{n+1} \in \Coh(\sfX_n)$ is a formal model for $\pi_{n+1}(\cO_X)$. We can assume further that $M_{n+1} \in \Coh(\sfX_n)^\heartsuit$ does not admit non-trivial $(t)$-torsion.
We define $\cO_{n+1} \in \fCAlg(\fX)_{/ \cO_{\sfX_n}}$ as the pullback of the diagram
	\begin{equation} \label{form:ext}
	\begin{tikzcd}
		\cO_{n+1} \ar{r} \ar{d} & \cO_{\sfX_n} \ar{d}{\delta} \\
		\cO_{\sfX_n} \ar{r}{d_0} & \cO_{\sfX_n} \oplus M_{n+1} [n+2]
	\end{tikzcd}
	\end{equation}
in $\fCAlg(\fX)_{/ \cO_{\sfX_n}}$. Define $\sfX_{n+1} \coloneqq (\mathfrak X,\cO_{n+1} )$.
By construction, $\sfX_{n+1}$ is a derived admissible $\Ok$-adic Deligne-Mumford stack.
Both $\sfX^\rig$ and $\trun_{\leq n+1} (X)$ have equivalent underlying $\infty$-topoi.
Thus, we have a rigidification functor
	\[
		\rigg \colon \fCAlg(\fX)_{/ \cO_{n+1}} \to \AnRing(\cX)_{/ \cO_X},
	\]
which
commutes with finite limits. Thus the diagram \eqref{form:ext} remains a pullback
diagram after rigidification. For this reason, we obtain a canonical morphism
	\[
		\alpha_{n+1} \colon (\cO_{n+1} )^\rig \to \tau_{\leq n+1} \cO_X
	\]
in the \infcat $\AnRing(\cX)_{/ \tau_{\leq n+1} \cO_X}$. The morphism of $\Tan$-structures produces a morphism
	\[
		\theta_{n+1} \colon \trun_{\leq n+1} (X) \to \sfX^\rig_{n+1}.
	\]
We claim that $\theta_{n+1}$ is an equivalence in the \infcat $\anTop$. Since the underlying $\infty$-topoi of both $\sfX_{n+1}$ and of $\sfX_0$ are equivalent,
we are reduced to
show that $\alpha_{n+1}$ is an equivalence of structures.
Thanks to \cref{loc_t} we have an
equivalence
	\[
		\left( \cO_{ n+1}^\rig \right)^\alg \simeq \left(
		\spe^{-1} \cO_{ n+1} \right)^\alg \otimes_{\Ok} k.
	\]
By the inductive hypothesis together with the pullback diagrams \eqref{form:ext}
it follows that the natural morphism
	\[
		(\spe^{-1} \cO_{ n+1})^\alg \otimes_{\Ok} k \to (\tau_{ \leq n+1} \cO_X )^\alg,
	\]
is an equivalence. By conservativity of $\functor^\alg$ it follows that
	\[
		\alpha_{n+1} \colon \cO_{ n+1}^\rig \to \tau_{\leq n+1} \cO_X,
	\]
is also an equivalence the \infcat $\AnRing(\cX)_{/ \tau_{\leq n+1} \cO_X}$.
We conclude that 
	\[
		\theta_{n+1} \colon \sfX^\rig_{n+1} \simeq \trun_{\leq n+1} X,
	\]
is an equivalence in $\anTop$. Define
	\[
		\sfX \coloneqq \colim_{n \geq 0} \sfX_n.
	\]
We claim that $\sfX$ is again an admissible derived $\Ok$-adic Deligne-Mumford
stack: the question being local on $\sfX$ we reduce ourselves to the case $\sfX = \Spf A$
and $\sfX_n \simeq \Spf A_n$, for suitable $A, \ A_n \in \admCAlg$.
By construction, $\tau_{\leq n-1} (A_n) \simeq A_{n-1}$, for each $n \geq 1$.
We have moreover an identification
	\[
		\sfX \simeq \Spf \left( \lim_{n \geq 0} A_n \right).
	\]
As $A_n$ is admissible, for every $n \ge 0$, we conclude that $\lim_{n \geq 0}A_n$ is again admissible.
This concludes the proof that $\sfX$ is a derived admissible $\Ok$-adic Deligne-Mumford
stack.

It only remains to show that $\sfX^\rig \simeq X$.
We have a sequence of equivalences	
	\[
		\trun_{\leq n} (\sfX^\rig) \simeq (\trun_{\leq n} \sfX)^\rig \simeq
		\sfX_n^\rig \simeq \trun_{\leq n } (X),
	\]
which is a consequence of convergence for derived $k$-analytic stacks, cf. \cite[\S 7]{porta2017representability}. Assembling these equivalences together, we produce a map
	\[
		f \colon \sfX^\rig \to X,
	\]
in the \infcat $\dAn$. The underlying morphism of $\infty$-topoi is an equivalence, since the underlying $\infty$-topoi are equivalent in every step of the inductive argument.
Furthermore, $f$ induces equivalences, for each $i \geq 0$,
	\[
		\pi_i (\cO_\sfX^\rig) \simeq \pi_i(\cO_X),
	\]
where $\cO_\sfX \coloneqq \lim_{n \geq 0} \cO_{\sfX_n}$. By hypercompletion of the $\infty$-topos
$\cX$, it follows that
	\[
		\cO^\rig_\sfX \simeq \cO_X.
	\]
This shows that $f$ is an equivalence, and the proof is complete.
\end{proof}

We now deal with our main result. We start with a useful lemma:

\begin{lemma} \label{lem_for_main}
Let $F \colon \cC \to \cD $ be a functor between \infcats. Suppose that for any $D
\in \cD$ the following assertions are satisfied:
\begin{enumerate}
	\item The \infcat $\cC_{/ D} \coloneqq \cC \times_{\cD} \cD_{/ D}$ is contractible;
	\item let $\cC_{/ D}'$ denote the full subcategory of $\cC_{/ D}$ spanned
	by those objects $(C, \psi \colon F(C) \to D)$ such that $\psi$ is an equivalence
	in $\cD$. Suppose further that $\cC_{/ D}'$ is non-empty and the
	inclusion $\cC_{/ D}' \to \cC_{/ D}$ is cofinal.
\end{enumerate}
Then $F \colon \cC \to \cD$ induces an equivalence of \infcats $\cC[S^{-1}] \to \cD$,
where $S$ denotes the class of morphisms $f \in \cC^{\Delta^1}$ such that $F(f)$
is an equivalence.
\end{lemma}
	
\begin{proof}
Let $\cE$ be an \infcat. We have to prove that precomposition along $F$ induces
a fully faithful embedding of \infcats
	\[
		F^* \colon \Fun \left( \cD, \cE \right) \to \Fun \left( \cC, \cE \right),
	\]
whose essential image consists of those functors $G \colon \cC \to \cE$ which
send morphisms in $S$ to equivalences in $\cD$. Given any functor $G \colon
\cD \to \cE$, the composite $G \circ F \colon \cC \to \cE$ sends each morphism
in $S$ to an equivalence $\cE$, as $F$ does. Both hypothesis (i) and (ii) combined with the colimit
formula for left Kan extensions imply that for every 
	\[
		G \colon \cC \to \cE,
	\]
such that every
morphism in $S$ is sent to an equivalence, the left Kan extension 
	\[
		F_! (G) \in 
		\Fun \left(\cD, \cE \right),
	\]
exists. Furthermore, we have natural equivalences
	\[
		F_! \circ F^* \simeq \mathrm{id}, \quad \textrm{and }
		F^* \circ F_! \simeq \mathrm{id}.
	\]
The
result now follows from the fact that $F_!$ is an inverse to $F^*$, when restricted
to the full subcategory of $\Fun \left( \cC, \cE \right)$ spanned by those functors
sending every morphism in $S$ to equivalences in $\cE$.
\end{proof}

\begin{remark}
\cref{lem_for_main} implies that the localization functor of classical Raynaud theorem is $\infty$-categorical, i.e. the usual category $\An^{\mathrm{qpcqs}}$ of quasi-paracompact and quasi-separated $k$-analytic spaces
is the $\infty$-categorical localization of $\fSch$. This is not a common phenomenon: if $\cC$ is a
$1$-category and $S$ a collection of morphisms in $\cC$ then the $\infty$-categorical localization $\cC[S^{-1}]$ is typically a genuine $\infty$-category.

\end{remark}
	
\begin{definition}
Let $\dAn^{\mathrm{qpcqs}} \subseteq \dAn$ denote the full subcategory of $\dAn$ spanned by quasi-paracompact and quasi-separated derived $k$-analytic spaces $X \in \dAn$.
\end{definition}	

\begin{definition}
Let $f \colon \sfX \to \sfY$ be a morphism between derived $\Ok$-adic schemes. We say that $f$ is \emph{rig-strong} if, for each $i> 0$, the induced morphism
	\[
		\pi_i \left((f^\rig)^{-1}\cO_{\sfY^\rig} \right) \otimes_{\pi_0((f^{\rig})^{-1}\cO_{\sfY^\rig})} \pi_0( \cO_{\sfX^\rig}) \to \pi_i \left( \cO_{\sfX^\rig} \right),
	\]
is an equivalence in the \infcat $\Mod_{\pi_0(\cO_{\sfX^\rig})}$.
\end{definition}

\begin{theorem}[Derived Raynaud Localization Theorem] \label{main}
Let $S$ denote the saturated class of morphisms in $\dfSch^\adm$ generated by rig-strong morphisms $f \colon \sfX \to \sfY$ whose $\trun_{\leq 0} (f)$ is an admissible blow-up of ordinary $\Ok$-adic
schemes.
Then the rigidification functor
	\[
		\rigg \colon \dfSch^\adm \to \dAn^{\mathrm{qpcqs}},
	\]
induces an equivalence of \infcats
	\[
		\dfSch^\adm[S^{-1}] \simeq \dAn^{\mathrm{qpcqs}}.
	\]
\end{theorem}

\cref{main} is an immediate consequence of \cref{lem_for_main} combined with the following Proposition:

\begin{proposition}
The rigidification functor $\rigg \colon \dfSch^\adm \to \dAn^{\mathrm{qpcqs}}$ satisfies the dual assumptions of the statement in \cref{lem_for_main}. 
\end{proposition}

\begin{proof}
The verification of the assumptions of \cref{lem_for_main} are made simultaneously: Let $X \in \dAn^{\mathrm{qpcqs}}$ and define
	\[
		\cC_X \coloneqq \big( \dfSch^\adm \big)_{X / }.
	\]
Denote by 
	\[
		p_0 \colon \cC_X \to \dfSch, \quad \textrm{ } p_1 \colon \cC_X \to \dAn^{\mathrm{qpcqs}},
	\] 
the canonical projections functors. We will
show that for every finite space $K$ and every functor $f \colon K \to \cC_X$, $f$ can be extended to a (cone) functor $f^\lhd \colon K^\lhd \to \cC_X$ in such a way that $f^\lhd(\infty) $ is a formal model for $X \in \dAn^{\mathrm{qpcqs}}$. We denote $\infty \in K^\lhd$ the
cone point. This will imply that $\cC_X$ is a cofiltered \infcat, hence its homotopy type is weakly contractible. It also follows, that the inclusion of the full subcategory spanned by formal models, for $X$, in $\cC_X$, is final. 

Let us first sketch the rough idea of proof: By induction on the Postnikov tower we are allowed to lift commutative diagrams of derived $k$-analytic spaces to the formal level. This is done, by reducing questions concerning liftings of $\Tad$-structures,
to analogue lifting questions at the level of the stable \infcats of coherent modules. This is achieved using the universal property of the adic cotangent complex. Furthermore, the corresponding questions for coherent modules can be dealt effectively using the results proved in Appendix A. The main
technical difficulty is to keep track of higher coherences, involved in commutative diagrams of derived $k$-analytic spaces, when lifting these to the $\Ok$-adic setting.

We will construct a sequence 
	\[
		\{ (\sfX_n, \trun_{\leq n} X \to \sfX^\rig)  \in \cC_{\trun_{\leq n} X} \}_{n \in \mathbb N}, \quad \textrm{with } \sfX_n \coloneqq (\cX_n , \cO_{\sfX_n}) \in \dfSch^\adm,
	\]
satisfiying the following
conditions:
\begin{enumerate}
\item For each $n  \geq 0$, $\sfX_n$ is $n$-truncated.
\item For each $n \geq 0$, we have an equivalence
	\[
		\big( \sfX_n \big)^\rig \simeq \trun_{\leq n} \sfX.
	\]
\item For each $n \geq 0$, there exists a canonical equivalence
	\[
		\sfX_n \to \trun_{\leq n } \sfX_{n+1} ,
	\]
in $\dfSch^\adm$. This implies, in particular, that the underlying $\infty$-topoi $\cX_n \in \Top$ are all equivalent.
\item For each $n \geq 0$, there is a functor $f_n^\lhd \in \Fun \big( K^\lhd, \cC_{\trun_{\le n} X} \big)$ whose restriction $(f_n^\lhd)_{| K}$ is naturally equivalent to $\trun_{\leq n} f $, in the \infcat
$\Fun \left( K , \cC_{\trun_{\leq n} X} \right)$. Additionally, $p_0 \left( f_n^\lhd(\infty) \right) \simeq \sfX_n$.
\end{enumerate}

Assume that we have constructed such a sequence $\{(\sfX_n, \trun_{\leq n} X \to \sfX^\rig)  \in \cC_{\trun_{\leq n} X} \}_{n \in \mathbb N}$ satisfying conditions (i) through (iv). Define $\sfX \coloneqq \colim_{n \geq 0} \sfX_n$ and notice that
the morphisms $\trun_{\leq n } X \to \sfX_n^\rig$ assemble to produce a morphism
	\[
		X \to \sfX^\rig,
	\]
in the \infcat $\dAn$. Moreover, by the universal property of filtered colimits, the diagrams $f^\lhd_n \in \Fun \left( K^\lhd, \cC_{\trun_{\leq n} X} \right)$ assemble. There exists thus a well defined extension $f^\lhd \in \Fun \left( 
K^\lhd, \cC_X \right)$, of $f \colon K \to \cC_X$. As the rigidification functor is compatible with $n$-th truncations, it follows that the functor $f^\lhd$, constructed in this way, when evaluated on $\infty$
	\[
		p_1 \big( f^\lhd(\infty) \big) \in \big( \dAn^{\mathrm{qpcqs}} \big)_{X/ }
	\]
agrees with the morphism
	\[
		X \to \sfX^\rig.
	\]
This finishes the proof of the claim. For this reason, we are reduced to prove the existence of a sequence $\{ (\sfX_n, \trun_{\leq n} X \to \sfX_n^\rig)  \in \cC_{\trun_{\leq n} X} \}_{n \in \mathbb N}$ satisfying conditions
(i) through (iv) above. 

\addtocontents{toc}{\SkipTocEntry}
\subsection*{Step 1}{(Case $n = 0$)}
Let $X_0 \coloneqq \trun_{\leq 0} X.$ By the universal property of $n$-truncation we can assume, without loss of generality, that for each vertex $x \in K$, the component $\big( \sfY_x,
\psi_x \colon X_0 \to \sfY_x^\rig \big) \coloneqq 
f(x) \in \cC_{X_0}$ is actually discrete. More concretely, each $\sfY_X$ is equivalent to an ordinary $\Ok$-adic formal scheme. In this case, the result is now a direct consequence of \cite[Theorem 3, page 204]{bosch2005lectures}. 

\addtocontents{toc}{\SkipTocEntry}
\subsection*{Step 2}{(Inductive datum)}
Assume that, for a given $n \geq 0$,
we have constructed a diagram $f_n^\lhd \in \Fun \left( K^\lhd, \cC_{ \trun_{\leq n} X } \right)$ satisfying conditions (i) through (iv), above. Denote by $\alpha_{n, x} \colon \sfX_n \to \sfY_{n, x} $ the morphism associated to
$\infty \to x$ in $K^\lhd$, where $\sfY_{n, x} \coloneqq \trun_{\leq n} \sfY_x = \big( \cY_{n, x}, \cO_{n, x} \big).$ The functor $f_n^\lhd \in \Fun \left( K^\lhd, \cC_{\trun_{\leq n} X} \right)$ corresponds to the datum:
\begin{enumerate}
\item A diagram $f^\lhd_{n, *}  \colon K^\lhd \to \Top$ such that $f(\infty) \simeq \cX_n$. Furthermore, for each $x \in K$, we have a morphism $\alpha_{n, x, *} \colon \cX_{n} \to \cY_{n, x}$ in $\Top$. We remark that this data is constant for $ 0 \leq m \leq n$.
\item A diagram $f_n^{\lhd, -1} \colon K^{\lhd, \op} \to \fCAlg(\cX_n)_{/ \cO_{\sfX_n} } $ such that $f^{\lhd, -1} (\infty) \simeq \mathrm{id}_{\cO_{\sfX_n}}$ and $f^{\lhd, -1} (x) $ corresponds to a (structural) morphism $h_{n, x} \colon \alpha_{n, x}^{-1} \cO_{\sfY_{n, x}} \to
\cO_{\sfX_n}$, in the \infcat $\fCAlg(\cX_n)_{/  \cO_{\sfX_n}}$.
\end{enumerate}

A similar analysis for the diagram $\trun_{\leq n+1} f \colon K \to \cC_{\trun_{\leq n+1} X}$ together with the Postnikov decomposition imply that we have a functor $f^{-1}_{n+1} \colon K^\op \times \big( \Delta_1 \big)^2 \to \fCAlg(\cX_n)_{/ \cO_{\sfX_n}}$,
such that, for each $ x \in K$, the induced morphism
	\[
		f^{-1}_{n + 1, x} \colon \big( \Delta_1 \big)^2 \to \fCAlg(\cX_n )_{/ \cO_{\sfX_n}},
	\]
corresponds to a pullback diagram of the form
	\begin{equation} \label{form_d}
	\begin{tikzcd}
		\tau_{\leq n + 1} (\alpha_{ x}^{-1} \cO_{\sfY_{ x}} ) \ar{r} \ar{d} & \tau_{\leq n}( \alpha^{-1}_{ x} \cO_{\sfY_{ x}} ) \ar{d}{d_{n, x}} \\
		\tau_{\leq n} ( \alpha_{ x}^{-1} \cO_{\sfY_{ x}}) \ar{r}{d_{n, x}^0} & \tau_{\leq n} ( \alpha_{ x}^{-1} \cO_{\sfY_{ x}}) \oplus \alpha^{-1}_{ x} \pi_{n+1} \big( \cO_{\sfY_x} \big) [ n+2]
	\end{tikzcd}
	\end{equation}
in $\fCAlg(\cX_n)_{/ \cO_{\sfX_n}}$. Here $d_{n, x}$ denotes a suitable $\Ok$-adic derivation and $d^0_{n, x}$ the trivial adic derivation.

\addtocontents{toc}{\SkipTocEntry}
\subsection*{Step 3}{(Functoriality of the construction $\fCAlg(\cX)_{\cO / / \cO}$)} Consider the functor $I \colon \fCAlg(\cX_n)_{/ \cO_{\sfX_n}} \to \Cat$ given on objects by the formula
	\[
		\big( \cO \to \cO_{\sfX_n} \big)  \in \fCAlg(\cX_n)_{/ \cO_{\sfX_n}} \mapsto \fCAlg(\cX_n)_{ \cO / / \cO} \in \Cat.
	\]
The transition morphisms correspond to (suitable) base change functors. Let $\cD \to \fCAlg(\cX_n)_{/ \cO_{\sfX_n}}$ denote the corresponding coCartesian fibration obtained via the unstraightening construction. Notice that pullback along $\cO \to \cO_{\sfX_n}
$ induces a functor $g_\cO \colon \fCAlg(\cX_n )_{\cO_{\sfX_n} // \cO_{\sfX_n}} \to \fCAlg(\cX_n)_{ \cO // \cO}$, which admits a left adjoint $f_\cO \colon  \fCAlg(\cX_n)_{ \cO // \cO} \to  \fCAlg(\cX_n )_{\cO_{\sfX_n} // \cO_{\sfX_n}}$. The latter is given by
base change along $\cO \to \cO_{\sfX_n}$. Therefore, applying the unstraightening construction, we obtain a well defined functor
	\[
		G \colon \fCAlg(\cX_n)_{\cO_{\sfX_n} // \cO_{\sfX_n}} \times \fCAlg(\cX_n)_{/ \cO_{\sfX_n}} \to \cD,
	\]
over $ \fCAlg(\cX_n)_{/ \cO_{\sfX_n}} $. Moreover, its fiber at $(\cO \to \cO_{\sfX_n} ) \in  \fCAlg(\cX_n)_{/ \cO_{\sfX_n}} $ coincides with the functor $g_\cO$, introduced above. Thanks to the (dual) discussion proceding \cite[Corollary 8.6]{porta2016higher}, it follows that
$G$ admits a left adjoint $F \colon  \cD \to  \fCAlg(\cX_n)_{\cO_{\sfX_n} // \cO_{\sfX_n}} \times \fCAlg(\cX_n)_{/ \cO_{\sfX_n}}$.

\addtocontents{toc}{\SkipTocEntry}
\subsection*{Step 4}(Base change of \eqref{form_d} along the morphism $\tau_{\leq n}( \alpha_{ x}^{-1} \cO_{\sfY_{ x}}) \to \cO_{\sfX_n}$)
The zero derivations $d^0_{n, x}$ in \eqref{form_d} assemble to give a a well defined functor $d^0_n \colon K^{\op} \to \cD$. Similarly the $d_{n, x}$ induce a well defined functor $d_{n} \colon K^{\op} \to \cD$. Denote 
	\[
		\Delta_0 \coloneqq F \circ d^0_n, \quad \textrm{and } 
		\Delta \coloneqq F \circ d_n \colon K^{\op} \to \fCAlg(\cX_n)_{ \cO_{\sfX_n}  / / \cO_{\sfX_n}}.
	\] 
Notice that $\Delta_0 \colon K^\op \to \fCAlg(\cX_n)_{ \cO_{\sfX_n} // \cO_{\sfX_n}}$ is given on objects by the formula
	\[
		x \in K^\op \mapsto \left( \cO_{\sfX_n} \to \cO_{\sfX_n} \oplus \alpha_{x}^* \pi_{n+1} \big( \cO_{\sfY_{ x}} \big) [n+2] 
		 \xrightarrow{d^0_{n,x}} \cO_{\sfX_n} \right) \in  \fCAlg(\cX_n)_{ \cO_{\sfX_n} // \cO_{\sfX_n}}.
	\]
Similary the functor $\Delta \colon K^\op \to \fCAlg(\cX_n)_{ \cO_{\sfX_n} // \cO_{\sfX_n}}$ satisfies
	\[
		x \in K^\op \mapsto  \left( \cO_{\sfX_n} \to  \cO_{\sfX_n}  \oplus \alpha_{ x}^* \pi_{n+1} \big( \cO_{\sfY_{ x}} \big) [n+2] 
		 \xrightarrow{d_{n,x}} \cO_{\sfX_n} \right) \in \fCAlg(\cX_n)_{ \cO_{\sfX_n} // \cO_{\sfX_n}}.
	\]
By construction, both functor $\Delta_0$ and $\Delta$ factor through the full subcategory
	\[
		\fCAlg^\der(\cX_n)_{\cO_{\sfX_n} // \cO_{\sfX_n}} \subseteq \fCAlg(\cX_n)_{\cO_{\sfX_n} // \cO_{\sfX_n}}
	\]
spanned by those objects $\cO_{\sfX_n} \to \cA \to \cO_{\sfX_n}$
which correspond to $\Ok$-adic derivations. 

\addtocontents{toc}{\SkipTocEntry}
\subsection*{Step 5}({Reduction of the above diagrams to diagrams of modules}) The universal property of the $\Ok$-adic cotangent complex implies that we have an equivalence of \infcats
	\[
		\functor^\der \colon \fCAlg^\der(\cX_n)_{\cO_{\sfX_n} // \cO_{\sfX_n}} \simeq \big( \Mod_{\cO_{\sfX_n}} \big)_{\bL^\ad_{\cO_{\sfX_n} / }}. 
	\]
Therefore, the functors $\Delta_0$ and $\Delta$ correspond, under the equivalence $\functor^\der$, to functors \[\Delta_0 , \ \Delta \colon K^\op \to  \big( \Mod_{\cO_{\sfX_n}} \big)_{\bL^\ad_{\cO_{\sfX_n} }},\] given on objects by the formulas
	\[
		x \in K^\op \mapsto \big( d_{n, x}^0 \colon \bL^\ad_{\cO_{\sfX_n}  } \to \alpha^*_{x} \pi_{n+1} \big(\cO_{\sfY_{ x}} \big) [n+2] \big) \in \big(  \Mod_{\cO_{\sfX_n}} \big)_{\bL^\ad_{\cO_{\sfX_n} / }},
	\]
and	
	\[
		x \in K^\op \mapsto \big( d_{n, x} \colon \bL^\ad_{\cO_{\sfX_n}  } \to \alpha^*_x \pi_{n+1} \big( \cO_{\sfY_{ x}}\big) [n+2] \big) \in \big(  \Mod_{\cO_{\sfX_n}} \big)_{\bL^\ad_{\cO_{\sfX_n} / }},
	\]
respectively. Thanks to the proofs of both \cite[Lemma 5.35 and Corollary 5.38]{porta2017representability} the $\Ok$-adic cotangent complex $\bL^\ad_{\cO_{\sfX_n}}$ is coherent and connective. Therefore the functors $\Delta_0, \ \Delta \colon K^{\op} \to 
 \big(  \Mod_{\cO_{\sfX_n}} \big)_{\bL^\ad_{\cO_{\sfX_n} / }}$ factor through the full subcategory $\Coh(\cO_{\sfX_n} )_{ \bL^\ad_{\cO_{\sfX_n}} /} \subseteq  \big(  \Mod_{\cO_{\sfX_n}} \big)_{\bL^\ad_{\cO_{\sfX_n} / }}$. 
 
 \addtocontents{toc}{\SkipTocEntry}
 \subsection*{Step 6}{(Rigidification of the corresponding diagrams of modules)} Consider now the composites
	\[
		\Delta_0^\rig \coloneqq \rigg \circ \Delta_0, \quad
		\Delta^\rig \coloneqq \rigg \circ \Delta \colon K^\op \to \Coh (\cO_{\sfX_n}^\rig )_{\bL^\an_{\sfX_n^\rig} / }.
	\]
The same reasoning as above, applied to the rigidification of the diagram $\trun_{\leq n+ 1} f \colon K \to \cC_{\trun_{\leq n+ 1} X}$, produces well-defined extensions
	\[
		\widetilde{\Delta}^\rig_0, \quad \widetilde{\Delta}^\rig \colon K^{\lhd, \op} \to \Coh(\cO^\rig_{\sfX_n} )_{ \bL^\an_{\sfX^\rig_n}},
	\]
of both $\rigg \circ \Delta_0$
and $\rigg \circ \Delta$, respectively. Moreover, these satisfy:
\begin{enumerate}
\item We have equivalences 
	\[
	 	\big( \widetilde{\Delta_0} \big)_{\vert K^\op} \simeq \rigg \circ \Delta_0 , \quad
		 \widetilde{\Delta}_{\vert K^\op} \simeq \rigg \circ \Delta ,
	\]
in the \infcat $\Fun \big( K^\op, \Coh (\cO_{\sfX_n}^\rig )_{\bL^\an_{\sfX_n^\rig} / } \big)$.
\item We have further equivalences
	\[
		\widetilde{\Delta_0}^\rig (\infty) \simeq \left( d_0 \colon \bL^\an_{\cO^\rig_{\sfX_n}} \to \pi_{n+1} \big( \cO_X \big) [n + 2] \right),
	\]
and 
	\[
		\widetilde{\Delta}^\rig (\infty) \simeq \left( d \colon \bL^\an_{\cO^\rig_{\sfX_n}} \to \pi_{n+1} \big( \cO_X \big) [n + 2] \right),
	\]
in the \infcat $\Coh(\trun_{\leq n } (X) )_{\bL^\an_{\trun_{\leq n} X}/}$.
\end{enumerate}
Moreover, the derivations $d_0$ and $d$ considered above are induced by the pullback diagram
	\begin{equation} \label{an_der}
	\begin{tikzcd}
		\tau_{\leq n+1}  \cO_X \ar{d} \ar{r} & \tau_{\leq n} \cO_X \ar{d}{d} \\
		\tau_{\leq n} \cO_X \ar{r}{d_0} & \tau_{\leq n} \cO_X \oplus \pi_{n+1} \big( \cO_X \big) [ n  + 2]
	\end{tikzcd},
	\end{equation}
computed in the \infcat $\AnRing(X)_{/ \cO_X}$. 

\addtocontents{toc}{\SkipTocEntry}
\subsection*{Step 7}{(Lifting of $ \widetilde{\Delta_0}^\rig$ and $\widetilde{\Delta}^\rig$ to diagrams in $\Coh (\sfX_n)$.)}
Thanks to \cref{O'Neill} and its proof, we can lift both diagrams $\widetilde{\Delta}^\rig_0 $ and $\widetilde{\Delta}^\rig$ to (formal model) diagrams $\underline{\Delta}_0, \ \underline{\Delta} \colon K^\op \to 
\Coh( \cO_{\sfX_n} )_{
\bL^\ad_{\cO_{\sfX_n} / }}$, respectively. We have natural equivalences 
	\[
		\underline{\Delta}_{0 \vert K^\op} \simeq \Delta_0 , \quad \underline{\Delta}_{\vert K^\op} \simeq \Delta.
	\]
We further have equivalences
	\begin{gather}
		\label{11} \underline{\Delta}_0  (\infty) \simeq  \left( \delta_0  \colon \bL^\ad_{\cO_{\sfX_n}} \to N[n+2] \right) \\
		\label{12} \underline{\Delta} (\infty) \simeq \left(\delta \colon \bL^\ad_{\cO_{\sfX_n}} \to N[n+2] \right).
	\end{gather}
Where $N \in \Coh(\cO_{\sfX_n})$ denotes a $(t)$-torsion free formal model of $\pi_{n+1} (\cO_{\sfX_n})$, concentrated in degree 0. The choice of such $N \in \Coh(\cO_{\sfX_n})$ can be realized as follows: 

First choose a given formal model $N \in 
\Coh(\cO_{\sfX_n})$ for $\pi_{n+1} (\cO_X)$. As the rigidification functor $\rigg$ is compatible with $n$-truncations, we can replace $N $ with $\tau_{\leq n} (N)$ and thus suppose that $N$ is truncated to begin with. We can kill the $(t)$-torsion on $N$
by multiplying it by a sufficiently large power of $t$. More precisely, we can consider $t^m N$, for $m > 0$ sufficiently large, such that $t^m N$ is $(t)$-torsion free. The conclusion is now a consequence of the fact that the canonical map $t^m N \to N$ induces an equivalence
$\big( t^m N \big)^\rig \simeq N^\rig$.

\addtocontents{toc}{\SkipTocEntry}
\subsection*{Step 8}{(Recovering the extension of the original diagram $ f_{n+1}^{-1}$  by means of the right adjoint $G$ above)}
Notice that the rigidication of both \eqref{11} and \eqref{12} concides with the derivations $d_0$ and $d$ displayed in \eqref{an_der}, respectively. We can also consider the diagrams $\underline{\Delta}_0$ and $\underline{\Delta}$ as morphisms
$\Delta_0 \to \delta_0$ and $\Delta \to \delta $, living in the \infcat $\Fun \big( K^\op, \fCAlg(\cX_n)_{\cO_{\sfX_n } / / \cO_{\sfX_n}} \big)$. Thanks to \cite[Proposition 3.3.3.2]{lurie2009higher} we can lift both diagrams $\underline{\Delta}_0$ and $\underline{\Delta}$
to functors 
	\[
		K^{\lhd, \op} \to \fCAlg(\cX_n)_{ \cO_{\sfX_n} / / \cO_{\sfX_n}} \times \fCAlg(\cX_n)_{/\cO_{\sfX_n} },
	\]
whose projection along the first component agrees with $\underline{\Delta}_0$ and $\underline{\Delta}$, respectively. Furthermore projection along
the second component agrees with the composition $F \circ f^{ \lhd, -1}$. By adjunction, we obtain thus diagrams 
	\[
		D_0, \ D \colon K^{ \lhd, \op} \to \cD,
	\]
inducing $D_0', D' \colon K^{\lhd, \op} \times \Delta^2 \to \fCAlg(\cX_n)_{\cO_{\sfX_n} / }$. Notice that, evaluation
on vertices $x \in K$ gives us assignments
	\begin{align*}
		x \in K^{\op}  & \\ & \mapsto \left( \tau_{\leq n} (\alpha^{-1}_x \cO_{\sfY_x }) \xrightarrow{d_{0, n}} \tau_{\leq n} (\alpha^{-1}_x \cO_{\sfY_x })  \oplus \pi_{n+1} \big( \cO_{\sfY_x} \big)  [n+2] \to \tau_{\leq n} ( \alpha^{-1}_x \cO_{\sfY_x })  \right) \in 
		\fCAlg(\cX_n)_{\cO_{\sfX_n} / } \\
		x \in K^{\op}  & \\ & \mapsto \left( \tau_{\leq n} (\alpha^{-1}_x \cO_{\sfY_x }) \xrightarrow{d_{ n}} \tau_{\leq n}( \alpha^{-1}_x \cO_{\sfY_x } ) \oplus \pi_{n+1} \big( \cO_{\sfY_x} \big)  [n+2] \to \tau_{\leq n} (\alpha^{-1}_x \cO_{\sfY_x })  \right) \in 
		\fCAlg(\cX_n)_{\cO_{\sfX_n} / } ,
	\end{align*}
respectively. Moreover, their value at $\infty$ agrees with
	\[
		\cO_{\sfX_n} \xrightarrow{d_0} \cO_{\sfX_n} \oplus N [n+ 2] \to \cO_{\sfX_n} , \quad \cO_{\sfX_n} \xrightarrow{d} \cO_{\sfX_n} \oplus N [n+ 2] \to \cO_{\sfX_n} ,
	\]
respectively.
\addtocontents{toc}{\SkipTocEntry}
\subsection*{Step 9}{(Obtaining an extension, $f_{n+1}^\lhd$, of the diagram $\trun_{\leq n+1} f$)}
By taking fiber products, along each $\{ x \} \times \Lambda_2^2$, we thus obtain a diagram $ \mathsf f^\lhd_{n+1} \colon K^{\lhd, \op} \to \fCAlg(\cX_{n})_{\cO_{\sfX_n}}$. Evaluation on each $x \in K$ agrees with 
	\[
		\mathsf f^\lhd_{n+1} (x) \simeq \tau_{\leq n+1} (\alpha^{-1}_{x} \cO_{\sfY_x}).
	\]
More precisely,  we have a canonical equivalence $\big( \mathsf f^\lhd_{n+1} \big)^\rig_{\vert K} \simeq \tau_{\leq n+1} (f^{-1})$ in $\AnRing(X)_{/ \tau_{\le n+1} \cO_X}$. Evaluation at $\infty$, $f^\lhd_{n+1}(\infty)  \simeq \cO_{n+1} \in \fCAlg(\cX_n)$ satisfies
	\[
		\cO_{n+1}^\rig \simeq \tau_{\leq n+1} (\cO_X).
	\]
Let $\sfX_{n+1} \coloneqq (\cX_n, \cO_{n+1} )$. We obtain thus a well defined functor
	\[
		f^{\lhd}_{n+1} \colon K^\lhd \to \dfSch^\adm ,
	\]
whose rigidification coincides with 
	\[
		\tau_{\leq n+1} f \colon K \to  \big( \dAn^{\mathrm{qpcqs}} \big)_{ \trun_{\leq n + 1} X/ }.
	\]
Assembling these diagrams together, we obtain a functor $f^\lhd_{n+1} \colon K^\lhd  \to \cC_X$ satisfying requirements (i) through (iv) above. The proof is thus concluded.
\end{proof}

The proof of \cref{main1} also implies:

\begin{corollary}
Let $f \colon X \to Y$ be a morphism between quasi-paracompact and quasi-separated derived
$k$-analytic spaces. Then $f$ admits a formal model, i.e.
there exists a morphism $\mathsf{f} \colon \sfX \to \sfY$ in $\dfSch^\adm$ such that
$\mathsf{f}^\rig \simeq f$ in the \infcat $\dAn$.
\end{corollary}	

\begin{corollary}
Let $S$ be the saturated class generated by those morphisms $f \colon A \to B$ in $\admCAlg$ such that the induced map
	\[
		\big( \Spf f \big)^\rig \colon \Spf (B)^\rig \to \Spf (A)^\rig
	\]
is an equivalence in $\dAfd$. Then the rigidification functor \break $\rigg \colon \big( \admCAlg \big)^\op \to \dAfd$ factors via a canonical functor	
	\[
		\big( \admCAlg \big)^\op[S^{-1}] \to \dAfd.
	\]
Moreover, the latter is an equivalence of \infcats.
\end{corollary}

\begin{proof}
The result is a direct application of the proof of \cref{main} when $X \in \dAfd$.
\end{proof}

\appendix
\section{Verdier quotients and Lemma on $\Coh$}
The results in this section where first proved in \cite{antonio_Hilbert}. We present them here, for the sake of completeness. We also present different proofs of the ones in \cite{antonio_Hilbert}.

\subsection{Verdier Quotients} Let $X$ be a quasi-compact and quasi-separated scheme and $Z$ denote the formal completion of $X$ along the $(t)$-locus. Consider the rigidification functor
	\[
		\rigg \colon \Coh \big( Z \big) \to \Coh \big(Z^\rig \big).
	\]
	
\begin{notation}
Let $\Cat^\mathrm{Ex}$ denote the \infcat of small stable \infcats and exact functors between them.
\end{notation}

\begin{proposition}{\cite[Theorem B.2]{hennion2016formal}} \label{exist_verdier}
Let $\cC$ be a stable \infcat and $\cA \hookrightarrow \cC$ a full stable subcategory. Then the pushout diagram
	\[
	\begin{tikzcd}
		\cA \ar{r} \ar{d} & \cC \ar{d} \\
		0 \ar{r} & \cD
	\end{tikzcd}
	\]
exists in the \infcat $\Cat^\mathrm{Ex}$.
\end{proposition}

\begin{definition}
Let $\cA, \ \cC$ and $\cD$ as in \cref{exist_verdier}. We refer to $\cD$ as the \emph{Verdier quotient} of $\cC$ by $\cA$.
\end{definition}

\begin{proposition} \label{t_ex:ver}
Let $X$ be a quasi-compact quasi-separated derived scheme almost of finite type over $\Ok$. We denote $Z$ its formal $(t)$-completion. Then there exists a cofiber sequence
	\begin{equation} \label{ref}
		\cK(Z) \hookrightarrow \Coh ( Z ) \to \Coh \big( Z^\rig \big),
	\end{equation}
in the \infcat $\Cat^\mathrm{Ex}$. Moreover, the functors in \eqref{ref} are t-exact. In particular, the rigidification functor
	\[
		\rigg \colon \Coh ( Z ) \to \Coh \big( Z^\rig \big)
	\]
exhibits $\Coh \big( Z^\rig \big)$ as a (t-exact) Verdier quotient of $\Coh(Z)$.
\end{proposition}

\begin{proof}
Let $\cK(Z)$ denote the full subcategory of $\Coh(Z)$ spanned by $(t)$-torsion almost perfect modules on $Z$. Recall that $M \in \Coh(Z)$ is of $t$-torsion if $\pi_*(M)$ is of $(t)$-torsion. Consider the (quasi-compact) \'etale site $X_\et$ of $X$. We define
a functor 
	\[
		\fCoh(-) / \cK (-) \colon X_\et \to \Cat^\mathrm{Ex}
	\]
given on objects by the formula
	\[
		(U \to X) \textrm{ quasi-compact and \'etale} \mapsto \Coh(U^\wedge_t)/ \cK(U^\wedge_t) \in \Cat^\Ex,
	\]
where $U^\wedge_t$ denotes the formal completion of $U$ along the $(t)$-locus. Thanks to \cite[Theorem 7.3]{hennion2016formal} this defines a unique, up to contractible indeterminacy, $\Cat^\Ex$-valued sheaf for the \'etale topology.

We will also need the following ingredient: define a functor
	\[
		\fCoh_{\rig} \colon X_\et \to \Cat^\Ex,
	\]
given on objects by the formula
	\[
		(U \to X) \textrm{ quasi-compact and \'etale} \mapsto \Coh \big( (U^\wedge_t)^\rig \big) \in \Cat^\Ex.
	\]
We remark that $\Coh \colon \An \to \Cat^\Ex$ satisfies fpqc descent for $k$-analytic spaces. Indeed, this follows by the main theorem in \cite{conrad2003descent} combined with the usual reasoning by induction on the Postnikov tower, for almost perfect modules. Moreover, the formal completion and rigidification functors are morphisms of sites. As a consequence we conclude that the assignment
$\fCoh_\rig \colon X_\et \to \Cat^\Ex$ is a sheaf for the \'etale topology on $X$. 

The universal property of pushout induces a canonical morphism of sheaves $\Psi \colon \fCoh (-) / \cK (-) \to \fCoh_{\rig}$ in the \infcat $\Shv_{\et} ( X, \Cat^\Ex )$. We affirm that $\Psi$ is an equivalence in $\Shv_\et(X, \Cat^\Ex)$. By descent, it suffices to
prove the statement on affine objects of $X_\et$. In such case, the result follows readily from the observation that for a derived $\Ok$-algebra $A_0$ the \infcat $\Coh \big( A_0 \otimes_{\Ok} k \big) \in \Cat^\Ex$ is obtained from $\Coh(A_0)$ by ''modding out''
the full subcategory spanned by $(t)$-torsion almost perfect modules. Moreover, thanks to \cite[Theorem 3.1]{porta2018derived} we have a canonical equivalence 
	\[
		 \Coh \big( (\Spf A_0)^\rig \big) \simeq \Coh \big( \Gamma \big( (\Spf A_0)^\rig \big) \big) ,
	\]
in the \infcat $\Cat^\Ex$. Where $\Gamma \big( (\Spf A_0)^\rig \big) \in \CAlg_k$ denotes the derived global sections of $\Spf A_0^\rig$. On the other hand $\Gamma \big( (\Spf A_0)^\rig \big) \simeq A_0 \otimes_{\Ok} k $ and the result follows.
\end{proof}

\subsection{Existence of formal models for modules} In this \S, we prove some results concerning the existence of formal models with respect to the functor $\rigg \colon \Coh ( \sfX ) \to \Coh(X)$ which prove to be fundamental in the proof of \cref{main}.

\begin{proposition} \label{O'Neill}
	Let $X \in \dAn$ be a derived $k$-analytic stack which admits a formal model $\sfX \in \dfSch$.
	Let $\cF \in \Coh(X)$ be concentrated in finitely many cohomological degrees.
	Then $\cF$ admits a formal model, i.e. there exists $\cG \in \Coh(\sfX)$ such that $\cG^\rig \simeq \cF$ in $\Coh(X)$. Moreover, the \infcat of formal models of $\cF$ is a
	filtered \infcat.
\end{proposition}
	
\begin{proof}
	Let $\cF \in \Coh(X)$, be as in the stament of the \cref{O'Neill}. Assume moreover that $\cF$ is connective, i.e. its non-zero cohomology lives in non-positive degrees. Notice that, by definition of ind-completion, $\cF \in \ind \left( \Coh(X) \right)$
	is a compact object.
	
	Let 
	$\Phi \colon \ind (\Coh(X)) \to \ind (\Coh(\sfX))$ denote a fully faithful right adjoint to $\rigg$. It follows from the construction of Ind-completion that we have a canonical equivalence
	\begin{equation} \label{ind_cat}
		\Phi \left( \cF \right) \simeq \colim_{\cG \in \Coh  \left( \sfX \right)_{/ \Phi(\cF)} } \cG ,
	\end{equation}
	in $\ind \left( \Coh(X) \right)$. Notice that, by construction, the limit indexing \infcat appearing on the right hand side of \eqref{ind_cat} is filtered. As $\Phi$ is a fully faithful functor, the counit of the 
adjunction $\left( \rigg, \Phi \right)$ is an equivalence. Our argument now follows by an inductive reasoning using the Postnikov tower as we now detail:

Suppose first that $\cF \in \Coh(X)$ has cohomology concentrated in degree $0$, then it is well known that $\cF$ admits a formal model $\widetilde{\cF} \in \mathrm{Coh}^{+, \heartsuit} ( \sfX)
$, which we can moreover choose to be of no $(t)$-torsion. Moreover, we can choose $\widetilde{\cF}$ in such a way that we have a monomorphism
$\widetilde{\cF} \hookrightarrow \cF$, in the heart $\mathrm{Coh}^{+, \heartsuit}(X)$). We are also allowed to chose in such a way that its rigidification becomes an equivalence, in 
the (heart of) $\ind \left( \Coh \left( \sfX \right) \right)$. We are then dealt with the base of our inductive reasoning.

Suppose now that $\cF$ lives in cohomological degrees $[0, n]$.
By the inductive hypothesis $\tau_{\leq n-1} \cF \in \Coh(X)$ admits a formal model $\widetilde{ \tau_{\leq n-1 } \cF } \in \Coh( \sfX)$, which we can assume to live in homological degrees $[0, n+1 ]$. We can also assume its associated homotopy sheaves are $(t)$-torsion free and we have a map $\widetilde{\tau_{\leq 
n-1}
\cF} \to \tau_{\leq n-1} \cF$, in $\ind \left( \Coh(\sfX) \right)$. Moreover, by construction its rigidification becomes an equivalence. We have a fiber sequence 
	\[
	\begin{tikzcd}
		 \cF \ar[r] & \tau_{\leq n-1} \cF \ar[r] & \pi_n \left( \cF \right) [n+1],
	\end{tikzcd}
	\]
in the \infcat $\Coh(X)$. By applying the exact functor $\Phi$ we obtain a fiber sequence in the \infcat $\Coh(\sfX)$.

As $\widetilde{\tau_{\leq n-1} \cF } \in \ind \left( \Coh( \sfX) \right)$ is a compact object, the composite $\widetilde{\tau_{\leq n-1}
\cF} \to \tau_{\leq n-1} \cF \to  \pi_n \left( \cF \right) [n+1]$ factors through $\cG[n+1]$, for a certain almost perfect complex $\cG \in \Coh(\sfX)^{\heartsuit}$. Moreover, such $\cG$ can be chosen in order to satisfy
	\[
		\cG^\rig \simeq \pi_n \left( \cF \right).
	\]
By the base step, we can further assume that it is $(t)
$-torsion free and admits a monomorphism $\cG \to \pi_n  \left( \cF \right)$, in the heart of the \infcat $\ind
\left( \Coh(\sfX) \right)$. 

Using the fact that $\Phi$ is a right adjoint and the the counit is an equivalence, we deduce that the rigidification of the constructed map $\widetilde{ \tau_{\leq n-1} \cF } \to \cG[n+1]$ is equivalent
to $\tau_{\leq n-1} \cF \to \pi_n( \cF) [n+1]$. 

Therefore $
		\widetilde{\cF}:=\mathrm{fib} \left( \widetilde{ \tau_{\leq n-1} \cF } \to \cG [n+1] \right),
	$
is a formal model for $\cF$, which lives in homological degrees $[0,n]$, of no $t$-torsion. Moreover, it admits 
a map $\widetilde{\cF} \to \cF$, in \infcat $\ind \left( \Coh ( \sfX ) \right)$, which become an equivalence after rigidification. The first part of \cref{O'Neill} now follows.

We are now left to prove that the full subcategory $\cC_{\cF}$, of the \infcat $\Coh(\sfX)_{/ \cF}$, spanned by those objects $\left( \widetilde{\cF}, \psi \colon \widetilde{\cF}^\rig \to \cF \right)$ such that $\psi$ is an equivalence, is
also filtered.

By construction, the \infcat $\Coh(\sfX)_{/ \cF}$ is filtered. In order to prove that $\cC_{\cF}$ is filtered, it suffices to show that every $\left( \cG, \phi \colon \cG^\rig \to \cF \right) \in \Coh( \sfX)_{/ \cF}$ admits a morphism to an object in $\cC_{\cF}$.

We first treat the case where $\cF  \in \Coh(X)$ lies in the heart so then we can write $\cF  \simeq \colim_{i \in I} \cG_i$ in $\ind \left( \Coh( \sfX ) \right)^{\heartsuit}$, where $I$ is filtered. Moreover, we can assume that the $\cG_i \in
\Coh( \sfX)^\heartsuit$ are $(t)$-torsion free and for each $i \in I$ they admit monomorphisms $\cG_i \to \cF$, whose rigidification $\cG_i^{\rig} \simeq \cF$ in $\ind \left( \Coh(X) \right)^{\heartsuit}$.
The structural morphism $\psi \colon \cG^\rig \to \cF$ corresponds by adjunction to a morphism $\cG \to \Phi( \cF) \simeq \colim_{i \in I} \Phi( \cG_i)$.
By compactness of $\cG \in \Coh(\sfX)$, it follows that the later factors through one of the $\cG_i$. To summarize, we have obtained a morphism $\cG \to \cG_i$ which induces a morphism in $\Coh(\sfX)_{/ \cF}$ whose source corresponds to
$\left( \cG, \phi \colon \cG^\rig \to \cF \right)$ and the target is an object lying in $\cC_{\cF}$, as desired.

Suppose now that $\cF \in \Coh(\sfX)$ is connective whose non-zero homology lives in degress $[0, n]$. Given $\left( \cG ,  \phi \colon \cG^\rig \to \cF \right) \in \Coh(\sfX)_{/ \cF}$ we know by induction that $\left( \cG,  \cG^\rig \to 
\cF \to \tau_{\leq n-1} \cF \right) \in \Coh(\sfX)_{/ \tau_{\leq n-1} \cF}$ admits a factorization through one object 
	\[
		\left( \widetilde{\tau_{\leq n-1} \cF }, \widetilde{ \tau_{\leq n-1} \cF^\rig} \to \tau_{\leq n-1} \cF \right) \in \Coh(\sfX)_{/ \tau_{\leq n-1} \cF},
	\]
as before. We
have a commutative
diagram
	\[
	\begin{tikzcd}
		\tau_{\leq n} \cG \ar[r]  \ar[d] & \tau_{\leq n-1 } \cG \ar[r] \ar[d] & \pi_n( \cG) [n+1] \ar[d] \\
		\cF \ar[r] & \tau_{\leq n-1} \cF \ar[r] & \pi_n ( \cF ) [n+1],
	\end{tikzcd}
	\]
where the horizontal maps form fiber sequences in the \infcat $\ind \Coh( \sfX)$. Moreover, there exists a sufficiently large formal model $\cH_n \in \Coh(\sfX)^\heartsuit$ for $\pi_n(\cF)$, which is $(t)$-torsion free, together with a monomorphism $\cH_n \to \pi_n( \cF)$ in $
\ind \left( \Coh(\sfX) \right)^\heartsuit$. Additionally, both composites 
	\[
		\tau_{\leq n-1} \cG \to \widetilde{\tau_{\leq n-1} \cF} \to \tau_{\leq n-1} \cF \to \pi_n ( \cF) [n+1],
	\]
and
	\[
		 \tau_{\leq n-1} \cG \to \pi_n( \cG) [n+1] \to \pi_n(\cF) [n+1],
	\]
factor through $\cH_n[n+1]$.
Thus we have a commutative diagram of fiber sequences in the \infcat $\ind \left( \Coh(\sfX) \right)$
	\[
	\begin{tikzcd}
		\tau_{\leq n} \cG \ar[r] \ar[d] & \tau_{\leq n-1} \cG \ar[r] \ar[d] & \pi_n ( \cG) [n+1] \ar[d] \\
		\widetilde{\cF} \ar[r] \ar[d]& \widetilde{ \tau_{\leq n-1}  \cF} \ar[r] \ar[d] & \cH_n[n+1] \ar[d] \\
		\cF \ar[r] & \tau_{\leq n-1} \cF \ar[r] & \pi_n(\cF) [n+1]
	\end{tikzcd}
	\]
which provides a factorization $\left( \cG , \phi \colon \cG^\rig \to \cF \right) \to \left(  \widetilde{\cF} , \psi \colon \widetilde{\cF}^\rig \to \cF \right) $ in the \infcat $\Coh(\sfX)_{/ \cF}$ where $\left(  \widetilde{\cF} , \psi \colon \widetilde{\cF}^\rig \to \cF \right) 
\in \cC_{\cF}$, this concludes the proof.
\end{proof}

\begin{corollary} \label{map_con_comp}
	Let $X \in \dAn$ and $f \colon \cF \to \cG$ be a morphism $\Coh(X)$, where $\cG$ is of bounded cohomology, i.e. $\cG \in \bCoh(X)$. Suppose we are given a formal model $\sfX \in \mathrm{
	dfSch}_{\Ok}^{\mathrm{taft}}$ of $X$. 
	
	Then we can find a morphism $\ff \colon \cF' \to \cG'$ 
	in $\Coh(\sfX)$ such that $\ff^{\rig}$ lies in the same connected component of $f$ in the mapping space $\Map_{\Coh(X)}( \cF, \cG)$.
\end{corollary}

\begin{proof}
	We will actually prove more: Fix $\cF' \in \Coh( \sfX)$ a formal model for $\cF$, whose existence is guaranteed by \cref{O'Neill}. Then we can find a formal model $\cG' \in \Coh(\sfX)$ for $\cG$ such that the morphism
	\[
		f \colon \cF \to \cG, 
	\]
	lifts to a morphism, 
	\[
		\ff \colon \cF' \to \cG',
	\]
	in the \infcat $\Coh(\sfX)$. Assume thus $\cF' \in \Coh(\sfX)$ fixed. Given a generic $\cG' \in \Coh(\sfX)$, denote by $\cHom \left( \cF', \cG' \right) \in \mathrm{QCoh}(\left( \sfX) \right)$ the Hom-sheaf of (coherent) $\cO_{\sfX}$-modules. Notice that if $
	\cG' \in \bCoh  \left( \sfX \right) $ then the Hom-sheaf $\cHom \left( \cF' , \cG' \right)$ is still an object lying in the \infcat $\Coh \left( \sfX \right)$.
	
	By our assumption on $\cG \in \Coh(X)$, we can find a cohomogically bounded formal model $\cG' \in \bCoh \left( \sfX \right)$ for $\cG$, and thus $\cHom \left( \cF', \cG' \right) \in \Coh  \left( \sfX \right)$. Consider the colimit,
	\begin{align} \label{Pedro&Ines}
		\colim_{\cG' \in \cC} \cHom \left( \cF', \cG' \right) & \simeq \cHom \left( \cF', G \left( \cG \right) \right) \\
		& \simeq  \cHom \left( \left( \cF' \right)^\rig, \cG \right)  \simeq G \left( \cHom \left( \cF, \cG \right) \right),
	\end{align}
	where $\cC$ denotes the \infcat of (cohomological bounded) formal models for $\cG$. The first equivalence in \eqref{Pedro&Ines}
	follows from the fact that $\cF' \in \Coh(\sfX)$ is a compact object in $\ind( \Coh( \sfX))$, thus the Hom-sheaf, with source $\cF'$, commutes with
	filtered 
	colimits. The second equivalence follows from adjunction. By applying the global sections functor on both sides of \eqref{Pedro&Ines} we obtain an equivalence of spaces (notice that $\Phi$ being a right adjoint commutes with global sections)
		\[	
			\colim_{\cG' \in \cC} \Map( \cF', \cG') \simeq \Map_{\Coh(X)}( \cF, \cG).
		\]
	We conclude thus that there exists $\cG'  \in \cC$ and $\ff \colon \cF' \to \cG'$ such that $(\ff)^\rig$ and $f$ lie in the same connected component of $\Map_{\Coh(X)}( \cF, \cG)$, as desired.
\end{proof}

\begin{corollary} \label{lift_p_maps}
	Let $X \in \dAn$ and $f \colon \cF \to \cG$ be a morphism in $\Coh(X)$, where $\cG$ is of bounded homology. Suppose we are given a formal model $\sfX$ for $X$ together with formal models $\cF', \ \cG' \in \Coh(\sfX)$ for 
	$
	\cF$ and $\cG$, respectively. Assume further that $\cG' \in\bCoh(\sfX)$. Then given an arbitrary $f \colon \cF \to \cG $ in $\Coh(X)$, we can find $\ff \colon \cF' \to \cG' $ in $\Coh(\sfX)$ lifting $t^n f \colon \cF \to \cG$, for sufficiently
	large
	$n > 0 $.
\end{corollary}

\begin{proof}
	Consider the sequence of equivalences in \eqref{Pedro&Ines}. Then by applying the same argument as in the proof of \cref{O'Neill} we obtain an equivalence,
	\[ 
	(\cHom( \cF' , \cG') )^\rig \simeq \cHom( \cF, \cG),
	\]
	in the \infcat $\Coh(X)$. Therefore, by taking global sections we obtain
	\[ 
		\Map_{\Coh(\sfX)}( \cF', \cG')[t^{-1}] \simeq \Map_{\Coh(X)}( \cF, \cG),
	\]
	where the left hand side term denotes the colimit $\colim_{\text{mult by } t} \Map_{\Coh(\sfX)}( \cF', \cG')$. 
	Therefore, by multiplying $f \in \Map_{\Coh(X)}( \cF, \cG)$ by a sufficiently large power of $t$, say $t^n$, then $t^n f$ lies in a connected component of $\Map_{\Coh(\sfX)}( \cF', \cG')$, as desired. 
\end{proof}

\nocite{*}
\bibliography{Raynaud}
\bibliographystyle{alpha}
\adress
\end{document}